\documentclass[10pt,a4paper,reqno]{amsart} 
\numberwithin{equation}{section}
\usepackage{latexsym}
\usepackage{verbatim,array}
\usepackage{epsfig}
\usepackage{rotating}
\usepackage{amssymb}
\usepackage[english]{babel}
\usepackage[utf8]{inputenc}
\usepackage[T1]{fontenc}
\usepackage{afterpage}
\usepackage{graphics,color}
\usepackage{paralist,mathrsfs}

\usepackage[normalem]{ ulem }
\usepackage{soul}

\usepackage{amssymb,amsfonts,amsmath,stmaryrd}
\usepackage{pifont,  tikz, subfigure}
\usetikzlibrary{plotmarks,shapes,arrows,positioning}
\usepackage{url}
\usepackage[plainpages=false,pdfpagelabels,colorlinks=true,citecolor=blue,hypertexnames=false]{hyperref}

\renewcommand{\ge}{\geqslant}
\renewcommand{\le}{\leqslant}
\renewcommand{\geq}{\geqslant}
\renewcommand{\leq}{\leqslant}
\usepackage{dsfont,enumitem}
\newcommand{\wx}{\widehat{w}}
\newcommand{\wy}{{w}}

\DeclareMathOperator{\Pol}{{Pol}}
\DeclareMathOperator{\Id}{{Id}}

\graphicspath{{Figures/}}

\makeatletter
\def\testb#1{\testb@i#1,,\@nil}%
\def\testb@i#1,#2,#3\@nil{%
  \draw[->, thick] (O) --++(#1);
  \ifx\relax#2\relax\else\testb@i#2,#3\@nil\fi}
\makeatother

\newcommand{\tpic}[1]{
\includegraphics[height=0.8cm]{#1-eps-converted-to}
}

\makeatletter
\def\testbb#1{\testbb@i#1,,\@nil}%
\def\testbb@i#1,#2,#3\@nil{%
  \draw[->] (O) --++(#1);
  \ifx\relax#2\relax\else\testbb@i#2,#3\@nil\fi}
\makeatother

\catcode`\@=11
\def\section{\@startsection{section}{1}%
 \z@{.7\linespacing\@plus\linespacing}{.5\linespacing}%
 {\normalfont\bfseries\scshape\centering}}

\def\subsection{\@startsection{subsection}{2}%
  \z@{.5\linespacing\@plus\linespacing}{.5\linespacing}%
  {\normalfont\bfseries\scshape}}

\def\subsubsection{\@startsection{subsubsection}{3}%
 \z@{.5\linespacing\@plus\linespacing}{-.5em}
 {\normalfont\bfseries}}
\catcode`\@=12

\newtheorem{Theorem}{Theorem}[section]
\newtheorem{Lemma}[Theorem]{Lemma}
\newtheorem{Proposition}[Theorem]{Proposition}
\newtheorem{Definition}[Theorem]{Definition}
\newtheorem{Corollary}[Theorem]{Corollary}

\def\qed{$\hfill{\vrule height 3pt width 5pt depth 2pt}$}
\def\qee{$\hfill{\Box}$}

\newfont{\bbold}{msbm10 scaled \magstep1}
\newfont{\bbolds}{msbm7 scaled \magstep1}

\newcommand{\ns}{\mathbb{N}}

\newcommand{\zs}{\mathbb{Z}}

\newcommand{\qs}{\mathbb{Q}}

\newcommand{\rs}{\mathbb{R}}
\newcommand{\cs}{\mathbb{C}}
\newcommand{\fps}{formal power series}

\newcommand{\bx}{\bar x}

\newcommand{\bu}{\bar u}
\newcommand{\bv}{\bar v}
\newcommand{\by}{\bar y}

\newcommand{\om}{\omega}
\newcommand{\GL}{\mathcal G_{\mathcal L}}
\newcommand{\cD}{\mathscr D}

\newcommand{\GK}{\mathbb{K}}

\newcommand{\cS}{\mathcal S}
\newcommand{\cL}{\mathcal L}
\newcommand{\cM}{\mathcal M}

\DeclareMathOperator{\id}{id}
\newcommand{\beq}{\begin{equation}}
\newcommand{\eeq}{\end{equation}}

\newcommand{\gf}{generating function}
\newcommand{\gfs}{generating functions}

\newcommand{\al}{\alpha}

\newcommand{\cG}{\mathcal{G}}
\newcommand{\cZ}{\mathcal{Z}}

\def\emm#1,{{\em #1}}

\newcommand{\vareps}{\varepsilon}

\DeclareMathOperator{\sign}{sign}

\author[O. Bernardi]{Olivier Bernardi}
\author[M. Bousquet-M\'elou]{Mireille Bousquet-M\'elou}
\author[K. Raschel]{Kilian Raschel}
\thanks{This project has received funding from the European Research Council (ERC) under the European Union's Horizon 2020 research and innovation programme under the Grant Agreement No 759702. OB was partially supported by NSF grants DMS-1400859 and DMS-1800681. KR was partially supported by the project MADACA from the R\'egion Centre-Val de Loire.}

\title[Counting  quadrant walks via Tutte's invariant method]
{Counting  quadrant walks via Tutte's invariant method\\\smallskip{\it\today}}
\address{OB: Brandeis University, Department of Mathematics, 415 South
  Street, Waltham, MA 02453, USA\\
   MBM: CNRS, LaBRI, Universit\'e de Bordeaux, 351 cours de la
   Lib\'eration,  33405 Talence Cedex, France\\
 KR: CNRS, Institut Denis Poisson, Universit\'e de Tours et Universit\'e d'Orl\'eans, Parc de Grandmont, 37200 Tours, France}
\keywords{Lattice walks, enumeration, differentially algebraic series, conformal mappings}

\begin{document}

\begin{abstract}
In the 1970s, William  Tutte developed a clever algebraic
approach, based on certain ``invariants'', to solve a functional
equation that arises in the enumeration of properly colored
 triangulations.  The enumeration of plane lattice walks
confined to the first quadrant is governed by similar equations, and
has led in the past 20 years to a rich collection of attractive results
dealing with the nature (algebraic, D-finite or not) of the associated
generating function, depending on the set of allowed steps, taken in
$\{-1, 0,1\}^2$.  

We first adapt Tutte's approach to prove (or reprove) the algebraicity
of all quadrant models known or conjectured to be
algebraic.
This includes Gessel's famous
model, and the first proof ever found for one  model with weighted
steps. To be applicable, the method requires the existence of two rational functions called   \emph{invariant} and  \emph{decoupling function} respectively. When they exist, algebraicity follows almost automatically.

Then, we move to a complex analytic viewpoint that has already proved very
powerful, leading in particular to integral expressions for the
generating function in the non-D-finite cases, as well as to proofs of
non-D-finiteness. We develop  in this
context a weaker notion of  invariant. Now all 
quadrant models have invariants, and for those that
have in addition a decoupling function, we obtain integral-free
expressions for 
the generating function, and a proof that this series is D-algebraic (that is, satisfies polynomial differential equations).
\end{abstract}

\maketitle

\begin{flushright} \emph{A tribute to William Tutte (1917-2002), on
    the occasion of his centenary}\end{flushright}

\section{Introduction}
\label{sec:intro}
We consider 2-dimensional lattice walks confined to the first quadrant $\ns^2$ of the plane, as illustrated  in Figure~\ref{fig:examples}. The enumeration of such  quadrant walks  has received a lot of attention in the past 20 years, and
given rise to many interesting methods and results.
Given a set of steps $\cS \subset \zs^2$, the main question is to determine the \gf
\begin{equation}\label{eq:qdef}
     Q(x,y;t) \equiv Q(x,y)=\sum_{i,j,n\geq 0} q(i,j;n) x^iy^jt^n,
\end{equation}
where $q(i,j;n)$ is the number of $n$-step quadrant walks from $(0,0)$
to $(i,j)$, taking their steps in $\cS$. This is one instance of a more general
question consisting in counting walks confined to a given cone. This is a natural and
versatile problem, rich of many applications in algebraic
combinatorics~\cite{mbm-ascending,chen-crossings,gessel-zeilberger,grabiner,kratt-skew}, queuing
theory~\cite{adan,cohen-boxma,flatto-hahn}, and of course in enumerative  combinatorics
via encodings of numerous discrete objects (e.g.\ permutations,
maps\dots) by lattice
walks~\cite{Bern07,bouvel-baxter,gessel-weinstein-wilf,kenyon-bipolar,li-schnyder}.

\begin{figure}[htb]
  \centering
  \includegraphics[height=3cm]{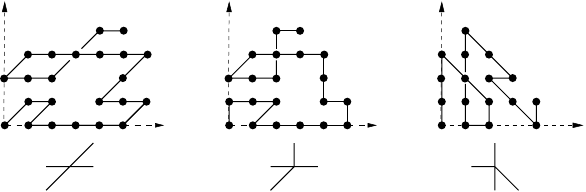}

  \caption{Three models of walks in the quadrant. The \gf\  $Q(x,y;t)$
    is algebraic for the first one, but only D-algebraic for the
    second. For the third one, it is not even D-algebraic.}
  \label{fig:examples}
\end{figure}

\medskip
\noindent{\bf At the crossroads of several mathematical fields.}
Most of the recent progress on this topic deals with  quadrant walks \emm
 with small
steps, (that is, $\cS\subset \{-1,0,1\}^2$). Then  there are~$79$ inherently
different and relevant step sets (also called \emm models,) and a lot
is known on the associated \gfs\ $Q(x,y;t)$. One of the charms of
these results is that  their proofs involve an attractive variety of
 mathematical fields. Let us illustrate this by two results:
\begin{itemize}
\item A certain group $\cG$ of birational transformations associated
  with the model plays a crucial role in the nature of
  $Q(x,y;t)$. Indeed, this series is \emm D-finite, (that
is, satisfies three linear differential equations, one in $x$, one in
$y$, one in $t$, with polynomial coefficients in $x, y$ and~$t$) if
and only if $\cG$  is finite. This happens for 23 of the 79 models. The positive
side of this result (D-finite cases) mostly involves algebra
on formal power
series~\cite{bousquet-versailles,BMM-10,gessel-proba,gessel-zeilberger,mishna-jcta}. The
negative part relies on a detour via complex analysis and a
Riemann-Hilbert-Carleman boundary value problem~\cite{Ra-12,KR-12},
or, alternatively, on a combination of ingredients coming from probability theory and
from the arithmetic properties of G-functions~\cite{BoRaSa-14}. The
complex analytic approach also provides integral expressions for
$Q(x,y;t)$ in terms of Weierstrass' function. 

\item Among the 23 models with a D-finite \gf, exactly 4 are in fact
  \emm algebraic, 
(that is, $Q(x,y;t)$ satisfies a polynomial equation with polynomial coefficients in $x$,
$y$ and $t$). For the most mysterious of them, called
  Gessel's model  (Figure~\ref{fig:examples}, left), a simple conjecture appeared
  around 2000 for the numbers $q(0,0;n)$, but resisted many attempts
  during a decade.  A first proof was then found, based on subtle (and heavy)
  computer algebra~\cite{KaKoZe08}. The algebraicity was only discovered a bit
  later, using even heavier computer algebra~\cite{BoKa-10}. Since then, two
  other proofs have been given: one is based on complex analysis~\cite{BKR-13},
  and the other is, at last, elementary~\cite{mbm-gessel}. 
\end{itemize}

\medskip
\noindent{\bf Classifying solutions of functional equations.}
Beyond the solution of a whole range of combinatorial problems, the
enumeration of quadrant  walks is motivated by an intrinsic interest in the class of functional equations that govern
the series $Q(x,y;t)\equiv Q(x,y)$. These equations involve \emm divided differences, (or
\emm discrete derivatives,) \emm in two variables,. For instance, for  Kreweras' walks (steps
$\nearrow$, $\leftarrow$, $\downarrow$), there holds:
\beq\label{eq:Kreweras}
Q(x,y)=1+ t xy Q(x,y) + t\, \frac{Q(x,y)-Q(0,y)}x +  t \, \frac{Q(x,y)-Q(x,0)}{y}.
\eeq
This equation  is almost self-explanatory, each term corresponding
to one of the three allowed steps. For instance, the term
\[ 
 t\, \frac{Q(x,y)-Q(0,y)}x
\] 
counts walks ending with a West step, which can never be added at the
end of a walk ending on the $y$-axis. The variables $x$ and $y$ are
sometimes called  \emm catalytic,. Such equations (sometimes linear
as above, sometimes polynomial) occur in many enumeration problems,
because divided differences like
\[ 
 \frac{F(x)-F(0)}x \qquad \hbox{or} \qquad \frac{F(x)-F(1)}{x-1} 
\] 
have a natural combinatorial interpretation for any \gf\
$F(x)$. Examples can be found in the enumeration of lattice
paths~\cite{banderier-flajolet,BMM-10,bousquet-petkovsek-recurrences}, maps~\cite{brown-quadrangulations,brown-tutte-non-separable,chapuy-fang,tutte-chromatic-revisited}, permutations~\cite{mbm-ascending,bousquet-motifs,bouvel-baxter}\dots\ A complete
bibliography would include hundreds of references. 

Given a class of functional equations, a natural question is to
decide if (and where) their solutions fit in a classical hierarchy of power
series:
\beq\label{hierarchy}
  \hbox{ rational } \subset \hbox{ algebraic } \subset  \hbox{ D-finite }
        \subset \hbox{ D-algebraic},
\eeq
where we say that a series (say $Q(x,y;t)$ in our case) is \emm D-algebraic,
if it satisfies three polynomial differential equations (one in $x$,
one in $y$, one in $t$). {A} 
historical example is H\"older's proof
that the gamma function  is \emm hypertranscendental, (that is, not D-algebraic), based on the difference
equation $\Gamma(x+1)=x\Gamma(x)$. Later, differential Galois theory
was developed (by Picard, Vessiot, then Kolchin) to study algebraic
relations between D-finite functions~\cite{vdP-singer-Galois-diff}. This theory was then
adapted to $q$-equations, to difference equations~\cite{vdP-singer-difference},  and also extended
to D-algebraic functions~\cite{malgrange-non-linear}. Let us also cite~\cite{dreyfus-Mahler} for recent results
on the 
hypertranscendence of solutions of Mahler's equations. 

Returning to equations with divided differences, it is known that
those involving  only \emm one, catalytic variable $x$
(arising  for instance when counting walks in {a} half-plane) have
algebraic solutions, and this result is effective~\cite{mbm-jehanne}. 
Algebraicity  also follows from a deep theorem in Artin's approximation
theory~\cite{popescu,swan-popescu}. 
For quadrant equations like~\eqref{eq:Kreweras} (with
\emm two, catalytic variables $x$ and $y$), the classification   with respect to the first three steps of the
hierarchy~\eqref{hierarchy} is now completely understood. One outcome of
this paper deals with the final step: D-algebraicity.

\medskip
\noindent{\bf Contents of the paper.}
 We introduce \emm two new objects, related to quadrant
equations, called \emm invariants, and \emm decoupling
functions,. Both are rational functions in $x$, $y$ and $t$. Not all
models admit invariants or decoupling functions. We show that these
objects play a key role in the classification of quadrant walks (see Table~\ref{tab:nat} for a summary):
\begin{itemize}
\item First, we prove that  invariants exist if
  and only if the group of the model is
  finite (that is, if and only if $Q(x,y;t)$ is D-finite); this happens for 23 models.  In this
  case, decoupling functions 
  exist if and only if the so-called \emm orbit sum, vanishes
  (Section~\ref{sec:extensions}). This holds precisely for  the $4$ algebraic models
  (Figure~\ref{fig:alg_models}, left).

\begin{figure}[t!]
\begin{center}
\begin{tikzpicture}[scale=.45] 
    \draw[->] (0,0) -- (0,-1);
    \draw[->] (0,0) -- (1,1);
    \draw[->] (0,0) -- (-1,0);
    \draw[-] (0,-1) -- (0,-1) node[below] {\phantom{$\scriptstyle 1$}};
  \end{tikzpicture}\hspace{2mm}
  \begin{tikzpicture}[scale=.45] 
    \draw[->] (0,0) -- (0,1);
    \draw[->] (0,0) -- (-1,-1);
    \draw[->] (0,0) -- (1,0);
    \draw[-] (0,-1) -- (0,-1) node[below] {\phantom{$\scriptstyle 1$}};
  \end{tikzpicture}\hspace{2mm}
    \begin{tikzpicture}[scale=.45] 
      \draw[->] (0,0) -- (0,-1);
    \draw[->] (0,0) -- (1,1);
    \draw[->] (0,0) -- (-1,0);
    \draw[->] (0,0) -- (0,1);
    \draw[->] (0,0) -- (-1,-1);
    \draw[->] (0,0) -- (1,0);
    \draw[-] (0,-1) -- (0,-1) node[below] {\phantom{$\scriptstyle 1$}};
  \end{tikzpicture}\hspace{2mm}
          \begin{tikzpicture}[scale=.45] 
      \draw[->] (0,0) -- (-1,-1);
    \draw[->] (0,0) -- (1,1);
    \draw[->] (0,0) -- (-1,0);
    \draw[->] (0,0) -- (1,1);
    \draw[->] (0,0) -- (1,0);
    \draw[-] (0,-1) -- (0,-1) node[below] {\phantom{$\scriptstyle 1$}};
  \end{tikzpicture}\hspace{12mm}
\begin{tikzpicture}[scale=.45] 
    \draw[->] (0,0) -- (-1,0) node[left] {$\scriptstyle 1$};
    \draw[->] (0,0) -- (-1,-1) node[left] {$\scriptstyle 1$};
    \draw[->] (0,0) -- (0,-1) node[below] {$\scriptstyle \lambda$};
    \draw[->] (0,0) -- (1,-1) node[right] {$\scriptstyle 1$};
    \draw[->] (0,0) -- (1,0) node[right] {$\scriptstyle 2$};
    \draw[->] (0,0) -- (1,1) node[right] {$\scriptstyle 1$};
  \end{tikzpicture}\hspace{2mm}
\begin{tikzpicture}[scale=.45] 
    \draw[->] (0,0) -- (-1,0) node[left] {$\scriptstyle 1$};
    \draw[->] (0,0) -- (-1,1) node[left] {$\scriptstyle 1$};
    \draw[->] (0,0) -- (0,1) node[above] {$\scriptstyle 2$};
    \draw[->] (0,0) -- (1,1) node[right] {$\scriptstyle 1$};
    \draw[->] (0,0) -- (1,0) node[right] {$\scriptstyle 2$};
    \draw[->] (0,0) -- (1,-1) node[right] {$\scriptstyle 1$};
    \draw[->] (0,0) -- (0,-1) node[below] {$\scriptstyle 1$};
  \end{tikzpicture}\hspace{2mm}
  \begin{tikzpicture}[scale=.45] 
    \draw[->] (0,0) -- (-1,0) node[left] {$\scriptstyle 2$};
    \draw[->] (0,0) -- (-1,1) node[left] {$\scriptstyle 1$};
    \draw[->] (0,0) -- (0,1) node[above] {$\scriptstyle 1$};
    \draw[->] (0,0) -- (-1,-1) node[left] {$\scriptstyle 1$};
    \draw[->] (0,0) -- (1,0) node[right] {$\scriptstyle 1$};
    \draw[->] (0,0) -- (1,-1) node[right] {$\scriptstyle 1$};
    \draw[->] (0,0) -- (0,-1) node[below] {$\scriptstyle 2$};
  \end{tikzpicture}\hspace{2mm}
  \begin{tikzpicture}[scale=.45] 
    \draw[->] (0,0) -- (-1,0) node[left] {$\scriptstyle 2$};
    \draw[->] (0,0) -- (-1,1) node[left] {$\scriptstyle 1$};
    \draw[->] (0,0) -- (0,1) node[above] {$\scriptstyle 2$};
    \draw[->] (0,0) -- (1,1) node[right] {$\scriptstyle 1$};
    \draw[->] (0,0) -- (1,0) node[right] {$\scriptstyle 1$};
    \draw[->] (0,0) -- (0,-1) node[below] {$\scriptstyle 1$};
        \draw[->] (0,0) -- (-1,-1) node[left] {$\scriptstyle 1$};
  \end{tikzpicture} 
  \end{center}
 \vskip -5mm  
\caption{\emph{Left:} The four algebraic quadrant models (Kreweras, reversed Kreweras, double
    Kreweras, Gessel).
    \emph{Right:} Four algebraic models with weights~\cite{KaYa-15}.}
  \label{fig:alg_models}
\end{figure}

\item In those $4$ cases,  we combine invariants and decoupling
  functions to give  \emm short
  and uniform proofs, of algebraicity.
This includes the shortest proof ever found for
  Gessel's famously difficult model, and extends to models with weighted steps~\cite{KaYa-15},
  for which algebraicity was sometimes still
  conjectural (Sections~\ref{sec:gessel} and~\ref{sec:extensions}). 
\item The 56 models with an infinite group have no  invariant. But we
  define  for them  a certain (complex analytic) \emm weak
  invariant,, which is explicit. Then for the $9$ infinite group models that admit decoupling functions, we
  give a new, integral free expression for 
  $Q(x,y;t)$   (Section~\ref{sec:analysis}). This  expression implies
  that $Q(x,y;t)$ is D-algebraic in $x$, $y$ and~$t$. This is
    the first time that D-algebraicity is proved for some non-D-finite quadrant models.
  We compute explicit differential equations
  in $y$ for $Q(0,y;t)$ (Section~\ref{sec:DA}).
  \item  The existence of invariants only depends on the
step set~$\cS$, but the existence of decoupling functions is also
sensitive to the starting point: in Section~\ref{sec:start}, we describe for
which points they actually exist. In particular, we show that some
quadrant models that have no decoupling function  when starting at
$(0,0)$  (and are now known to be non-D-algebraic, as discussed below) still admit decoupling functions
when starting at  other points. Even though we have not worked out the
details, we expect them to be D-algebraic for
these points.
\end{itemize}

 \begin{figure}[h!]
\begin{tabular}{cccccccccc}
\begin{tikzpicture}[scale=.45] 
    \draw[->] (0,0) -- (0,1);
    \draw[->] (0,0) -- (1,0);
    \draw[->] (0,0) -- (-1,0);
    \draw[->] (0,0) -- (-1,-1);
  \end{tikzpicture} &   \begin{tikzpicture}[scale=.45] 
    \draw[->] (0,0) -- (0,1);
    \draw[->] (0,0) -- (1,0);
    \draw[->] (0,0) -- (-1,1);
    \draw[->] (0,0) -- (-1,-1);
  \end{tikzpicture} &     \begin{tikzpicture}[scale=.45] 
    \draw[->] (0,0) -- (0,1);
    \draw[->] (0,0) -- (1,1);
    \draw[->] (0,0) -- (0,-1);
    \draw[->] (0,0) -- (-1,0);
  \end{tikzpicture} & \begin{tikzpicture}[scale=.45] 
    \draw[->] (0,0) -- (0,1);
    \draw[->] (0,0) -- (1,0);
    \draw[->] (0,0) -- (1,-1);
    \draw[->] (0,0) -- (-1,0);
  \end{tikzpicture} &     \begin{tikzpicture}[scale=.45] 
    \draw[->] (0,0) -- (0,1);
    \draw[->] (0,0) -- (1,0);
    \draw[->] (0,0) -- (1,1);
    \draw[->] (0,0) -- (-1,-1);
    \draw[->] (0,0) -- (-1,0);
  \end{tikzpicture} &       \begin{tikzpicture}[scale=.45] 
    \draw[->] (0,0) -- (0,1);
    \draw[->] (0,0) -- (0,-1);
    \draw[->] (0,0) -- (1,1);
    \draw[->] (0,0) -- (-1,-1);
    \draw[->] (0,0) -- (-1,0);
  \end{tikzpicture} &
     \begin{tikzpicture}[scale=.45] 
    \draw[->] (0,0) -- (-1,1);
    \draw[->] (0,0) -- (-1,0);
    \draw[->] (0,0) -- (1,0);
    \draw[->] (0,0) -- (-1,-1);
    \draw[->] (0,0) -- (0,1);
  \end{tikzpicture} & \begin{tikzpicture}[scale=.45] 
    \draw[->] (0,0) -- (1,1);
    \draw[->] (0,0) -- (0,-1);
    \draw[->] (0,0) -- (1,0);
    \draw[->] (0,0) -- (-1,0);
    \draw[->] (0,0) -- (0,1);
  \end{tikzpicture} &    \begin{tikzpicture}[scale=.45] 
    \draw[->] (0,0) -- (1,0);
    \draw[->] (0,0) -- (0,-1);
    \draw[->] (0,0) -- (0,1);
    \draw[->] (0,0) -- (-1,1);
    \draw[->] (0,0) -- (1,-1);
  \end{tikzpicture} \\
\end{tabular}
  \caption{The nine D-algebraic models having an infinite group.}
  \label{tab:infinite}
\end{figure}

An extended abstract of this paper, establishing D-algebraicity for
these 9 non-D-finite models, first appeared in 2016 in the
proceedings of the  FPSAC conference (Formal power series and algebraic
combinatorics~\cite{BeBMRa-FPSAC-16}). Later, 
Dreyfus, Hardouin, Roques and Singer
  completed the differential classification of quadrant walks by
proving that the remaining $47$ infinite group models  are neither
D-algebraic in $x$ (nor in $y$, by symmetry)~\cite{DHRS-17,DHRS-sing}, nor in $t$~\cite{DH-t} (see also the recent preprint~\cite{hardouin-singer2020} on walks with weighted steps).
Their proofs rely  on   Galois theory for
 difference equations. The complete classification of
 quadrant models with small steps  can now be
   summarized as in Table~\ref{tab:nat}, which emphasizes the key role
   of invariants and decoupling functions. 

\begin{table}[htb]
    \begin{center}
      \begin{tabular}{|l|c|c|}
\hline
&  Decoupling function & No decoupling function\\
\hline
        Rational invariant  & \color{red}{4 models} & \color{red}{19 models}
        \\
 ($\Leftrightarrow$ Finite group)  &      Algebraic  \cite{BMM-10,BoKa-10} & D-finite \cite{BMM-10} \\
                                     & Uniform proofs in Section~\ref{sec:extensions} &  
 and transcendental    \cite{MR3588720}\\
\hline
        No rational invariant  & \color{red}{9 models} & \color{red}{47 models} \\
($\Leftrightarrow$ Infinite group)        &{D-algebraic} (Theorem~\ref{thm:DA})
                                                             & Not
                                                               D-algebraic 
        \\
& and not  D-finite \cite{KR-12,BoRaSa-14} & \cite{DHRS-17,DHRS-sing,DH-t}\\
\hline
     \end{tabular}
     \end{center}
 \caption{Algebraic and differential nature of the generating function
     $Q(x,y;t)$.   }
 \label{tab:nat}
\end{table}

\medskip \noindent
{\bf The genesis of invariants.} This paper is inspired by a series of
nine papers published by Tutte between
1973 and 1984, starting with~\cite{lambda12} and ending with~\cite{tutte-differential},  and  later surveyed
in~\cite{tutte-chromatic-revisited},   devoted to the following functional
equation in two catalytic variables:
\beq\label{eq:Tutte}
G(x,y;t)\equiv G(x,y)=xq(q-1)t^2+\frac{xy}{qt}G(1,y)G(x,y)-x^2yt\frac{G(x,y)-G(1,y)}{x-1}+x\frac{G(x,y)-G(x,0)}{y}.
\eeq
This equation appears naturally when counting  planar triangulations
 {properly} colored in $q$ colors. Tutte worked on it for  a decade, and finally established that  $G(1,0)$ is
D-algebraic  in $t$. One key step in his study was to prove
that for certain (infinitely many) values of $q$, the series $G(x,y)$
is algebraic, using a pair  of (non-rational) series that he called \emm
invariants,~\cite{tutte-chromatic-revisited}. They are replaced in
our approach by  (rational) invariants and  decoupling functions. 
After an extension of Tutte's approach to more general map
  problems~\cite{BeBM-11,BeBM-15}, this is now the third time that his
  notion of  invariants proves useful, and we believe  it to have a strong
  potential in the study of  equations with divided differences.

Strictly speaking, adapting Tutte's ideas to quadrant walks
  only provides the algebraicity results of
  Section~\ref{sec:extensions}. In terms of techniques, this  simply involves algebraic
  manipulations on formal power series. The
  D-algebraicity results of Section~\ref{sec:DA} require however an \emm analytic,
  notion of invariants, which is the topic of
  Section~\ref{sec:analysis}. The analytic framework that we use there
  was first developed to study stationary distribution of random
  walks~\cite{FIM-99}, and then adapted to counting problems~\cite{Ra-12}. In fact,
  the \emm weak invariant, that we introduce coincides with the so-called \emm gluing
  function, that was already a key object in these analytic
  approaches. Hence the notion of invariants appears as one way to bridge the gap
between the algebraic and analytic approaches to quadrant walks.

\section{First steps to quadrant walks}
\label{sec:tools}
In this section we introduce some basic tools and notation for the study of quadrant walks with
small steps (see e.g.~\cite{BMM-10} or~\cite{Ra-12}).

Let us begin with some standard notation on power series.
For a ring $R$, we denote by $R[t]$
 (resp.~$R[[t]]$, $R((t))$) 
the ring of
polynomials (resp.\ formal power series, Laurent series) 
in $t$ with coefficients in
$R$. If $R$ is a field, then $R(t)$ stands for the field of rational functions
in $t$. 
This notation is generalized to several variables.
For instance, the series $Q(x,y)$ belongs to $\qs[x,y][[t]]$.
The \emm valuation, of a series in
$R[[t]]\setminus\{0\}$ is the smallest $n$ such that the coefficient of $t^n$ is non-zero.

We will often use bars to denote reciprocals (as long as we remain in
an algebraic, non-analytic context): $\bx:=1/x$, $\by:=1/y$.

\medskip
Consider now the  generating function of quadrant walks $Q(x,y)$
defined by~\eqref{eq:qdef}, and recall that the set of steps  $\mathcal S$ is contained in $\{\-1,0,1\}^2$.
 A simple step-by-step construction of the  walks gives the following functional
equation:
\beq
\label{eq:functional_equation}
     K(x,y) Q(x,y) = K(x,0) Q(x,0) + K(0,y) Q(0,y)-K(0,0) Q(0,0)- xy,
\eeq
where 
\[  
     K(x,y) = xy\Bigg(t\sum_{(i,j)\in \mathcal S} x^iy^j-1\Bigg)
\]  
is  the \emm kernel,  of the model.
It is a  polynomial of degree $2$ in $x$ and $y$,
which  we often write  as
\beq\label{K-abc}
     K(x,y)  =\widetilde a(y)x^2+\widetilde b(y)x+\widetilde c(y)
= a(x)y^2+ b(x)y+ c(x) .
\eeq
We shall also denote  
\[ 
K(x,0)Q(x,0)=R(x) \qquad \hbox{and} \qquad K(0,y) Q(0,y)=S(y) .
\] 
Note that  $K(0,0) Q(0,0)=R(0)=S(0)$, so that the basic functional
equation~\eqref{eq:functional_equation} reads
\beq\label{eqfunc}
K(x,y)Q(x,y)=R(x)+S(y)-R(0)-xy.
\eeq

 Seen as a polynomial in $y$,
the kernel  has two roots $Y_0$ and $Y_1$, which are
Laurent series in $t$ with coefficients in $\qs(x)$. 
If the series $Q(x,Y_i)$ is well defined, setting $y=Y_i$ in~\eqref{eqfunc} shows that
\beq
\label{eq:func_spec}
     R(x)+S(Y_i)= xY_i+R(0).
\eeq
If this holds for $Y_0$ and $Y_1$, then
\beq\label{SYi}
S(Y_0)-xY_0=S(Y_1)-xY_1.
\eeq
This equation will be crucial in our paper.

We define analogously 
the roots $X_0$ and $X_1$ of $K(x,y)=0$ (when
solved for $x$). 

The \emm group of the model,, denoted by  $\mathcal G (\cS)$,
{is the group of birational transformations of {ordered} pairs $(u,v)$ generated by the following two transformations:
\beq\label{eq:generators}
  \Phi(u,v) = \left(\frac{\widetilde c(v)}{\widetilde
      a(v)}\frac{1}{u},v\right) \qquad \hbox{and } \qquad    \Psi(u,v) = \left(u,\frac{c(u)}{a(u)}\frac{1}{v}\right),
\eeq
where the polynomials $a, \widetilde a, c$ and $\widetilde c$ are the coefficients
of $K$ defined by~\eqref{K-abc}. The group operation is 
composition. For instance,
\[
\Psi \circ \Phi(u,v) = \Psi \left(U,v\right) =  
\left(U,\frac{c(U)}{a(U)}\frac{1}{v}\right),
\]
where $U=\frac{\widetilde c(v)}{\widetilde
  a(v)}\frac{1}{u}$.}
One easily checks that both  transformations $\Phi$ and $\Psi$
are involutions {(that is, $\Phi \circ \Phi= \Psi\circ\Psi=\Id$)}. Thus $\mathcal G (\cS)$ is  a dihedral group, which,
depending on the step set $\cS$,  is finite or not.

Let us take for instance 
$\cS=\{\uparrow,\leftarrow,\searrow \}$, {so that
  $K(x,y)=txy^2+ty+tx^2-xy$. Then $\widetilde a(y)=t$, $\widetilde c(y)=ty$,
  $a(x)=tx$, $c(x)=tx^2$, and} the basic   transformations are 
\[ 
\Phi: (u,v)\mapsto (\bu v,v) \qquad \hbox{and}   \qquad \Psi: (u,v)\mapsto(u,u\bv),
\] 
with $\bu:=1/u$ and $\bv:=1/v$. They generate a group of order 6:
\[ 
(u,v)  
 {\overset{\Phi}{\longleftrightarrow}} (\bu v,v)
 {\overset{\Psi}{\longleftrightarrow}} (\bu v,\bu)
 {\overset{\Phi}{\longleftrightarrow}} (\bv,\bu)
 {\overset{\Psi}{\longleftrightarrow}} (\bv,u\bv)
 {\overset{\Phi}{\longleftrightarrow}} (u,u\bv)
 {\overset{\Psi}{\longleftrightarrow}} (u,v).
 \] 
 
Returning to the general case, note that $\Phi$ and $\Psi$
{never} depend on $t$, although $K$ does. Indeed,
\[ 
\frac{c(u)}{a(u)}= \frac{\sum_{(i,-1) \in \cS} u^i}{\sum_{(i,1) \in
    \cS} u^i},
\] 
and {analogously} 
for ${\widetilde c(v)}/{\widetilde a(v)}$.
 One key property of the transformations $\Phi$ and $\Psi$ is that they
leave the \emm step polynomial,, namely
\[ 
P(u,v):=\sum_{(i,j)\in \cS} u^i v^j,
\] 
unchanged. {This is readily checked from the definition of
  $\Phi$ and $\Psi$. By composition, the same holds for all elements
  of $\mathcal G (\cS)$}.

{This group was first introduced in the probabilistic context of random
walks confined to the quadrant~\cite{FIM-99}. 
In our applications, we  will  typically let it act on pairs $(u,v)$
formed of} algebraic functions of the variables $x$, $y$ and $t$. 
In particular, note that
\[ 
\Phi(X_0,y)=(X_1,y) \qquad \hbox{and} \qquad \Psi(x,Y_0)=(x,Y_1).
\] 
More generally, since $K(x,y)=xy (tP(x,y)-1)$, every element $(x',y')$
in the orbit of $(x,Y_0)$ (or $(X_0,y)$) satisfies $K(x',y')=0$.

The above constructions (functional equation, kernel, roots, group\dots)\ can be
extended in a straightforward fashion to the case of weighted
steps. {In this
  context, if the step $(i,j)$ is weighted by $w_{i,j}$, the weight of
  a quadrant walk is the product $W$ of the weights of its
  steps, and this walk  contributes
  $Wx^k y^\ell t^n$ to the \gf\ $Q(x,y;t)$ if it has $n$ steps and ends at
  $(k, \ell)$.}
In particular, the kernel becomes:
\beq\label{K-w}
K(x,y)=xy\Bigg(t \sum_{(i,j)\in \cS} w_{i,j} x^i y^j -1\Bigg).
\eeq

A step set $\cS$ is \emm singular, if
each step $(i,j)\in \cS$ satisfies $i+j \geq 0$.

\section{A new solution of Gessel's model}
\label{sec:gessel}
In this section we illustrate the notions of invariants and decoupling functions, and their use in the solution of quadrant models, by solving
Gessel's model.
This model, with steps $\rightarrow, \nearrow, \leftarrow, \swarrow$, appears as the most difficult  model
with a finite group. Around 2000, Ira Gessel conjectured that
the number of $2n$-step quadrant walks  starting and ending at $(0,0)$ was
\[ 
q(0,0;2n)=16^n\, \frac{ (1/2)_n(5/6)_n}{(2)_n(5/3)_n},
\] 
where $(a)_n=a(a+1) \cdots (a+n-1)$ is the 
rising  factorial.
This 
conjecture was  proved in 2009 by Kauers, Koutschan and
Zeilberger~\cite{KaKoZe08}. A year later, by a computer algebra \emm
tour de force,, Bostan and
Kauers~\cite{BoKa-10} proved that the 
three-variate series $Q(x,y;t)$ 
is not only D-finite, but even algebraic.
Two other, more ``human'', proofs  have then  been
given~\cite{BKR-13,mbm-gessel}.
Here, we give yet another proof based on Tutte's idea of \emm
 invariants,.

The basic functional equation~\eqref{eqfunc} holds with $K(x,y)=t (y +x^2y
 +x^2y^2+1)-xy$, 
$R(x)=tQ(x,0)$, and $S(y)=t(1+y)Q(0,y)$. 

It follows from  $K(x,Y_0)=K(x,Y_1)=0$ that
\beq\label{inv-J-G}
J(Y_0)=J(Y_1),  \qquad \hbox{ with } \qquad J(y)= \frac y{t(1+y)^2}
+t\by(1+y)^2.
\eeq
In Tutte's terminology,  $J(y)$ is a (rational) \emm $y$-invariant.,
Note that \emm checking, that $J(Y_0)=J(Y_1)$ from the identities  $K(x,Y_0)=K(x,Y_1)=0$ is
  {straightforward}. We explain in the next section
  (Theorem~\ref{thm:finite-invariant}) how $J(y)$ can be \emm
  constructed,.

We now introduce a new variable $u$ that replaces $x$, and 
 grants interesting properties to the series $Y_i(x)$ once they are expressed
  in terms of $u$.

\begin{Lemma}\label{lem:conv-gessel}
      Let $X=t+t^2(u+\bu)$, where $u$ is a new variable and $\bu$
  stands for $1/u$. We slightly abuse notation by denoting $Y_0$ and $Y_1$ the roots of $K(X,y)$. Then  $Y_0$ and $Y_1$ are  Laurent series in $t$
with coefficients in $\qs(u)$, satisfying
\[ 
Y_0= \frac u t + \frac{u^2(3+2u^2)}{1-u^2} + O(t), \qquad
Y_1= \frac \bu t + \frac{\bu^2(3+2\bu^2)}{1-\bu^2} + O(t).
\] 
The series $Y_0$ and $Y_1$ simply differ by the transformation $u\mapsto \bu$.
For $i\in\{0,1\}$, the series   $Q(X,Y_i)$ and $Q(0,Y_i)$
are well defined as series in $t$ (with coefficients in $\qs(u)$).
\end{Lemma} 
\begin{proof} The expansions of the $Y_i$ near $t=0$ are found either
  by solving explicitly $K(X,Y_i)=0$, or using Newton's polygon
  method~\cite{abhyankar}. To prove the second point, let us write
\[ 
Q(x,y)= \sum_{\substack{a+b \geq c+d \\ a\geq c} }{\widehat q}(a,b,c,d)
x^{a+b-c-d} y^{a-c} t^{a+b+c+d},
\] 
where ${\widehat q}(a,b,c,d)$ is the number of quadrant walks consisting of $a$
North-East steps, $b$ East steps,~$c$ South-West steps and $d$ West
steps. 
Given that  $X$ and $Y_i$ are series in $t$ with respective
valuation $\alpha=1$ and $\gamma=-1$, the valuation of the summand
associated with the 4-tuple $(a,b,c,d)$ in $Q(X,Y_i)$ is
\[ 
  v(a,b,c,d)= \alpha( a+b-c-d) +\gamma (a-c)+(a+b+c+d)
= a+2 b + c.
\] 
For $Q(X,Y_i)$ to be well defined, we want that for any 
$n \in \ns$, only finitely many 4-tuples  $(a,b,c,d)$ satisfy $ a+b
\geq c+d$,  $a\geq c$ and  $v(a,b,c,d)\leq n$. The above expression for 
$v$ shows that $a$, $b$ and $c$
must be bounded (for instance by $n$), and the inequality $a+b \geq c+d$
bounds $d$ as well. 
Hence $Q(X,Y_i)$ is well defined.

 This implies that $Q(0,Y_i)$ is also well defined, as
$Q(0,y)$ is just obtained by selecting  the 4-tuples such that $a+b=c+d$. 
in the expression  of 
$Q(x,y)$.
\end{proof}

Applying now the generalities of Section~\ref{sec:tools}, we conclude
from Lemma~\ref{lem:conv-gessel} that~\eqref{SYi}
holds:
\beq\label{diff-gessel}
S(Y_0)-{ X} Y_0=S(Y_1)-{X}Y_1.
\eeq
Moreover, the kernel equation $K(X,Y_i)=0$ implies that
\[ 
XY_0-XY_1= \frac 1{t(1+Y_1)}- \frac 1{t(1+Y_0)}.
\] 
Note that this is not specific to the choice of $X$ of the form $t+t^2(u+\bu)$: when $Y_0$ and $Y_1$ are the roots of $K(x,y)$, we still have
  \beq\label{decouple-gessel}
xY_0-xY_1= \frac 1{t(1+Y_1)}- \frac 1{t(1+Y_0)}. 
\eeq
We will later say that  $G(y):=- \frac 1{t(1+y)}$ is a \emph{decoupling function} for Gessel's model (see Section~\ref{subsec:decoupling_functions} for a precise definition). 
We can {then} rewrite~\eqref{diff-gessel} as
\beq\label{L-def}
L(Y_0)=L(Y_1), \qquad
\hbox{with } \qquad L(y)=S(y)+\frac{1}{t(1+y)}.
\eeq
This should be compared to~\eqref{inv-J-G}. In Tutte's terminology, the series $L(y)$ is, as $J(y)$, an \emm
invariant,, but this time it is (most likely) {non}-rational. 
The connection between $J(y)$ and $L(y)$ will stem from the following
lemma, which states, roughly speaking, that  invariants with \emm polynomial
coefficients in, $y$ are trivial.
\begin{Lemma}\label{lem:inv-gessel}
  Let $A(y)$ be a Laurent series in $t$ with coefficients in $\qs[y]$, of the form 
\[ 
A(y)=\sum_{0\leq j\leq n/2+n_0} a(j,n) y^jt^{n} 
\] 
for some $n_0\geq 0$.  {Let $X=t+t^2(u+1/u)$, and define $Y_0$ and $Y_1$ as in Lemma~\ref{lem:conv-gessel}.} Then the series $A(Y_0)$ and $A(Y_1)$ are 
well defined Laurent series in $t$, with coefficients in $\qs(u)$. 
If they coincide, then $A(y)$ is in fact independent of $y$. 
\end{Lemma}
\begin{proof}
  By considering $A(y)-A(0)$, we can  assume that $A(0)=0$. In this case,
\[ 
A(y)=\sum_{1\leq j\leq n/2+n_0} a(j,n) y^jt^{n} .
\] 
Assume that $A(y)$ is not uniformly zero, and let
 \[ 
m=\min_{n,j}\{n-j: a(j,n) \not = 0\}
\] 
(see Figure~\ref{fig:subst} for an illustration).
The inequalities $1\leq j\leq n/2+n_0$ imply that $m$ is finite,
at least equal to $1-2n_0$.  Moreover, only finitely many pairs
$(j,n)$ satisfy $m=n-j$ and $j\leq n/2+n_0$.

\begin{figure}
  \scalebox{0.8}{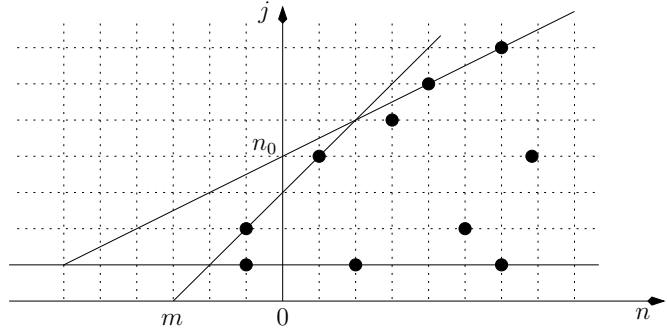}
  \caption{The support of the series $A(y)$, shown with dots, {and} the
    definition of $m$.}
  \label{fig:subst}
\end{figure}

Since $j\leq n/2 +n_0$, 
any  series $Y=c/t+O(1)$ can be substituted for $y$ in $A(y)$, and
\[ 
A(Y)= t^m \left(\sum_{n-j=m} a(j,n) c^j\right) + O(t^{m+1}).
\] 
Applying this to  $Y_0= u /t +O(1)$ and  $Y_1=\bu /t+O(1)$, and
writing that $A(Y_0)=A(Y_1)$, we obtain
\[ 
P(u):=\sum_{j\ge 1} a(j,m+j) u^j=\sum_{j \ge 1} a(j,m+j) \bu^j=P(\bu).
\] 
Hence the polynomial $P(u)$ must vanish, which is incompatible with the definition of $m$.
\end{proof}
The  series $J$ and $L$ defined by~\eqref{inv-J-G} and~\eqref{L-def} do not satisfy the assumptions of the
lemma, as  their coefficients are \emm rational, in $y$ with
poles at $y=0, -1$ (for $J$) and $y=-1$ (for $L$):
\beq\label{def-JL-G}
 J(y)= \frac y{t(1+y)^2}
+t\by(1+y), \qquad L(y)=S(y)+\frac{1}{t(1+y)},
\eeq
with $S(y)=t(1+y)Q(0,y)$.
Still, we can
construct from them a series~$A(y)$ 
satisfying the assumptions of the lemma. First, we eliminate the
simple pole of $J$ at $0$ by
considering  $(L(y)-L(0))J(y)$, which still takes the same value at
$Y_0$ and $Y_1$.  The coefficients
of this series have a pole of order at most $3$ at $y=-1$. By
subtracting an appropriate series of the form $C_1L(y)^3+C_2L(y)^2+C_3L(y)$, where $C_1, C_2$ and $C_3$ depend on $t$ but not on
$y$, we obtain a Laurent series in $t$ satisfying the assumptions
of the lemma: the polynomiality of the coefficients in $y$ holds by
construction, and the fact that in each monomial $y^j t^n$, the
exponent of $j$ is (roughly) at most half the exponent of $n$ comes
from the fact that this holds in $S(y)$, due to the choice of the step
set (a walk ending at $(0,j)$ has at least $2j$ steps). Thus this
series must be constant, equal for instance to its value at $y=-1$. In
brief, 
\beq\label{eq:JL}
(L(y)-L(0))J(y)= C_3 L(y)^3+C_2L(y)^2+C_1L(y)+C_0
\eeq
for some series $C_0,C_1,C_2,C_3$ in $\mathbb{Q}((t))$.
Expanding this identity near $y=-1$ determines the series $C_0$, $C_1$,
$C_2$, $C_3$ in terms of $S$. Their expressions are given in the following proposition. 
\begin{Proposition}\label{prop:abcd}
  For $J$ and $L$ defined by~\eqref{def-JL-G}, and
  $S(y)=t(1+y)Q(0,y)$, Equation~\eqref{eq:JL} holds with 
\[ 
 C_3=-t, \qquad C_2=2+tS(0), \qquad C_1=-S(0)+ 2S'(-1)-1/t, 
\] 
and
\[
  C_0=-2S(0)S'(-1)-3S'(-1)/t+S''(-1)/t.
\]
\end{Proposition}
Replacing in~\eqref{eq:JL} the series $J$, $L$ and $C_0, \ldots, C_3$
by their expressions in terms of $t$, $y$ and $S$ gives  for $S(y)$
a cubic equation, involving $t$, $y$, and three auxiliary unknown series in $t$, namely
{$A_1:=S(0)$,
$A_2:=S_y(-1)$,
and $A_3:=S_{yy}(-1)$: 
\begin{multline}
  {t}^{2}y \left( y+1 \right) ^{2} S (y) ^{3}
-ty  \left( y+1 \right) \left( {t} \left( y+1 \right) A_1 + \left(
    2\,y-1 \right)  \right)   S(y) ^{2}
\label{eq-a-voir}
\\  + \left( ty \left( y-1 \right)  \left( y+1 \right) A_1-2\,ty \left( y+1 \right) ^{2}A_2+{t}^{2}{y}^{4}+4\,{t}^{2}{y}
^{3}+6\,{t}^{2}{y}^{2}+4\,{t}^{2}y+{y}^{3}+{t}^{2}-{y}^{2} \right) S(y)
\\+ t( y+1 ) ^{2} ( 2y A_2-{t} ( y+1 ) ^{2} ) A_1
+y ( y+1 ) 
 ( 3\,y+1 ) A_2-y ( y+1 ) ^{2}A_3-ty
 ( y+1 ) ^{3}
 =0.
\end{multline}
(The letter $A$ stands
  for ``auxiliary'', and these series $A_i$, which depend on $t$
  only, have no direct connection with the series $A(y)$ of
  Lemma~\ref{lem:inv-gessel}.)
It is not hard to see that this
equation defines a unique 4-tuple of power series, with $A_i\in {t}\qs[[t]]$ and $S(y)$ in $t\qs[y][[t]]$. 

Equations of the form
\[  
\Pol({S(y)}, A_1, \ldots , A_k,t,y)=0
\]  
occur in the enumeration of many combinatorial objects (lattice paths,
maps, permutations\dots). 
The variable $y$ is often said to be a \emm catalytic variable,.
Under certain hypotheses 
(which generally hold for  combinatorially founded equations, and
essentially say that these equations have a unique solution
$({S(y) }, A_1, \ldots, A_k)$ in the world of power series),
the solutions of such equations are always algebraic, and a
procedure for finding them is given in~\cite{mbm-jehanne}.

Applying the procedure of~\cite{mbm-jehanne}
to the equation  obtained just above for Gessel's walks
(Proposition~\ref{prop:abcd})
shows in particular that $A_1, A_2$ and $A_3$ belong to  ${t}\qs(Z)$, where
$Z$ is the  unique series in $t$ with constant term 1
satisfying $Z^2=1+256t^2Z ^6/(Z^2+3)^3$. Details on the solution are
given in Appendix~\ref{app:gessel}. Let us mention that, in the other
``elementary'' solution of this model, one has to solve an analogous
 equation satisfied by $R(x)$~\cite[Sec.~3.4]{mbm-gessel}.
Once {$S(y)$ (or equivalently $Q(0,y)$)} is proved to be algebraic, the
  algebraicity of $Q(x,0)$, and finally of $Q(x,y)$, follow using~\eqref{eq:func_spec} and~\eqref{eqfunc}.

\section{Extensions and obstructions: uniform algebraicity proofs}
\label{sec:extensions}
We now formalize and generalize
the three main ingredients in the above solution of Gessel's model: the rational invariant $J(y)$ given
by~\eqref{inv-J-G},  the
identity~\eqref{decouple-gessel} expressing $xY_0-xY_1$ as a
difference $G(Y_0)-G(Y_1)$, and finally the ``invariant lemma''
(Lemma~\ref{lem:inv-gessel}). We discuss the existence of rational
invariants $J$, and of \emm decoupling functions, $G$, for all
quadrant models with small steps in Sections~\ref{sec:invariants} and~\ref{subsec:decoupling_functions} respectively. In particular, we relate the existence
of rational invariants to the finiteness of the group $\cG(\cS)$.
Then, in Sections~\ref{sec:invariant-lemma} to~\ref{sec:effective},
we show that the above solution of Gessel's model extends,
in a uniform fashion, to all quadrant models (possibly weighted) known
or conjectured to have an algebraic \gf \ (see
Figure~\ref{fig:alg_models}). {These~8 models} are precisely those that have a
rational invariant \emm and, a decoupling function. 
For one
of them, we need an algebraic variant of the invariant lemma, which is
described in Section~\ref{sec:alternative}. 


\subsection{Invariants}
\label{sec:invariants}
To begin with, let us observe that for all models $\cS$ that we consider, the associated kernel $K(x,y)$ is irreducible in $\qs(t)[x,y]$. This could be (tediously) checked case by case, but has been proved more generally in~\cite[Lem.~2.3.2]{FIM-99}.
\begin{Definition}
  Given a quadrant model $\cS$, and the associated kernel $K(x,y)$, we define an equivalence relation on elements of $\qs(x,y,t)$ as follows:
  \[
    A(x,y) \equiv B(x,y) \Leftrightarrow A(x,y)-B(x,y)= K(x,y) \frac{N(x,y)}{D(x,y)}
  \]
  for $N(x,y)$ and $D(x,y)$ in $\qs(t)[x,y]$ such that $D(x,y)$ is not divisible by $K(x,y)$ in $\qs(t)[x,y]$.
\end{Definition}
We have the following simple property.
\begin{Lemma}\label{lem:equiv-equiv}
  Let $A(x,y)$ and $B(x,y)$ be two elements of $\qs(x,y,t)$ that do not have a factor $K(x,y)$ in their denominator (once written in an irreducible form). Then the following conditions are equivalent:
  \begin{itemize}
  \item $A(x,y)$ and $B(x,y)$ are equivalent,
  \item $A(x,Y_0)=B(x,Y_0)$,
        \item $A(x,Y_1)=B(x,Y_1)$.
  \end{itemize}
\end{Lemma}
\begin{proof}
  The first point implies the second (or the third) by setting $y=Y_i$ in the definition of equivalence. The second (or third) point implies the first by writing $A(x,y)-B(x,y)$ in irreducible form, and using the fact that $K(x,y)$ is irreducible.
 \end{proof}
\begin{Definition}\label{def:rat-inv}
  A quadrant model  admits \emm invariants, if there exist
  rational functions $I(x) \in \qs(x,t)$ and $J(y)\in \qs(y,t)$, not
  both in $\qs(t)$, such  that $I(x)\equiv J(y)$.
 The functions $I(x)$ and $J(y)$ are said to be an $x$-invariant and
a $y$-invariant for the model, respectively.
\end{Definition}
Our definition is more restrictive than that of
Tutte~\cite{tutte-chromatic-revisited}, who was simply requiring $I(x)$ and $J(y)$ to be series in
$t$ with rational coefficients in $x$ (or $y$).

\smallskip  The existence of (rational) invariants is equivalent to the following
(apparently weaker) condition, which is the one we met in Section~\ref{sec:gessel} (see~\eqref{inv-J-G}).
\begin{Lemma}\label{lem:xy-inv}
 Assume that there exists a rational function $J(y) \in \qs(t,y)
\setminus \qs(t)$ such that $J(Y_0) = J(Y_1)$
when $Y_0$ and $Y_1$ are the roots of the kernel $K(x,y)$, solved for
$y$. 
Then $I(x):=J(Y_0)=J(Y_1)$ is a rational function of $x$, and
$(I(x),J(y))$ forms a pair of invariants. 
\end{Lemma}
\begin{proof}
We have  $I(x)= (J(Y_0)+J(Y_1))/2$, hence $I(x)$ is a
rational function of $x$ and $t$ as any symmetric function of
the roots $Y_0$ and $Y_1$. 
The property $I(x)=J(Y_i)$ then allows us to conclude that $I(x)\equiv J(y)$, using Lemma~\ref{lem:equiv-equiv}.
\end{proof}
\noindent{\bf Example.} In Gessel's case, $J(y)$ was given by~\eqref{inv-J-G}, and we find 
\[ 
I(x)=\frac 1 2 \left( J(Y_0)+J(Y_1)\right) = -\frac t{x^2}+\frac 1
x+2t+x-tx^2.
\] 
We can also check that $K(x,y)$ divides $I(x)-J(y)$. Indeed,
\[ 
I(x)-J(y) = -\frac{K(x,y)K(\bx,y)}{ty(1+y)^2}.
\] 
The factor $K(\bx,y)$ shows that the pair $(I(x),J(y))$ also forms a
pair of invariants for the model $\{\rightarrow, \nwarrow, \leftarrow,
\searrow\}$ obtained by reflection in a vertical line.
\qee

\medskip

We now generalize this observation, by showing that
two models differing by a symmetry of the
square have (or have not) invariants simultaneously. Since these
 symmetries  are generated by 
the  reflections in the main diagonal and in the vertical axis, it
 suffices to consider these two cases.
\begin{Lemma}
\label{lem:invariant_invariants}
  Take  a model $\cS$ with kernel $K(x,y)$ and its diagonal reflection
  $\widetilde \cS$, with kernel
$
\widetilde K(x,y)= K(y,x).
$
Then $\widetilde \cS$ admits invariants if and only if $\cS$ does, and in this case
a possible choice is
$
\widetilde I(x)= J(x)$  and $\widetilde J(y)= I(y).
$
A similar statement holds for the vertical reflection $\overline{\cS}$, with
  kernel
$
\overline{K}(x,y)= x^2 K(\bx,y),
$
where $\bx:=1/x$. A possible choice is then
$
\overline{I}(x)= I(\bx)$ and  $ \overline{J}(y)= J(y).
$
\end{Lemma}   
\begin{proof}
 The proof is elementary.  
\end{proof}

\medskip
We can now tell exactly which models admit invariants. Note that it is
easy to \emm decide, whether a given pair $(I,J)$ is a pair of invariants: it
suffices to check whether $I(x)-J(y)$ has a factor $K(x,y)$. 
The following result tells us how to \emm construct, such pairs.

\begin{Theorem}\label{thm:finite-invariant}
A (possibly weighted) quadrant model $\cS$ has rational invariants if and only
 if the associated group $\mathcal G (\cS)$ defined
 by~\eqref{eq:generators} is finite.  

Assume this is the case, and let $H(x,y)$ be a
rational function in $\qs(x,y,t)$. 
Consider the rational function
\[ 
H_\sigma(x,y):= \sum_{\gamma \in \mathcal G (\cS)} H(\gamma(x,y)).
\] 
Then 
\[ 
I(x)=  H_\sigma(x,Y_0) \qquad \hbox{and}  \qquad J(y)=H_\sigma(X_0,y)
\] 
are respectively rational functions in $(t,x)$ and $(t,y)$, and they form a pair
of invariants, as long as they do not both belong to $\qs(t)$.
\end{Theorem}
\begin{proof}
Assume that the model has invariants $I(x),J(y)$, with
$I(x)\not \in \qs(t)$.  
If $(x',y')$ is any element in the orbit of
$(x,Y_0)$, then $K(x',y')=0$, hence $I(x')=J(y')$. But the form of
$\Phi$ and~$\Psi$ implies by transitivity 
that  $I(x')=I(x)$ and
$J(y')=J(y)$.  

Assume that the group of the model is infinite.  Then the orbit
of $(x,Y_0)$ is infinite as well. Indeed, if it were finite, then
there would exist $\gamma\in  \mathcal G (\cS)$, different from the identity, such that
$\gamma(x,Y_0)=(x,Y_0)$. Denoting $\gamma(x,y)=(r(x,y), s(x,y))$,
where both coordinates $r$ and $s$ are in $\qs(x,y)$, this would mean
in particular that $r(x,y)-x$ vanishes at $y=Y_0$, forcing this rational function to zero, or to have a factor
$K(x,y)$ in its numerator. But this is impossible since $r(x,y)$ does not involve the
variable $t$ (while $K$ does), hence $r(x,y)=x$. By the same argument, $s(x,y)=y$,
hence $\gamma$ is the identity, which contradicts our assumption.
Hence the orbit of $(x,Y_0)$ is infinite. This implies that infinitely
many series~$x'$ occur in it (as the first coordinate of a pair), and thus the equation
(in $x'$)
$I(x')=I(x)$ has infinitely many solutions. This is clearly impossible
since we have assumed that $I(x)\not \in \qs(t)$.

Conversely, take a model with finite group, a rational function $H(x,y)$  in
$\qs(x,y,t)$, and define $H_\sigma$ as above. For instance, for a model
$\cS$ with a vertical symmetry, $\mathcal G (\cS)$ has 
order 4, and the orbit of $(x,Y_0)$ reads:
\[ 
(x,Y_0)  {\overset{\Phi}{\longleftrightarrow}} (\bx, Y_0)
{\overset{\Psi}{\longleftrightarrow}} (\bx, Y_1)
{\overset{\Phi}{\longleftrightarrow}} (x, Y_1).
\] 
 Thus if we take $H(x,y)=x$, then $H_\sigma(x,y)=2(x+\bx)$ and $J(y)=2(X_0+X_1)=-2\frac{\widetilde b(y)}{\widetilde a(y)}$.

Returning to a general group, observe that $H_\sigma$ takes the same value,
by construction, on all elements of the orbit of $(x,y)$. In
particular, $H_\sigma(x,Y_0)=H_\sigma(x,Y_1)$. Hence the
above defined function   $I(x)$ is rational in $x$ and $t$. Analogously,  $J(y)$ is rational in $y$ and $t$.
Moreover,  $J(Y_0)$, being the sum of $H$ over the orbit
of $(x,Y_0)$, coincides with $I(x)$, and by Lemma~\ref{lem:equiv-equiv}, $(I,J)$ is a pair of invariants (unless $I$ and $J$ both  depend on $t$ only).

For instance, for the reverse Kreweras model $ \left\{ \rightarrow, \uparrow,
    \swarrow\right\}$, and $H(x,y)=x$, we find $I(x)=J(y)=1/t$. But
    taking instead  $H(x,y)=1/x$ gives true invariants:
\[ 
I(x)= \bx + x/t-x^2 , \qquad J(y)= \by +y/t-y^2.
\]

Let us finally prove that, for any (possibly weighted) model, there exists
  $k\ge 1$ such that the function   $I^{(k)}(x)$ obtained from the function
  $H^{(k)}(x,y)=x^k$ actually depends on $x$. Assume this is not the
  case. Let $\cG(\cS)$ have order $2n$, and let $x_0=x, x_1, \ldots,
  x_{n-1}$ be the $n$ distinct series~$x'$ that occur in the orbit of
  $(x,Y_0)$ as the first coordinate of some pair. Then by assumption, $I^{(k)}(x)= 2\sum _{i=0}^{n-1} x_i^k$ is
  an element of $\qs(t)$ for all $k$, which shows that all symmetric functions of
  the $x_i$'s depend on $t$ only. This implies that each $x_i$ is an
  algebraic function of $t$ only, which is impossible since $x_0=x$.
 \end{proof}

Since one of the main objectives of this section is to obtain a uniform solution for algebraic quadrant models,
we only give explicit invariants for the four algebraic
  (unweighted) models (see Table~\ref{tab:ratinv}). The remaining $19$ models with a
finite group either have a vertical symmetry (in which case they admit
$I(x)=x+\bx$ as $x$-invariant), or differ from an algebraic
model by a symmetry of the square (in which case
Lemma~\ref{lem:invariant_invariants} applies). Invariants for the four weighted models of Figure~\ref{fig:alg_models} are given in
Table~\ref{tab:decoupling_functions-finite}.

\newcommand\Tstrut{\rule{0pt}{7.0ex}}         
\newcommand\Bstrut{\rule[-1.9ex]{0pt}{2.0ex}}

\begin{table}[htb]

\small
 \begin{tabular}{|l|c|c|c|c|c|}
\hline
& \begin{tikzpicture}[scale=.4] 
    \draw[->] (0,0) -- (1,1);
    \draw[->] (0,0) -- (-1,0);
    \draw[->] (0,0) -- (0,-1);
  \end{tikzpicture}
&\begin{tikzpicture}[scale=.4] 
    \draw[->] (0,0) -- (-1,-1);
    \draw[->] (0,0) -- (1,0);
    \draw[->] (0,0) -- (0,1);
  \end{tikzpicture}
 &\begin{tikzpicture}[scale=.4] %
    \draw[->] (0,0) -- (-1,-1);
    \draw[->] (0,0) -- (1,1);
    \draw[->] (0,0) -- (-1,0);
    \draw[->] (0,0) -- (1,0);
    \draw[->] (0,0) -- (0,-1);
    \draw[->] (0,0) -- (0,1);
  \end{tikzpicture}
 &\begin{tikzpicture}[scale=.4] %
    \draw[->] (0,0) -- (1,1);
    \draw[->] (0,0) -- (-1,-1);
    \draw[->] (0,0) -- (-1,0);
    \draw[->] (0,0) -- (1,0);
  \end{tikzpicture}
\Tstrut
\\
\hline
 $I$& 
 $\frac{t}{x^2}-\frac{1}{x}-t x$ &  $tx^2-x-\frac{t}{x}$&
$\frac{t}{x}-t x-  \frac{1+2t}{1+x}$
&$x+\frac{1}{x}-tx^2-\frac{t}{x^2}+2t$ {\rule{0pt}{3.0ex}}   
\\ & & & &
\\$J$& 
 $\frac{t}{y^2}-\frac{1}{y}-ty$&$ty^2-y-\frac{t}{y}$&
$\frac{t}{y}-ty-\frac{1+2t}{1+y}$  &
$\frac y {t(1+y)^2} + t\frac{(1+y)^2}{y} $
 \Bstrut
\\
\hline
  \end{tabular}
\medskip
 \caption{Rational invariants for algebraic unweighted
   models.}
 \label{tab:ratinv}
\end{table}
\normalsize

\renewcommand\Tstrut{\rule{0pt}{3.0ex}}         
\begin{table}[htb]
\small
\begin{tabular}[h!]{|c|c|c|c|}
\hline
&  \begin{tikzpicture}[scale=.3] 
    \draw[->] (0,0) -- (-1,0) node[left] {$\scriptstyle 1$};
    \draw[->] (0,0) -- (-1,-1) node[left] {$\scriptstyle 1$};
    \draw[->] (0,0) -- (0,-1) node[below] {$\scriptstyle \lambda$};
    \draw[->] (0,0) -- (1,-1) node[right] {$\scriptstyle 1$};
    \draw[->] (0,0) -- (1,0) node[right] {$\scriptstyle 2$};
    \draw[->] (0,0) -- (1,1) node[right] {$\scriptstyle 1$};
  \end{tikzpicture}&\begin{tikzpicture}[scale=.3] 
    \draw[->] (0,0) -- (-1,0) node[left] {$\scriptstyle 1$};
    \draw[->] (0,0) -- (-1,1) node[left] {$\scriptstyle 1$};
    \draw[->] (0,0) -- (0,1) node[above] {$\scriptstyle 2$};
    \draw[->] (0,0) -- (1,1) node[right] {$\scriptstyle 1$};
    \draw[->] (0,0) -- (1,0) node[right] {$\scriptstyle 2$};
    \draw[->] (0,0) -- (1,-1) node[right] {$\scriptstyle 1$};
    \draw[->] (0,0) -- (0,-1) node[below] {$\scriptstyle 1$};
  \end{tikzpicture} & \begin{tikzpicture}[scale=.3] 
    \draw[->] (0,0) -- (-1,0) node[left] {$\scriptstyle 2$};
    \draw[->] (0,0) -- (-1,1) node[left] {$\scriptstyle 1$};
    \draw[->] (0,0) -- (0,1) node[above] {$\scriptstyle 1$};
    \draw[->] (0,0) -- (-1,-1) node[left] {$\scriptstyle 1$};
    \draw[->] (0,0) -- (1,0) node[right] {$\scriptstyle 1$};
    \draw[->] (0,0) -- (1,-1) node[right] {$\scriptstyle 1$};
    \draw[->] (0,0) -- (0,-1) node[below] {$\scriptstyle 2$};
  \end{tikzpicture}\ \ 
    \begin{tikzpicture}[scale=.3] 
    \draw[->] (0,0) -- (-1,0) node[left] {$\scriptstyle 2$};
    \draw[->] (0,0) -- (-1,1) node[left] {$\scriptstyle 1$};
    \draw[->] (0,0) -- (0,1) node[above] {$\scriptstyle 2$};
    \draw[->] (0,0) -- (1,1) node[right] {$\scriptstyle 1$};
    \draw[->] (0,0) -- (1,0) node[right] {$\scriptstyle 1$};
    \draw[->] (0,0) -- (0,-1) node[below] {$\scriptstyle 1$};
        \draw[->] (0,0) -- (-1,-1) node[left] {$\scriptstyle 1$};
  \end{tikzpicture} 
\\
\hline
$I$ &{$-\frac {t^2}{x^2}+\frac t x-t^2+tx(1+\lambda t) $}
&$\frac{{t}^{2}}{x^{2}}- \frac{( 1+2t) t}{x}-{ { (1+ 3t
 ) t}x}-{\frac { ( 1+3t )  (1+ 4t ) 
}{x+1}}+{\frac { (1+ 3t ) ^{2}}{ ( x+1 ) ^{2}}}$ &see Lemma~\ref{lem:invariant_invariants} \Tstrut
\\ & & & and the
\\
$J$ & $t^2y +\frac{1+\lambda t}{y+1}- \left(\frac{1+\lambda
    t}{y+1}\right)^2$&
$\frac{{t}^{2}}{y^{2}}- \frac{( 1+2t) t}{y}-{ { (1+ 3t
 ) t}y}-{\frac { ( 1+3t )  ( 1+4t ) 
}{y+1}}+{\frac { ( 1+3t ) ^{2}}{ ( y+1 ) ^{2}}}$ & previous model \Bstrut
\\
\hline
  \end{tabular}
\medskip
  \caption{Rational invariants for weighted models.}
  \label{tab:ratinv-weighted}
\end{table}
\normalsize

In Section~\ref{sec:analysis}, we introduce a weaker notion of {(possibly
non-rational)}   invariants, which guarantees that any {non-singular} quadrant
model now has a weak invariant. One key difference with the
algebraic setting of this section is that the new notion is
analytic in nature.

 \subsection{Decoupling functions}
 \label{subsec:decoupling_functions}
We now return to  the
identity~\eqref{decouple-gessel}, which we first formalize into an
apparently more demanding condition.
\begin{Definition}\label{def:decoupled}
  A quadrant model  is \emm decoupled, if there exist
  rational functions $F(x) \in \qs(x,t)$ and $G(y)\in \qs(y,t)$ such
  that $xy\equiv F(x)+G(y)$.
   The functions $F(x)$ and   $G(y)$ are said to form a \emm  decoupling pair, for the model.
\end{Definition}

Again, this is equivalent to a statement involving a single function
$G(y)$, as used in the previous section (see~\eqref{decouple-gessel}).

\begin{Lemma}\label{lem:dec-weak}
  Assume that there exists a rational function $G(y) \in \qs(y,t)$
  such that 
\beq\label{G-eq}
xY_0-xY_1= G(Y_0)-G(Y_1),
\eeq
where $Y_0$ and $Y_1$ are the
  roots of the kernel $K(x,y)$, solved for $y$. Define
  $F(x):=xY_0-G(Y_0)=xY_1-G(Y_1)$. Then $F(x) \in \qs(x,t)$, and
  $(F(x), G(y))$ is a decoupling pair for the model.
\end{Lemma}
\begin{proof}
  We have 
\beq\label{F-G}
F(x)= \frac 1 2 \left( xY_0-G(Y_0)+xY_1-G(Y_1)\right),
\eeq
hence $F(x)$ is a rational function of $x$ and $t$ since it is
symmetric in $Y_0$ and $Y_1$. By Lemma~\ref{lem:equiv-equiv}, the property  $F(x)=xY_i-G(Y_i)$ tells
us precisely that $F(x)+G(y)\equiv xy$.
\end{proof}

\noindent{\bf Example.} In Gessel's case, we had $G(y)=-1/(t(1+y))$ (see~\eqref{decouple-gessel}),
corresponding to $F(x)=1/t-1/x$. \qee

\medskip
\noindent {\bf Remark.} By combining~\eqref{SYi} and~\eqref{G-eq}, we see
that if both series $Q(x,Y_i)$ are well defined, then
\[ 
S(Y_0)-G(Y_0)=S(Y_1)-G(Y_1),
\] 
with $S(y)=K(0,y)Q(0,y)$. In Tutte's terminology, this would make
$S-G$ a second ``invariant''. But our terminology is more restrictive,
as our invariants must be rational.

\medskip
Now, which of the $79$ quadrant models are decoupled? Not all, at any rate: for any model
that has a vertical symmetry, the series $Y_i$ are symmetric in $x$
and $1/x$, and so any expression for 
$x$ of the form $
(G(Y_0)-G(Y_1))/(Y_0-Y_1)$ would be at the same time an expression for
$1/x$. 

In the case of a finite group, we give in Theorem~\ref{Thm:decoupling_orbit-sum}
below a criterion for the existence of a decoupling pair, as well as
{an explicit pair} when the criterion holds. This shows that exactly four of
the $23$ finite group models are decoupled  (and these are, as one can expect from the algebraicity
result of Section~\ref{sec:gessel}, those with an
algebraic \gf). The four weighted models of Figure~\ref{fig:alg_models}, right, are
also decoupled. 

For models with an infinite
group, we have first resorted to an experimental approach to construct  decoupling functions.
Indeed, one can try to prescribe the form of the partial fraction expansion of $G(y)$: we first set
\[ 
G(y)= \sum_{i=1}^d a_i y^i + \sum_{i=1}^m \sum_{e=1}^{d_i}
\frac{\alpha_{i,e}}{(y-r_i)^e},
\] 
for fixed values of $d$, $m$, $d_1,  \ldots, d_m$,  with the
values $r_i$ of the poles and the coefficients $a_i$ and
$\alpha_{i,e}$ being yet to determine.
We then express $(G(Y_0)-G(Y_1))/(Y_0-Y_1)$ as a rational function in $t$, $x$, the $r_i$ and $\alpha_{i,e}$,  and require that this is equal to $x$. This gives a system of polynomial equations
relating the $a_i$, $\alpha_{i,e}$ and $r_i$. Solving this system
tells us whether the model has a decoupled pair  for our
choice of $d$, $m$ and the $d_i$.

In this way, we discovered 9 decoupled models among the 56 that
have an infinite group. We could then prove that there are no others. 
The following theorem summarizes our results.

\begin{Theorem}
\label{thm:identification_decoupling}
Among the $79$ 
quadrant models, exactly $13$ are  decoupled: the $4$ 
 models of Figure~\ref{fig:alg_models}, left, and the $9$
models of Table~\ref{tab:decoupling_functions-infinite}.
Moreover, the $4$ weighted models shown on the right of
Figure~\ref{fig:alg_models} are also  decoupled.
\end{Theorem}

The rest of Section~\ref{subsec:decoupling_functions} is devoted to
proving the above theorem. Tables~\ref{tab:decoupling_functions-finite}
and~\ref{tab:decoupling_functions-infinite} give explicit 
decoupling pairs, respectively for finite and infinite groups. One can
easily check that they satisfy Definition~\ref{def:decoupled}. The key
point is then to prove that there are no other (unweighted) decoupled models. To prove this, we consider separately the finite and
infinite group cases.

\begin{table}[ht]
\begin{center}
\begin{tabular}{|c|c|c|c|c||c|c|}
\hline
Model & \begin{tikzpicture}[scale=.4] 
    \draw[->] (0,0) -- (1,1);
    \draw[->] (0,0) -- (-1,0);
    \draw[->] (0,0) -- (0,-1);
        \draw[-] (0,-1) -- (0,-1) node[below] {\phantom{$\scriptstyle 1$}};
  \end{tikzpicture} &  \begin{tikzpicture}[scale=.4] 
    \draw[->] (0,0) -- (0,1);
    \draw[->] (0,0) -- (-1,-1);
    \draw[->] (0,0) -- (1,0);
    \draw[-] (0,-1) -- (0,-1) node[below] {\phantom{$\scriptstyle 1$}};
  \end{tikzpicture} &     \begin{tikzpicture}[scale=.4] 
      \draw[->] (0,0) -- (0,-1);
    \draw[->] (0,0) -- (1,1);
    \draw[->] (0,0) -- (-1,0);
    \draw[->] (0,0) -- (0,1);
    \draw[->] (0,0) -- (-1,-1);
    \draw[->] (0,0) -- (1,0);
    \draw[-] (0,-1) -- (0,-1) node[below] {\phantom{$\scriptstyle 1$}};
  \end{tikzpicture} &   \begin{tikzpicture}[scale=.4] 
      \draw[->] (0,0) -- (-1,-1);
    \draw[->] (0,0) -- (1,1);
    \draw[->] (0,0) -- (-1,0);
    \draw[->] (0,0) -- (1,1);
    \draw[->] (0,0) -- (1,0);
    \draw[-] (0,-1) -- (0,-1) node[below] {\phantom{$\scriptstyle 1$}};
  \end{tikzpicture} & \begin{tikzpicture}[scale=.3] 
    \draw[->] (0,0) -- (-1,0) node[left] {$\scriptstyle 1$};
    \draw[->] (0,0) -- (-1,-1) node[left] {$\scriptstyle 1$};
    \draw[->] (0,0) -- (0,-1) node[below] {$\scriptstyle \lambda$};
    \draw[->] (0,0) -- (1,-1) node[right] {$\scriptstyle 1$};
    \draw[->] (0,0) -- (1,0) node[right] {$\scriptstyle 2$};
    \draw[->] (0,0) -- (1,1) node[right] {$\scriptstyle 1$};
  \end{tikzpicture} & \begin{tikzpicture}[scale=.3] 
    \draw[->] (0,0) -- (-1,0) node[left] {$\scriptstyle 1$};
    \draw[->] (0,0) -- (-1,1) node[left] {$\scriptstyle 1$};
    \draw[->] (0,0) -- (0,1) node[above] {$\scriptstyle 2$};
    \draw[->] (0,0) -- (1,1) node[right] {$\scriptstyle 1$};
    \draw[->] (0,0) -- (1,0) node[right] {$\scriptstyle 2$};
    \draw[->] (0,0) -- (1,-1) node[right] {$\scriptstyle 1$};
    \draw[->] (0,0) -- (0,-1) node[below] {$\scriptstyle 1$};
   \end{tikzpicture}  \\
\hline
$F$& $-\frac{1}{x}+\frac 1 t  $& $\frac x t -x^2$&
$\frac{x-t-tx^2}{t(1+x)}$
&$-\frac{1}{x}  +\frac  1t  $                    
& $-x-\frac 1 x +\frac 1 t$ 
& $-x+\frac{1}{x}-\frac{1+3t}{t(1+x)}+ \frac{1+4t}t$ \Tstrut\Bstrut
\\
$G$& $-\frac{1}{y}$& $-\frac{1}{y}$& $-\frac{1}{y}$
 &$- \frac1{t(1+y)}$ &$-\frac{1+\lambda t}{t(1+y)}$
& $-y+\frac{1}{y}-\frac{1+3t}{t(1+y)}$ \Bstrut
 \\ 
\hline
\end{tabular}

\vskip 2mm
\begin{tabular}{|c|c|}
\hline
  \begin{tikzpicture}[scale=.3] 
    \draw[->] (0,0) -- (-1,0) node[left] {$\scriptstyle 2$};
    \draw[->] (0,0) -- (-1,1) node[left] {$\scriptstyle 1$};
    \draw[->] (0,0) -- (0,1) node[above] {$\scriptstyle 1$};
    \draw[->] (0,0) -- (-1,-1) node[left] {$\scriptstyle 1$};
    \draw[->] (0,0) -- (1,0) node[right] {$\scriptstyle 1$};
    \draw[->] (0,0) -- (1,-1) node[right] {$\scriptstyle 1$};
    \draw[->] (0,0) -- (0,-1) node[below] {$\scriptstyle 2$};
  \end{tikzpicture} & \begin{tikzpicture}[scale=.3] 
    \draw[->] (0,0) -- (-1,0) node[left] {$\scriptstyle 2$};
    \draw[->] (0,0) -- (-1,1) node[left] {$\scriptstyle 1$};
    \draw[->] (0,0) -- (0,1) node[above] {$\scriptstyle 2$};
    \draw[->] (0,0) -- (1,1) node[right] {$\scriptstyle 1$};
    \draw[->] (0,0) -- (1,0) node[right] {$\scriptstyle 1$};
    \draw[->] (0,0) -- (0,-1) node[below] {$\scriptstyle 1$};
        \draw[->] (0,0) -- (-1,-1) node[left] {$\scriptstyle 1$};
  \end{tikzpicture}  \\
\hline
%
$
  \frac{(1+3t)^2}{t^2(1+x)^2}-\frac{(1+2t)(1+3t)}{t^2(1+x)}+\frac{1+2t}t
  -x$
&$-x-\frac{1+3t}{t(1+x)}+\frac{1+t}t$ \Tstrut\Bstrut\\
$-y^2 +\frac{ y(1+{t})}t+\frac{1+3t}{yt}$ & $-y-\frac{1}{y}$ \Bstrut\\ 
\hline
\end{tabular}
\end{center}
\medskip
  \caption{Decoupling functions for algebraic models
(unweighted or weighted).}
  \label{tab:decoupling_functions-finite}
\end{table}

\renewcommand\Tstrut{\rule{0pt}{9.0ex}}         
 \begin{table}[h!]
\begin{center}
\begin{tabular}{|c|c|c|c|c|c|c|}
\hline
Model & \begin{tikzpicture}[scale=.4] 
    \draw[->] (0,0) -- (0,1);
    \draw[->] (0,0) -- (1,0);
    \draw[->] (0,0) -- (-1,0);
    \draw[->] (0,0) -- (-1,-1);
    \draw[-] (0,-1) -- (0,-1) node[below] {$\scriptstyle \# 1$};
  \end{tikzpicture} &   \begin{tikzpicture}[scale=.4] 
    \draw[->] (0,0) -- (0,1);
    \draw[->] (0,0) -- (1,0);
    \draw[->] (0,0) -- (-1,1);
    \draw[->] (0,0) -- (-1,-1);
    \draw[-] (0,-1) -- (0,-1) node[below] {$\scriptstyle \# 2$};
  \end{tikzpicture} &     \begin{tikzpicture}[scale=.4] 
    \draw[->] (0,0) -- (0,1);
    \draw[->] (0,0) -- (1,1);
    \draw[->] (0,0) -- (0,-1);
    \draw[->] (0,0) -- (-1,0);
    \draw[-] (0,-1) -- (0,-1) node[below] {$\scriptstyle \# 3$};
  \end{tikzpicture} & \begin{tikzpicture}[scale=.4] 
    \draw[->] (0,0) -- (0,1);
    \draw[->] (0,0) -- (1,0);
    \draw[->] (0,0) -- (1,-1);
    \draw[->] (0,0) -- (-1,0);
    \draw[-] (0,-1) -- (0,-1) node[below] {$\scriptstyle \# 4$};
  \end{tikzpicture} &     \begin{tikzpicture}[scale=.4] 
    \draw[->] (0,0) -- (0,1);
    \draw[->] (0,0) -- (1,0);
    \draw[->] (0,0) -- (1,1);
    \draw[->] (0,0) -- (-1,-1);
    \draw[->] (0,0) -- (-1,0);
    \draw[-] (0,-1) -- (0,-1) node[below] {$\scriptstyle \# 5$};
  \end{tikzpicture} &       \begin{tikzpicture}[scale=.4] 
    \draw[->] (0,0) -- (0,1);
    \draw[->] (0,0) -- (0,-1);
    \draw[->] (0,0) -- (1,1);
    \draw[->] (0,0) -- (-1,-1);
    \draw[->] (0,0) -- (-1,0);
    \draw[-] (0,-1) -- (0,-1) node[below] {$\scriptstyle \# 6$};
  \end{tikzpicture} \Tstrut\\
\hline
$F$& $-x^2+\frac{x}{t}-1$ &$-x^2+\frac{x}{t}$ &$-\frac{1}{x}+\frac 1 t$
&$\frac{1}{x^2}-\frac{1}{xt}-x+1$ 
& $-\frac{1}{x}+\frac 1 t +1$ &$\frac{x-t}{t(x+1)}$ 
\renewcommand\Tstrut{\rule{0pt}{3.0ex}}         
\Tstrut \Bstrut\\
$G$& $-\frac{1}{y}$ &$-y-\frac{1}{y}$ &$-y-\frac{1}{y}$ & $-y^2+\frac{y}{t}+\frac{1}{y}$&$-\frac{1+t}{t(y+1)}-y$ & $-\frac{1}{y}$ \Bstrut \\
\hline
\end{tabular}

\medskip\renewcommand\Tstrut{\rule{0pt}{9.0ex}}         
\begin{tabular}{|c|c|c|}
\hline
     \begin{tikzpicture}[scale=.4] 
    \draw[->] (0,0) -- (-1,1);
    \draw[->] (0,0) -- (-1,0);
    \draw[->] (0,0) -- (1,0);
    \draw[->] (0,0) -- (-1,-1);
    \draw[->] (0,0) -- (0,1);
    \draw[-] (0,-1) -- (0,-1) node[below] {$\scriptstyle \# 7$};
  \end{tikzpicture} & \begin{tikzpicture}[scale=.4] 
    \draw[->] (0,0) -- (1,1);
    \draw[->] (0,0) -- (0,-1);
    \draw[->] (0,0) -- (1,0);
    \draw[->] (0,0) -- (-1,0);
    \draw[->] (0,0) -- (0,1);
    \draw[-] (0,-1) -- (0,-1) node[below] {$\scriptstyle \# 8$};
  \end{tikzpicture} &    \begin{tikzpicture}[scale=.4] 
    \draw[->] (0,0) -- (1,0);
    \draw[->] (0,0) -- (0,-1);
    \draw[->] (0,0) -- (0,1);
    \draw[->] (0,0) -- (-1,1);
    \draw[->] (0,0) -- (1,-1);
    \draw[-] (0,-1) -- (0,-1) node[below] {$\scriptstyle \# 9$};
  \end{tikzpicture} \Tstrut\\
\hline
 $-x^2+\frac{x}{t}-1$ &$-x-\frac{1}{x}+\frac 1 t$
&$\frac{(t+1)^2}{(x+1)^2t^2}-\frac{(1+t)(1+2t)}{(x+1)t^2}-\frac{1}{x}-x+1+\frac
  1 t$ 
\renewcommand\Tstrut{\rule{0pt}{3.0ex}}         
\Tstrut \Bstrut\\
$-y-\frac{1}{y}$ &$-\frac{1}{y}-y$  &$-\frac{1}{y^2}+\frac{1}{ty}+\frac{(t+1)y}{t}-y^2$   \Bstrut\\
\hline
\end{tabular}
\normalsize
\end{center}
\medskip
  \caption{Decoupling functions for nine infinite group    models.}
  \label{tab:decoupling_functions-infinite}
\end{table}

\subsubsection{The finite group case}
In this case, we have found a systematic procedure to decide whether there exists a decoupling pair, and to construct one (when it exists). We consider in fact a more general problem, consisting
in writing a rational function $H(x,y)$ as $F(x)+G(y)$ when the pair
$(x,y)$ cancels the kernel. The above definition of decoupled models deals
with the case $H(x,y)=xy$, but the general case is not harder and allows
us to consider starting points other than $(0,0)$. This is further
discussed in Section~\ref{sec:start}. 

 As we shall see, decoupling functions exist if and only if  a certain
 rational function, called \emm orbit sum,, vanishes. Our approach
 adapts~\cite[Thm.~4.2.9 and Thm.~4.2.10]{FIM-99} to our context\footnote{In~\cite{FIM-99}, decoupling functions are called
  \emm particular rational solutions,, as indeed they are particular
  solutions of the functional equation~\eqref{SYi}. In~\cite{DHRS-17}, they
  are a certain type of \emm telecopers,.}.

\medskip
\noindent {\bf Notation.} We recall that  the
group $\mathcal G (\cS)$ is generated by the involutions $\Phi$ and
$\Psi$ defined by~\eqref{eq:generators}. Its elements consist of all
alternating products of $\Phi$ and $\Psi$. We denote $\Theta=\Psi\cdot\Phi$, and observe that
$\Theta^{-1}=\Phi\cdot\Psi$. Each element $\gamma$ of the group has a
\emm sign,, depending on the number of generators $\Phi$ and $\Psi$ that it
involves:  $\sign(\Theta^k)=1$, while $\sign(\Phi\cdot\Theta^k)=-1$ for all $k$.
 
 For $A(x,y)\in \mathbb{Q}(x,y)$, and $\omega=\sum_{\gamma\in
  \mathcal G (\cS)}c_\gamma\cdot \gamma$ an element of the group
algebra $\mathbb{Q}[\mathcal G(\cS)]$, we denote 
\begin{equation}
\label{eq:A_omega}
     A_\omega(x,y):=\sum_{\gamma\in  \mathcal G (\cS)} c_\gamma\,
     A(\gamma(x,y)).
\end{equation}
This is again a rational function in $x$ and $y$. 
Defining $\sigma$ as the sum of all elements of 
the group  $\mathcal G(\cS)$,
we obtain as a special case the notation
$H_\sigma(x,y):=\sum_{\gamma\in  \mathcal G (\cS)}
H(\gamma(x,y))$ used in Theorem~\ref{thm:finite-invariant}.

\medskip

We now generalize Definition~\ref{def:decoupled}.
\begin{Definition}\label{def:dec-gen}
 Given a quadrant model $\cS$, and its kernel $K(x,y)$,  a function $H(x,y)\in \mathbb{Q}(x,y)$ is \emm decoupled, if there exist $F(x) \in
 \mathbb{Q}(x,t)$ and $G(y) \in \qs(y,t)$ such that
\[  
H(x,y)\equiv F(x)+G(y).
\]  
\end{Definition}

\begin{Theorem}
\label{Thm:decoupling_orbit-sum} 
Let $\mathcal S$ be a step set such that the associated group
$\mathcal G (\cS)$ is finite of order $2n$. 
 Then $H\in \mathbb{Q}(x,y)$ is decoupled 
if and only if $H_\alpha(x,y)=0$, where $\alpha$ is the following alternating sum:
\[  
\alpha=  \sum_{\gamma \in  \mathcal G (\cS)}\sign(\gamma)\gamma.
\] 
In this case, one can take $F(x)=H_\tau(x,Y_0)+H_\tau(x,Y_1)$,
where
\begin{equation*}
     \tau = -\frac{1}{n}\sum_{i=1}^{n-1}i\, \Theta^i.
\end{equation*}  
The corresponding value of $G$ is then
$
G(y)= H_{\widetilde \tau} (X_0,y)+ H_{\widetilde \tau} (X_1,y) + (1-1/(2n))
J(y)$, where
$J(y)$ is the invariant defined in Theorem~\ref{thm:finite-invariant} and
\[ 
\widetilde \tau= -\frac{1}{n}\sum_{i=1}^{n-1}i\, \Theta^{-i}.
\] 
\end{Theorem}

\begin{proof} Assume that $H$ is decoupled. Then for every pair
  $(u,v)$ in the orbit of $(x,Y_0)$, we have $K(u,v)=0$, and hence $H(u,v)=F(u)+G(v)$.
  Now recall that if $(u',v')=  \Phi(u,v)$, then $v'=v$ (and
  analogously
  for the transformation $\Psi$). Hence taking the
  alternating sum of $H(u,v)=F(u)+G(v)$ over the orbit of $(x,Y_0)$,
  we find  that $H_\alpha(x,Y_0)=0$, which implies that $H_\alpha(x,y)$ is
  uniformly zero since $Y_0$ depends on $t$ while $x$ and $H$ do not.

\medskip
Suppose now that $H_\alpha=0$, and define $F(x)$ and $G(y)$ as
above. Note that 
\[ 
F(x) = H_{\tau + \tau \Psi}(x,Y_0)
\] 
and
\[ 
G(Y_0)= H_{\widetilde \tau + \widetilde \tau \Phi+ (1-1/(2n))\sigma}(x,Y_0),
\] 
so that
\beq\label{FplusG}
  F(x)+G(Y_0) = H_{\tau + \tau \Psi+ \widetilde \tau + \widetilde \tau
    \Phi+ (1-1/(2n))\sigma}(x,Y_0).
\eeq
But $\Theta^i \Psi= \Theta^{i+1}\Phi$ and $\Theta^{-i}=
\Theta^{n-i}$. Hence
\begin{align*}
\tau + \tau \Psi+ \widetilde \tau + \widetilde \tau    \Phi &= 
- \frac 1 n \sum_{i=1}^{n-1} i \left( \Theta^i + \Theta^i \Psi +
                                                      \Theta^{-i} +
                                                      \Theta^{-i}\Phi\right)\\
& = - \frac 1 n \sum_{i=1}^{n-1} i \left( \Theta^i+ \Theta^{i+1}\Phi +
                                                      \Theta^{n-i} +
                                                      \Theta^{n-i}\Phi\right).
\end{align*}
Upon grouping the terms in $\Theta^j$, and those in $\Theta^j \Phi$,
we obtain
\begin{align*}
\tau + \tau \Psi+ \widetilde \tau + \widetilde \tau    \Phi+ (1-1/(2n))\sigma
& 
=\id  - \frac 1 {2n} \alpha,
\end{align*}
where $\sigma$ and $\alpha$ are respectively the sum, and the
alternating sum, of the elements of $\mathcal G(\cS)$. Returning
to~\eqref{FplusG}, and using $H_\alpha=0$, we obtain that
$F(x)+G(Y_0)=H_{\id}(x,Y_0)=H(x,Y_0)$, so that the pair $(F,G)$ indeed decouples $H$, by Lemma~\ref{lem:equiv-equiv}.
\end{proof}

\begin{proof}[Proof of Theorem~\ref{thm:identification_decoupling} in the finite group case]
We now apply Theorem~\ref{Thm:decoupling_orbit-sum} to $H(x,y)=xy$ and
to the $23$ models
with a finite group (listed  for instance in~\cite[Sec.~8]{BMM-10}). We find
indeed that the alternating orbit sum of $H(x,y)$ vanishes in four
cases only. Applying to them the procedure of the theorem gives
the decoupling functions of Table~\ref{tab:decoupling_functions-finite}. The  procedure applies as
well to weighted models. Note that we have sometimes replaced the pair $(F,G)$
of Theorem~\ref{Thm:decoupling_orbit-sum} by the decoupling pair $(F+cI, G-cJ)$, where
$(I,J)$ is a pair of invariants, to simplify the expression  $G(y)$. \end{proof}

\subsubsection{The infinite group case.}\label{sec:infinite}

Recall that we have discovered experimentally $9$ decoupled models with an infinite group, shown in Table~\ref{tab:decoupling_functions-infinite}. Of course, it is straightforward to check, in each case, that $F(x)+G(y)-xy$ is divisible by the kernel. Now how can we prove that the remaining $47$ models with an infinite group are \emm not, decoupled? Unfortunately, when $\cG(\cS)$ is infinite,
  we have not found any criterion comparable to Theorem~\ref{Thm:decoupling_orbit-sum}  that would decide whether {the function $H(x,y)=xy$ is decoupled}. 
Our approach involves
 some case-by-case analysis, and relies on the two following observations. Here, we assume that $(F(x), G(y))$ is a decoupling pair, and denote by $\GK$ the algebraic closure of $\qs(t)$. 
\begin{itemize}
\item
   If $X \in \GK$ is a pole of $F(x)$, and $Y\in \GK$ satisfies $K(X,Y)=0$, then $Y$ is a pole of $G(y)$. By a symmetric argument, if $X'\in \GK$ satisfies $K(X',Y)=0$, then $X'$ is another pole of~$F(x)$. Since a rational function has only finitely many poles, this procedure must stop or loop at some point. This is formalized in Lemma~\ref{lem:infinite-orbit}.
\item
  If $u\in \cs$ is a root of the polynomial $a(x)= [y^2] K(x,y)$,  then either it is a pole of $F(x)$, or $G(y)/y$ tends to $u$ at infinity (Lemma~\ref{lem:infinite-orbit2}).
\end{itemize}
The idea is then to argue \emm ad absurdum,. We  choose a root $u$ of $a(x)$, and prove, thanks to the first observation,  that it cannot be a pole of  $F$ (because the propagation of poles does not stop for the value $u$). We then use the second observation to determine the behavior of $G(y)/y$ at infinity, and derive some contradiction from it.
Note that the idea of propagating poles is  classical when  searching for rational
solutions of  difference equations~(see e.g.~\cite{abramov}), and is also used in the
quadrant context in~\cite{DHRS-17}.

In order to formalize the above observations, we first need to introduce a variant of the transformations $\Phi$ and $\Psi$ defined in~\eqref{eq:generators}.

\begin{Definition}\label{def:phi-psi}
  Let $(X,Y)\in \GK^2$ satisfy $K(X,Y)=0$. We define $\phi(X,Y)=(X',Y)$, where $X'$ is
the other root (if any) of the equation $K(x,Y)=0$ (solved for
$x$). We define analogously
$\psi(X,Y)=(X,Y')$, where $Y'$ is the other root of {the equation}
  $K(X,y)=0$, solved for $y$.

For $X\in \GK$, the equation $K(X,y)=0$, when solved for $y$, 
has at most two solutions  $Y$ and
  $Y'$, which belong to $\GK$ as well (we ignore infinite solutions). The
\emm $x$-orbit, of $X$ is the set of pairs in $\GK^2$ that can be
obtained from $(X,{Y})$ or $(X,{Y'})$ by repeated  applications of the
transformations $\phi$ and $\psi$ (as long as they are well defined).

We define the $y$-orbit of an element $Y$ of $\GK$ in a similar
fashion, starting from the pairs $({X},Y)$ and
$({X'},Y)$ {such that $K(X,Y)=K(X',Y)=0$}.
\end{Definition}

Recall the expansion~\eqref{K-abc} of $K(x,y)$ in powers of $x$ or
  $y$, and the definition~\eqref{eq:generators} of the transformations $\Phi$ and
  $\Psi$. Then $\phi(X,Y)$ is well defined if and only if $\widetilde
  a(Y)\not = 0$. In this case,  $\phi(X,Y)$ coincides with
  $\Phi(X,Y)$, unless $X=0$. Note that $X=0$ implies that $\widetilde
  c(Y)=0$, and then $X'=-\widetilde b(Y)/\widetilde a(Y)$, while  $\Phi(X,Y)$
  is undefined. If $\widetilde d(Y)=0$, where $\widetilde d(y)=\widetilde b(y)^2
  -4 \widetilde a(y) \widetilde c(y)$, then $X'=X$ and the iterated
  application of $\phi$ and $\psi$ will not produce any new pair. Of
  course, analogous statements hold for the determination of $\psi(X,Y)$.

\medskip

\noindent {\bf Examples.}  Consider the (decoupled) model $\cS=\{\nearrow, \uparrow,
\leftarrow, \downarrow\}$, and take $X=0$, which satisfies $a(X)=0$. Then $K(X,y)=K(0,y)=ty$, 
and the equation $K(X,y)=0$ admits only one root, 
which is $Y=0$. So we start from the pair $(0,0)$. But $\widetilde a(Y)=tY^2=0$, thus
$K(x,Y)=K(x,0)=tx$, and the $x$-orbit of $0$ reduces to the pair $(0,0)$.

Consider now
$\cS=\{\uparrow, \nwarrow, \downarrow, \searrow, \rightarrow\}$ (which
is also decoupled), and let us determine again the $x$-orbit of
$X=0$, which satisfies $d(X)=0$. Since $K(X,y)=K(0,y)=ty^2$, we start from the pair $(X,Y)=(0,0)$. Now
$K(x,Y)=K(x,0)=tx(x+1)$ so we add the pair $\phi(0,0)=(-1,0)=(X',Y)$ (note that
$\Phi(0,0)$ is not well defined). Finally, $a(X')=a(-1)=0$, thus $K(X',y)=K(-1,y)=y(1-t)$ admits
only the root $0$. Hence the $x$-orbit of $0$ consists of $(0,0)$ and $(-1,0)$.\qee

\medskip
We can now formalize the two observations made above  Definition~\ref{def:phi-psi}.

\begin{Lemma}\label{lem:infinite-orbit}
Let $(F(x),G(y))$ be a decoupling pair for a model $\cS$.
If $X\in \GK$ is a pole of $F(x)$, then for
 each element $(X',Y')$ in its $x$-orbit, $X'$ is a pole of $F$ and $Y'$ is a
 pole of $G$. In particular, the $x$-orbit of $X$ must be finite.

 Consequently, if $X\in \GK$ has an infinite $x$-orbit, then it is not a pole of $F(x)$.
 \end{Lemma}
 
\begin{proof} Let us denote 
\begin{equation}\label{eq:fracFG}
F(x)=\frac{M(x)}{D(x)}~ \textrm{ and }~G(y)=\frac{N(y)}{E(y)},
\end{equation}
for some coprime polynomials $M(x),D(x)\in \cs(t)[x]$, and some coprime polynomials $N(y),E(y)\in \cs(t)[y]$.
Since $(F,G)$ is a decoupling pair, there exists a polynomial $P(x,y)\in \cs(t)[x,y]$ such that 
\[
  F(x)+G(y)-xy=\frac{K(x,y)P(x,y)}{D(x)E(y)},
\]
or equivalently
\begin{equation}\label{eq:decoupling-poly}
M(x)E(y)+D(x)N(y)-xyD(x)E(y)= K(x,y)P(x,y).
\end{equation}
If $X$ is a pole of $F$ (that is, a root of $D$) and $K(X,Y)=0$, then~\eqref{eq:decoupling-poly} gives $M(X)E(Y)=0$. Since $M$ and $D$ have no common root,
$E(Y)=0$ and $Y$ is a pole of $G(y)$. Propagating the reasoning along the $x$-orbit of $X$ proves the first statement of Lemma~\ref{lem:infinite-orbit}. The second statement follows because $F(x)$ has a finite number of poles.
\end{proof}

\begin{Lemma}\label{lem:infinite-orbit2}
Let $(F(x),G(y))$ be a decoupling pair for a model $\cS$.
Let $u\in \cs$ be a root of $a(x):=[y^2]K(x,y)$ that
is not a pole of $F(x)$. Then $\lim_{y\to{\infty}}G(y)/y=u$.
\end{Lemma}

\begin{proof} We use the notation~\eqref{eq:fracFG}, so that~\eqref{eq:decoupling-poly} holds.   Let $E_d \, y^d$ (resp.\ $N_\delta\, y^\delta$) be the leading monomial of $E(y)$ (resp.\ $N(y)$). Note that $G(y) \sim N_\delta / E_d\, y^{\delta-d}$ at infinity. Let $p(x)$
  be the leading coefficient  of $P(x,y)$ 
  in the variable $y$. Let us examine the leading monomials, and leading coefficients, in both sides of~\eqref{eq:decoupling-poly}:
  \begin{itemize}
\item if $\delta<1+d$, we find $-xD(x)E_d=a(x)p(x)$, 
\item if $\delta=1+d$, we find $D(x)(N_\delta-x E_d)=a(x) p(x)$,
 \item if $\delta>1+d$, we find $N_\delta D(x)=a(x) p(x)$.
  \end{itemize}
Assume now that $u\in \cs$ is a root of  $a(x)$ but not of $D(x)$. In the first case, we get  $u=0=\lim_{y\rightarrow \infty} G(y)/y$; in the second case, we get $N_\delta-u E_d$ so that  $\lim_{y\rightarrow \infty} G(y)/y =N_\delta/ E_d= u$; the third case is impossible.
\end{proof}

We now derive from the above two lemmas three corollaries that will form our toolbox to prove that none of the 47 models that remain under consideration are decoupled. The first corollary builds on the observation that no
decoupled model found so far has $a(x)=t(1+x^2)$ nor $a(x)=t(1+x+x^2)$.

\begin{Corollary}\label{cor:infinite-2-roots}
If the polynomial $a(x)$ has two distinct roots $u$ and $u'$, each of them with an infinite $x$-orbit, then the model  is not decoupled.
\end{Corollary}

\begin{proof}
Assume on the contrary that $(F(x), G(y))$ is a decoupling pair.   By Lemma~\ref{lem:infinite-orbit}, neither~$u$ nor $u'$ can be a pole of $F(x)$. Hence by Lemma~\ref{lem:infinite-orbit2}, we would have $G(y)/y \to u$ and $G(y)/y\to u'\not = u$, which is of course impossible.
\end{proof}

\begin{Corollary}\label{cor:infinite-asympt}
  Assume that the polynomial $a(x)$ has a root $u$ with infinite $x$-orbit, and that one of the branches $Y_i(x)$ grows   as $|x|^\nu$ as $x\to -\infty$ (up to a multiplicative constant, and for infinitely many values of $t$), where $\nu>0$ and $\nu\not \in \ns$.
  Then the model is not decoupled.
\end{Corollary}
In the applications of this corollary that follow,  the value of $\nu$ will always be $1/2$.

\begin{proof} Suppose on the contrary that a decoupling pair $(F(x),G(y))$ exists. We fix $t\in\cs$ such that $F(x)$ and $G(y)$ are well defined and $Y_i(x)\sim_{x\to -\infty}  |x|^{\nu}$  (up to a multiplicative constant).
  By Lemmas~\ref{lem:infinite-orbit} and~\ref{lem:infinite-orbit2} we get $G(Y_i(x))\sim_{x\to -\infty} uY_i(x)$. Hence using the decoupling identity $F(x)+G(Y_i(x))=xY_i(x)$ as $x\to -\infty$, we get $F(x)\sim_{x\to -\infty}-|x|^{1+\nu}$, which is not a possible asymptotic behavior for a rational function.
 \end{proof}

\begin{Corollary}\label{cor:2-infinite}
Assume that  $a(x)$ has a root $u$ with infinite $x$-orbit, that $\widetilde a(y):=[x^2]K(x,y)$ has a root $v$ with infinite $y$-orbit, and that moreover one of the branches 
$Y_i(x)$ tends to infinity as $x\to +\infty$ (for infinitely many $t\in\cs$). Then the model is not decoupled.
\end{Corollary}

\begin{proof} Suppose on the contrary that a decoupling pair $(F(x),G(y))$ exists.
  The decoupling identity gives $\frac{F(x)}{xY_i(x)}+\frac{G(Y_i(x))}{xY_i(x)}=1$, but Lemmas~\ref{lem:infinite-orbit} and~\ref{lem:infinite-orbit2} (and their counterparts obtained by swapping $x$ and $y$) imply that,  as $x\to +\infty$,
  \[
\frac{F(x)}{xY_i(x)}\sim    \frac v{Y_i(x)}, \qquad \frac{G(Y_i(x))}{xY_i(x)}\sim \frac u{x},
  \]
which both tend to $0$. This yields a contradiction.
\end{proof}

As suggested by the above three corollaries, it will be
  crucial, in what follows, to prove that for a given model $\cS$ and
  a given root $u$ of $a(x)$, the $x$-orbit of $u$ is infinite. How can one
  prove this? We start from the (unique) pair $(x_0,y_0):=(u,y_0)$ such that
  $K(u,y_0)=0$, and iterate $\phi$ and $\psi$, thus producing a sequence
  of $x$-orbit elements:
  \beq\label{x-orbit}
(x_0,y_0){\overset{\phi}{\longrightarrow}}   (x_1,y_0){\overset{\psi}{\longrightarrow}}  (x_1 , y_1) {\overset{\phi}{\longrightarrow}} (x_2, y_1){\overset{\psi}{\longrightarrow}} \cdots
\eeq
Each $x_k, y_k$ is rational in $t$, with coefficients in $\qs(u)$.

In some cases, it is very simple to prove that
the above sequence does not stop nor loop, because $(x_0,y_0)$
is the only pair $(X,Y)$ of $\qs(u,t)^2$ such that $K(X,Y)=0$ and
$a(X)d(X)\widetilde a(Y) \widetilde d(Y)=0$. As discussed below
Definition~\ref{def:phi-psi}, these are the only pairs where the above chain can
stop or loop. Consider for instance
$\cS=\{\uparrow, \nwarrow, \swarrow, \searrow\}$, with $u=-1$. This
value $u$ is
the only root of $a(x)$. We find that $y_0=-2t$, so that the $x$-orbit of $u$ is non-empty. The polynomial $\widetilde a(y)$ has no
root, and the discriminants $d(x)$ and $\widetilde d(y)$ do not factor
over $\qs(t)$. Hence the $x$-orbit of $u=-1$ is infinite. This
argument applies to~$8$ cases below and can
be adapted to models that are symmetric in the first or second
diagonal and have exactly two pairs $(X,Y)$ such that $K(X,Y)=0$ and
$a(X)d(X)\widetilde a(Y) \widetilde d(Y)=0$. This proves infiniteness of the
orbit for $7$ other models. 

There is also a more tedious method that applies uniformly to all
models under consideration, and proves that the sequence~\eqref{x-orbit} is infinite (and does not loop) by looking at
the expansions at $t=0$ of the rational functions $x_k$ and $y_k$. Two cases occur:
\begin{itemize}
\item If the model $\cS$ is singular, that is, $c(x)=tx^2$ and $\widetilde
  c(y)=ty^2$, then it is easy to see, by induction on $t$, that
$x_k$ has valuation $2k$, while $y_k$ has valuation $2k+1$. Indeed,
$x_0=u$ is a root of $a(x)$, and is thus a non-zero element of $\cs$ since
$a(0)=1$. Then $y_0= -c(u)/b(u)$ has valuation $1$, because $b(u)$ is
a non-zero multiple of $t$ while $c(u)$ is a polynomial in $t$ with
constant term $-u\not = 0$. Now, suppose
that the assumption holds true for $x_k$ and~$y_k$. We have
\[
  x_{k+1}= \frac 1 {x_k} \frac {\widetilde c(y_k)} {\widetilde a(y_k)}.
\]
Using $\widetilde c(y)=ty^2$ and $\widetilde a(y)=t(1+O(y))$, we find that
$x_{k+1}$ has valuation $2k+2$. Then
\[
  y_{k+1}= \frac 1 {y_{k}} \frac { c(x_{k+1})} { a(x_{k+1})},
\]
and an analogous argument proves that $ y_{k+1}$ has valuation $2k+3$
in $t$.
\item If $\cS$ is not singular, the details of the argument depend on the
  details of the model.   Consider for instance $\cS=\{\nearrow, \nwarrow, \swarrow,
\downarrow\}$, with $u=i$. Using again the
notation~\eqref{x-orbit}, we have $x_0=i$, $y_0=(1-i)t$, and pushing
the calculations further suggests that
$x_{3k}=i+4kit^2+O(t^4)$. Clearly, this would imply that the
$x$-orbit is infinite. To prove this, we proceed as follows. Let 
$X_0=i+\sum_{n\ge 2 } a_nt^n$ be a series in $t$ with complex
coefficients, with $a_2\not = -4i$ (the reason for this condition will
appear later). One of the roots of the equation $K(X_0,Y)=0$ reads
 $Y_0=(1-i)t+O(t^3)$, as can be seen from the equation $Y= \frac t{X_0}
 (1+X_0+Y^2+X_0^2Y^2)$. The coefficients of $Y_0$ lie in $\qs[i, a_2,
 a_3, \ldots]$. We then compute $X_1, X_2, X_3$, and $Y_1, Y_2 ,Y_3$
 such that
 \[
   (X_0,Y_0){\overset{\phi}{\longrightarrow}}
   (X_1,Y_0){\overset{\psi}{\longrightarrow}}  (X_1 , Y_1)
   {\overset{\phi}{\longrightarrow}} (X_2,  Y_1)
   {\overset{\psi}{\longrightarrow}} (X_2,  Y_2)
  {\overset{\phi}{\longrightarrow}} (X_3,  Y_2)
   {\overset{\psi}{\longrightarrow}} (X_3,  Y_3).
   \]
(In practice, we compute them using the transformations $\Phi$ and
$\Psi$.) We thus obtain:
\begin{alignat*}{6}
  X_1&=  \frac 1 {2t^2} + O(1),  &&
  Y_1=(1+i)t + O(t^3),\\
  X_2&=-i +(4i+a_2) t^2+O(t^3), &\hskip 8mm &
  Y_2=\frac 1 {2(4i+a_2) t^3}+ O(t^{-2}),\\
  X_3&=i+(4i+a_2) t^2+O(t^3),&&
  Y_3=(1-i)t+O(t^3).
\end{alignat*}
The condition $a_2\not = -4i$ is required
 because, while the coefficients of $X_1, X_2, X_3$, and $Y_1, Y_3$
 lie in $\qs[i, a_2, a_3, \ldots]$, those of $Y_2$ lie in  $\qs[i,
 1/(a_2+4i), a_2, a_3, \ldots]$. We observe that the pair $(X_3,Y_3)$ has the
 same form as $(X_0, Y_0)$, with $a_2$ replaced by $4i+a_2$, and this
 completes the proof. We say that the $x$-orbit of $u=i$ has \emm
 pseudo-period, $3$. 
\end{itemize}

We now apply our toolbox to the proof of
Theorem~\ref{thm:identification_decoupling}, in the infinite group case. We distinguish
several cases, depending on the value of $a(x)= [x^2] K(x,y)$.  A {\sc Maple} session, available on the authors' webpages,
examines in detail all non-decoupled models.

\medskip
\paragraph{\bf Case 1:  $\boldsymbol{a(x)=t(1+x^2)}$.}  The polynomial
$a(x)$ has two roots, namely $i$ and $-i$. By
Corollary~\ref{cor:infinite-2-roots}, it suffices to prove that their
$x$-orbits are infinite to conclude that the model does not decouple.
The corresponding $7$ models  are labeled 1a in
Table~\ref{tab:nodec}. One is singular, and for the others, the
$x$-orbit of $u=\pm i$ has pseudo-period ranging from $2$ to $7$.
The same argument applies to the $10$ models labeled 1b in
Table~\ref{tab:nodec}, for which $\widetilde a(y)=t(1+y^2)$, upon
exchanging the roles of~$x$ and $y$. Again, one of these models is singular,  and for the others, the
$y$-orbit of $v=\pm i$ has pseudo-period ranging from $2$ to $7$. We have thus proved non-decoupling for 17 models with an infinite group.

\medskip
\paragraph{\bf Case 2: $\boldsymbol{a(x)=t(1+x+x^2)}$.} 
 We now apply Corollary~\ref{cor:infinite-2-roots} to the  roots $j$ and $1/j$, with $j=e^{2i\pi/3}$. This proves  non-decoupling for 12 additional models, indicated
in Table~\ref{tab:nodec} by 2a and~2b (depending on whether the
argument is applied to $a(x)$ or $\widetilde a (y)$). As before, we
find $x$- and $y$-orbits with pseudo-periods ranging from $2$ to $7$,
in addition to one singular model.

\medskip
\paragraph{\bf Case 3: $\boldsymbol{a(x)=t(1+x)}$.} It can be checked
that for every (yet untreated) model such that $a(x)=t(1+x)$, the
$x$-orbit of $u=-1$ is infinite. 

If moreover $\deg(b(x))\le 1$, then $\deg(c(x))=2$ (because there is at least one step with $x$-coordinate $+1$), and we have, for any non-zero real $t$,
\[
  Y_i(x)=\frac{-b(x)\pm \sqrt{b(x)^2-4a(x)c(x)}}{2a(x)}\sim_{x\to -\infty} |x|^{1/2},
\]
up to some multiplicative constant. We then conclude by Corollary~\ref{cor:infinite-asympt}. This takes care of the~$7$ models labeled 3.1a in Table~\ref{tab:nodec}.

For the four models labeled 3.2a in the table, we can apply instead
Corollary~\ref{cor:2-infinite}. First, $u=-1$ is a root of $a(x)$ and
has an infinite orbit. Then, $\widetilde a(y)=t(1+y)$, and $v=-1$ is a
root of $\widetilde a(y)$ with an infinite $y$-orbit.  Finally, $b(x)$ has
degree 2, and for $t>0$,
\[
  Y_1(x)=\frac{-b(x)-\sqrt{b(x)^2-4a(x)c(x)}}{2a(x)}\sim_{x\to \pm \infty} -x.
\]
 Let us now exchange the roles of $x$ and $y$. Then  the same argument
 applies to the model labeled~3.2b: we have  $\widetilde a(y)=t(1+y)$ and
 $v=-1$ has an infinite $y$-orbit, $a(x)=tx$ and $u=0$ has an infinite
 $x$-orbit, and finally one of the branches $X_i(y)$ diverges at infinity.

As before, we find $x$- and $y$-orbits with pseudo-periods ranging from $2$ to $7$ (in addition to two singular models).

\medskip
\paragraph{\bf Case 4:  $\boldsymbol{a(x)=tx}$.} The argument is similar to the one used in the previous case.  This time we apply it  with $u=0$. It can be checked that for each (yet untreated) model of the table such that $a(x)=tx$, the $x$-orbit of $0$ is infinite.

For the three non-symmetric such models (labeled 4.1a), $b(x)$ has degree $0$, the
branches $Y_i(x)$ grow like $| x|^{1/2}$ at $-\infty$, and we
conclude using Corollary~\ref{cor:infinite-asympt}. For the symmetric
one, labeled~4.2a, $b(x)$ has degree $2$, and we apply Corollary~\ref{cor:2-infinite} instead, with $u=v=0$.

The pseudo-periods are found to be $7$ or $8$ in these four cases.

\medskip
\paragraph{\bf Case 5: the last two models.} We are left with the two
symmetric ``forks'', $\cS_1=\{\uparrow, \swarrow, \rightarrow,
\nearrow\}$ and its reverse $\cS_2=\{\leftarrow, \swarrow,
\downarrow, \nearrow\}$. Note that when a symmetric model  decouples,
there exists a decoupling pair of the form $(F(x),F(y))$.

For $\cS_1$ we have $a(x)=tx(1+x)$, and we can check that the
$x$-orbit of $-1$ is infinite, with pseudo-period $4$  (the $x$-orbit of $0$ is empty).  Hence by Lemma~\ref{lem:infinite-orbit2}, $G(y) \sim
-y$ at infinity. Both branches $Y_i$ are finite at infinity, but as
${x\to 0^+}$, one of the branches $Y_i$ is equivalent to $x^{-1/2}$, so that
$G(Y_i(x))\sim - x^{-1/2}$. Plugging this into the decoupling identity
$F(x)+G(Y_i(x))=xY_i(x)$ shows that $F(x)\sim x^{-1/2}$, which is impossible for a rational function.
 
For $\cS_2$ we have $a(x)=tx^2$. We can check that the $x$-orbit of 0
is infinite,  again with pseudo-period $4$, hence by Lemma~\ref{lem:infinite-orbit} the decoupling
function $F(x)$ is finite at $x=0$. As ${x\to 0}$, one of the
branches $Y_i(x)$ is equivalent to $-1/x^2$, hence expanding the
decoupling identity $F(x)+G(Y_i(x))=xY_i(x)$ around $x=0$ gives
$G(Y_i(x))\sim -1/x$. However since $Y_i(x)$ grows like $-1/x^2$, this is impossible for a rational function.

This concludes the proof of Theorem~\ref{thm:identification_decoupling}.\qed

\renewcommand\Tstrut{\rule{0pt}{6.0ex}}         

\begin{table}[htb]
  \centering
  \begin{tabular}{|c|c|c|c|c|c|c|c|c|c|c|c|}
  \hline
\tpic{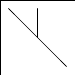}& \tpic{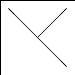}& \tpic{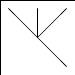}&  \tpic{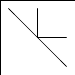}&
\multicolumn{4}{c|}{ 
%
  \tpic{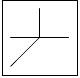} 
\tpic{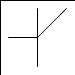} \tpic{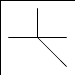} \tpic{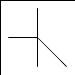}}&
 \multicolumn{4}{c|}{ \tpic{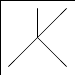} \tpic{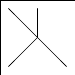}  \tpic{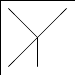} \tpic{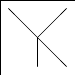}}  \Tstrut
\\
3.1a  & 1a & 1b & 3.2a & 
\multicolumn{4}{l|}{{\bf dec.}  \ \ {\bf dec.} \  \ {\bf dec.}\ \ \ \ 4.1a }& 
\multicolumn{4}{l|}{ \ \ 1b \ \ \ \  3.1a \ \ \ \  1a\ \ \  \ \ \ 1a}
\\ \hline

\multicolumn{4}{|c|}{  \tpic{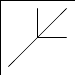} \tpic{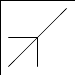} \tpic{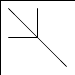}}&
 \multicolumn{4}{c|}{ 
%
\tpic{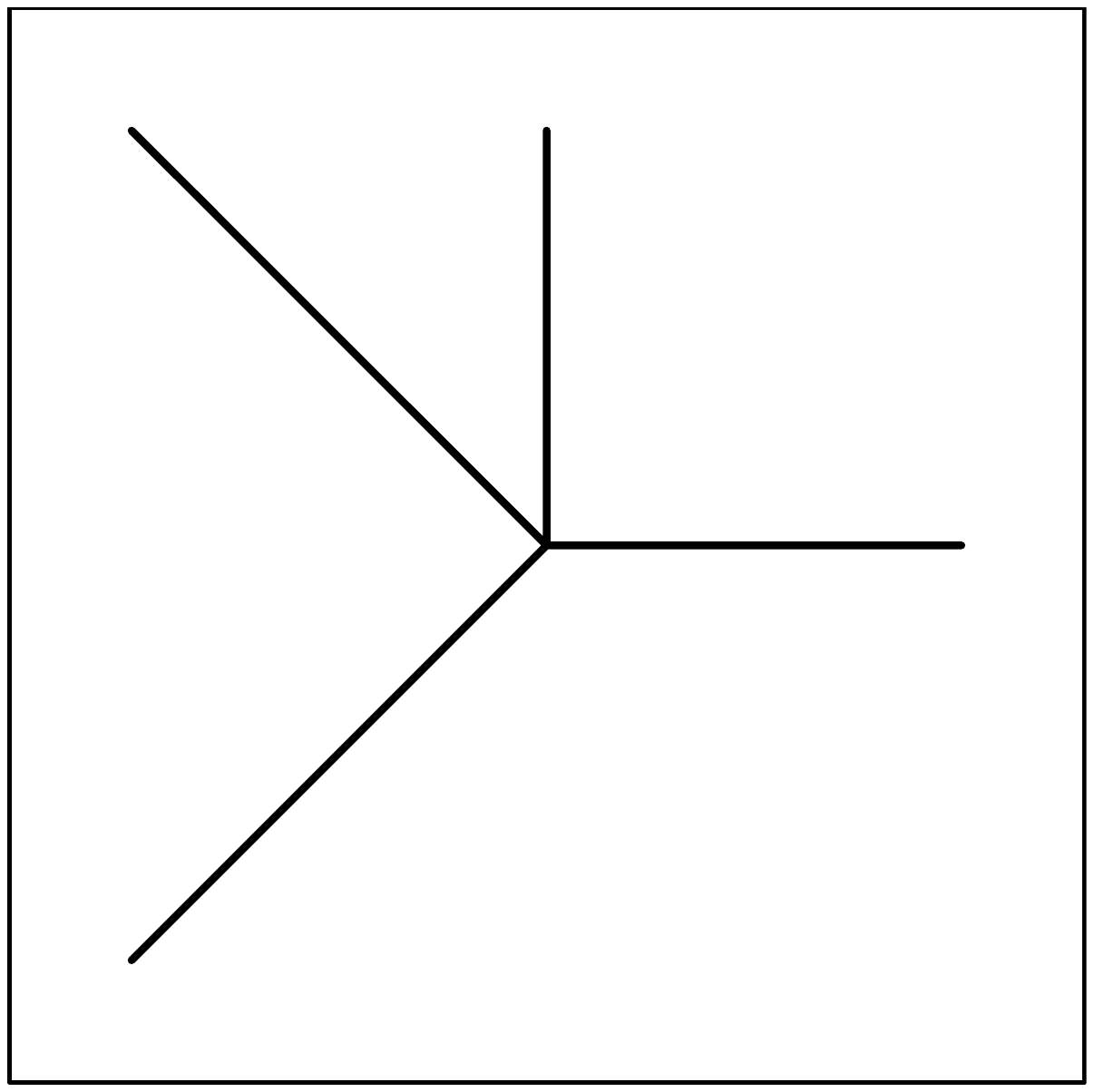} \tpic{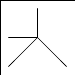} \tpic{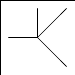} \tpic{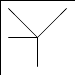}}&
 \multicolumn{4}{c|}{ \tpic{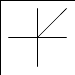} \tpic{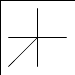} \tpic{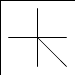}} \Tstrut
\\
\multicolumn{4}{|l|}{\ \ \ \  \ \ \ \ \ 5 \ \ \ \ \ \  5 \ \ \ \ \   3.1a} &
 \multicolumn{4}{l|}{\ {\bf dec.} \ \   4.1a \ \ \ \  1b\ \ \ \ \ \ 1a} &
 \multicolumn{4}{l|}{\ \ \ \ \ {\bf dec.}\ \  \   4.2a\ \ \    3.2b} 
\\ \hline

 \multicolumn{4}{|c|}{ 
%
\tpic{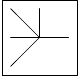} \tpic{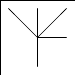}  \tpic{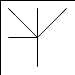}   \tpic{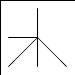} }
 &  \multicolumn{4}{c|}{ \tpic{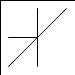}  \tpic{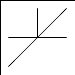} \tpic{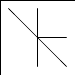}  \tpic{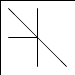}}&
 \multicolumn{4}{c|}{   \tpic{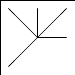}  \tpic{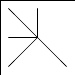} \tpic{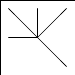} \tpic{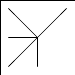}} \Tstrut
\\
 \multicolumn{4}{|l|}{\ \ \ {\bf dec.} \ \ \  2a\ \ \ \ \  2a\ \ \ \ \  4.1a} &
 \multicolumn{4}{l|}{\  {\bf dec.}\   \ \  {\bf dec.} \   \ {\bf dec.} \  \ 3.1a }&
 \multicolumn{4}{l|}{\ \ 2a \ \ \ \ \ 3.1a\ \ \ \   1b\ \ \ \ \  1a} 
\\ \hline

 \multicolumn{4}{|c|}{\tpic{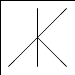} \tpic{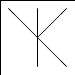}   \tpic{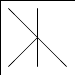}   \tpic{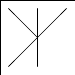}}&
 \tpic{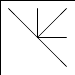}& 
\multicolumn{3}{c|}{ \tpic{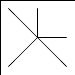}  \tpic{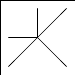} \tpic{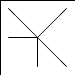}}  & 
\multicolumn{4}{c|}{\tpic{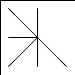}  \tpic{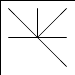}  \tpic{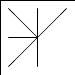}   \tpic{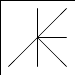}} \Tstrut
\\
 \multicolumn{4}{|l|}{\ \ \ 1b\ \ \ \ \ \  1b\ \ \ \ \  3.1a\ \ \ \ \  2a} & 2a &
 \multicolumn{3}{l|}{\  \   3.2a\ \ \ \   1b \ \ \ \  1a}&
 \multicolumn{4}{l|}{\  3.1a \ \ \ \   2a\ \ \ \ \ \   2a \ \ \ \ \  2b}
\\ \hline
\multicolumn{4}{|c|}{ \tpic{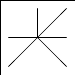}   \tpic{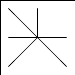} \tpic{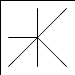}   \tpic{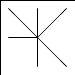}}&
\multicolumn{4}{c|}{ \tpic{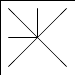}  \tpic{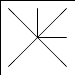}   \tpic{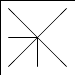}} &
\multicolumn{4}{c|}{ \tpic{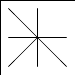}  \tpic{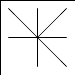} \tpic{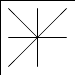}}  \Tstrut
\\
\multicolumn{4}{|l|}{\ \ \ 2b\ \ \ \ \ \  3.2a\ \  \ \  1b \ \ \ \ \ 1b }&
\multicolumn{4}{l|}{\ \ \ \ \ \   1b\ \  \ \ \ \  2a\ \ \ \ \ \ 1a}&
\multicolumn{4}{l|}{\ \ \ \ \ \  3.2a\ \ \ \ \  2a\ \ \ \ \  2a}
\\ \hline
  \end{tabular}
\vskip 3mm
  \caption{The $56$ models with an infinite group. Exactly $9$ are
    decoupled. For the others, we give a label that tells which method can be used to prove that it is not decoupled. These labels refer to the numbering in Section~\ref{sec:infinite}.
 We have put in the
  same cell models that only differ by a symmetry of the
  square.}
  \label{tab:nodec}
\end{table}

\subsection{The invariant lemma}
\label{sec:invariant-lemma}
At this stage, we have found $8$ models ($4$ unweighted, $4$ weighted),
which, as Gessel's model, admit invariants and are decoupled. They
are in fact the $8$ algebraic (or conjecturally algebraic) models of
Figure~\ref{fig:alg_models}. In order to prove their algebraicity as we did in
Section~\ref{sec:gessel} for Gessel's model, we still need to adapt the third
and last ingredient of Section~\ref{sec:gessel}, namely
Lemma~\ref{lem:inv-gessel}.  We can do this for $7$ of our $8$ models. The
resisting model is the reverse Kreweras model, with steps
$\rightarrow, \uparrow, \swarrow$. We shall circumvent this difficulty
in the next subsection.
\begin{Lemma}[\bf{The invariant lemma}]\label{lem:inv} Let $\cS$ be one of the models of
  Figure~\ref{fig:alg_models}, distinct from the reverse Kreweras
  model. If $\cS$ is one of the last two models, let  $X=(1+u)(1+\bu)t$, 
with $\bu=1/u$.  Otherwise, let $X=t+(u+\bu) t^\beta$, where
  $\beta$ is given in Table~\ref{tab:abc}
  below. As in Section~\ref{sec:gessel}, we slightly abuse notation by denoting $Y_0$ and $Y_1$ the roots of $K(X,y)$. Then $Y_0$ and $Y_1$  can be expanded around $t=0$ as    Puiseux series in $t$ with coefficients in $\cs(u)$, starting with
\[ 
Y_0= ut^\gamma (1+ o(1)) \qquad \hbox{and} \qquad Y_1= \bu t^\gamma  (1+
o(1)),
\] 
where $\gamma$ is given by Table~\ref{tab:abc}. The series $Q(X,Y_i)$ and $Q(0,Y_i)$ are
well defined as series in $t$ (or $\sqrt t$ when $\gamma$ is a half-integer) with coefficients in
$\qs(u)$.

Let $A(y)$ be a Laurent series in $t$ with polynomial coefficients in $y$,
of the form
\[ 
A(y)= \sum_{0\leq j \leq \rho n+n_0} a(j,n) y^j t^n,
\] 
where $a(j,n)  \in \qs$ and  $n_0,~\rho$ are  constants such that $\rho<1/\vert \gamma\vert$ if $\gamma<0$. Then $A(Y_0)$
and $A(Y_1)$ are well defined  series in $t$ (or $\sqrt t$) with coefficients in
$\qs(u)$. If they coincide, then $A(y)$ is in fact independent of $y$.
\end{Lemma}
The proof mimics the proofs of  Lemmas~\ref{lem:conv-gessel}
and~\ref{lem:inv-gessel} used in the Gessel case, where we had
$\beta=2$, $\gamma=-1$ and $\rho=1/2$.
  
\medskip

\begin{table}[htb]
  \centering
  \begin{tabular}{|c|c|c|c||c|c|c|c|}
\hline
Model & \begin{tikzpicture}[scale=.3] 
    \draw[->] (0,0) -- (1,1);
    \draw[->] (0,0) -- (-1,0);
    \draw[->] (0,0) -- (0,-1);
        \draw[-] (0,-1) -- (0,-1) node[below] {\phantom{$\scriptstyle 1$}};
  \end{tikzpicture} &      \begin{tikzpicture}[scale=.3] 
      \draw[->] (0,0) -- (0,-1);
    \draw[->] (0,0) -- (1,1);
    \draw[->] (0,0) -- (-1,0);
    \draw[->] (0,0) -- (0,1);
    \draw[->] (0,0) -- (-1,-1);
    \draw[->] (0,0) -- (1,0);
    \draw[-] (0,-1) -- (0,-1) node[below] {\phantom{$\scriptstyle 1$}};
  \end{tikzpicture} &   \begin{tikzpicture}[scale=.3] 
      \draw[->] (0,0) -- (-1,-1);
    \draw[->] (0,0) -- (1,1);
    \draw[->] (0,0) -- (-1,0);
    \draw[->] (0,0) -- (1,1);
    \draw[->] (0,0) -- (1,0);
    \draw[-] (0,-1) -- (0,-1) node[below] {\phantom{$\scriptstyle 1$}};
  \end{tikzpicture} & \begin{tikzpicture}[scale=.3] 
    \draw[->] (0,0) -- (-1,0) node[left] {$\scriptstyle 1$};
    \draw[->] (0,0) -- (-1,-1) node[left] {$\scriptstyle 1$};
    \draw[->] (0,0) -- (0,-1) node[below] {$\scriptstyle \lambda$};
    \draw[->] (0,0) -- (1,-1) node[right] {$\scriptstyle 1$};
    \draw[->] (0,0) -- (1,0) node[right] {$\scriptstyle 2$};
    \draw[->] (0,0) -- (1,1) node[right] {$\scriptstyle 1$};
  \end{tikzpicture} & \begin{tikzpicture}[scale=.3] 
    \draw[->] (0,0) -- (-1,0) node[left] {$\scriptstyle 1$};
    \draw[->] (0,0) -- (-1,1) node[left] {$\scriptstyle 1$};
    \draw[->] (0,0) -- (0,1) node[above] {$\scriptstyle 2$};
    \draw[->] (0,0) -- (1,1) node[right] {$\scriptstyle 1$};
    \draw[->] (0,0) -- (1,0) node[right] {$\scriptstyle 2$};
    \draw[->] (0,0) -- (1,-1) node[right] {$\scriptstyle 1$};
    \draw[->] (0,0) -- (0,-1) node[below] {$\scriptstyle 1$};
  \end{tikzpicture}&
  \begin{tikzpicture}[scale=.3] 
    \draw[->] (0,0) -- (-1,0) node[left] {$\scriptstyle 2$};
    \draw[->] (0,0) -- (-1,1) node[left] {$\scriptstyle 1$};
    \draw[->] (0,0) -- (0,1) node[above] {$\scriptstyle 1$};
    \draw[->] (0,0) -- (-1,-1) node[left] {$\scriptstyle 1$};
    \draw[->] (0,0) -- (1,0) node[right] {$\scriptstyle 1$};
    \draw[->] (0,0) -- (1,-1) node[right] {$\scriptstyle 1$};
    \draw[->] (0,0) -- (0,-1) node[below] {$\scriptstyle 2$};
  \end{tikzpicture} & \begin{tikzpicture}[scale=.3] 
    \draw[->] (0,0) -- (-1,0) node[left] {$\scriptstyle 2$};
    \draw[->] (0,0) -- (-1,1) node[left] {$\scriptstyle 1$};
    \draw[->] (0,0) -- (0,1) node[above] {$\scriptstyle 2$};
    \draw[->] (0,0) -- (1,1) node[right] {$\scriptstyle 1$};
    \draw[->] (0,0) -- (1,0) node[right] {$\scriptstyle 1$};
    \draw[->] (0,0) -- (0,-1) node[below] {$\scriptstyle 1$};
        \draw[->] (0,0) -- (-1,-1) node[left] {$\scriptstyle 1$};
  \end{tikzpicture}  \\
\hline
$\beta $ & 5/2 & 3/2 & 2 & 2 & 3/2 &&\\
$\gamma$ & $-1/2$ & $-1/2$ & $-1$ & $-1$ & $1/2$ & $0$ & $0$ \\
\hline
\end{tabular}
\medskip
 \caption{The values  $\beta$ and $\gamma$ occurring in Lemma~\ref{lem:inv}.}
\label{tab:abc}
 \end{table}

It now remains, for each of the $7$ models to which the above lemma applies,
to construct a series $A(y)$ satisfying the conditions of the
lemma, by combining the invariants of  Tables~\ref{tab:ratinv} and~\ref{tab:ratinv-weighted} and the decoupling
functions of Table~\ref{tab:decoupling_functions-finite}. Applying
Lemma~\ref{lem:inv} gives for each of these $7$ models a polynomial equation of the form 
\beq\label{cateq}
\Pol(Q(0,y), A_1, \ldots, A_k,t,y)=0,
\eeq
where the series $A_i$ are derivatives of $Q(0,y)$ with respect to
$y$, evaluated at $y=0$ or $y=-1$. These equations are made explicit in
Appendix~\ref{app:alg}. The next step will be to solve them, using the
general procedure of~\cite{mbm-jehanne}. This is described in
Section~\ref{sec:effective}, and detailed in Appendix~\ref{app:alg}, but we  delay this description to establish an
equation of the same form for the reverse Kreweras model, to which
Lemma~\ref{lem:inv} does not apply.

\subsection{An alternative to the invariant lemma}
\label{sec:alternative}
The above  method fails for the reverse Kreweras model. The reason
  is that we have no counterpart
of Lemma~\ref{lem:inv}: there exists no Puiseux series $X$  in $t$,
with coefficients in $\cs(u)$,  such that  both roots of  $K(X,y)=0$
  can be substituted for $y$ in $Q(X,y)$ (the proof is elementary, by
considering the valuation of $X$ in $t$).
However, we will now show that this tool is not
essential:
  if instead of using the equations 
  \[
    J(Y_0)=J(Y_1) \qquad \hbox{ and } \qquad xY_0-xY_1=G(Y_0)-G(Y_1),
\]  
we stick to their algebraic origin, namely the fact that 
$
I(x)-J(y)$ and $F(x)+G(y)-xy
$
are both divisible by the kernel $K(x,y)$, then we can still complete
the proof using an algebraic argument that
does not involve substituting $y$ by the series $Y_i$. This
substitution-free approach does work in particular for the reverse Kreweras model.

To clarify this, let us first consider Kreweras' model, with steps
$\nearrow, \leftarrow, \downarrow$, for which the method that we have
described in the previous subsections works. The kernel is
\[ 
K(x,y)=t(x^2y^2+x+y)-xy,
\] 
and the functional equation reads
\beq\label{eq-Kreweras}
K(x,y)Q(x,y)=R(x)+S(y)-xy,
\eeq
with $R(x)=txQ(x,0)=S(x)$.  The invariants and decoupling functions
can be taken as
\[ 
I(x)=\frac t {x^2}-\frac 1 x -tx =J(x)
\] 
and 
\[ 
F(x)=-\frac 1 x + \frac 1 {2t}=G(x).
\] 
 The fact that we can take $I=J$ and $F=G$ comes from the
 $x/y$-symmetry of the model.

How does the method presented so far work?
First, when $x$ is taken to be $t+(u+\bu)t^{5/2}$, both $Y_0$ and $Y_1$ can be
substituted for $y$ in the functional equation, yielding
\beq\label{S-diff}
S(Y_0)-xY_0=S(Y_1)-xY_1.
\eeq
Then, the decoupling function allows us to rewrite this as
\beq\label{SG}
S(Y_0)-G(Y_0)= S(Y_1)-G(Y_1).
\eeq
The invariant $J$ satisfies the same equation as $S-G$:
\[ 
J(Y_0)=J(Y_1).
\] 
We now form a third series $A(y)$ satisfying this equation, but having no pole at $y=0$ (nor $t=0$):
\beq\label{LK-def}
A(y):=t^2\big(S(y)-G(y)\big)^2-tJ(y).
\eeq
By the invariant lemma (Lemma~\ref{lem:inv}), $A(y)$  must be independent of $y$.

We now give a substitution-free version of this argument.
What plays the role of~\eqref{S-diff} is simply the functional
equation~\eqref{eq-Kreweras}. The decoupling property stems from
\[ 
G(x)+G(y)=xy- \frac{K(x,y)}{xyt},
\] 
and allows us to rewrite the functional equation as
\[ 
\big(S(x)-G(x) \big) + \big( S(y)-G(y) \big)= K(x,y) \left(
Q(x,y)+\frac 1{xyt}\right).
\] 
This is the counterpart of~\eqref{SG}. We multiply this equation  by $\big(S(x)-G(x) \big) - \big( S(y)-G(y)
\big)$, which gives:
\[ 
\big(S(x)-G(x) \big)^2 - \big( S(y)-G(y) \big)^2= K(x,y) \left(
Q(x,y)+\frac 1{xyt}\right)\big((S(x)-G(x))-(S(y)-G(y))\big).
\] 
This should be compared to the equation
\[ 
J(x)-J(y)=K(x,y) \frac{y-x}{x^2y^ 2},
\] 
which underlies the invariant property of $J$.
We now derive from the last two equations an equation satisfied by the
pole-free series $A(y)$ defined by~\eqref{LK-def}:
\[ 
A(x)-A(y)=K(x,y) C(x,y),
\] 
with 
\[ 
C(x,y)=t^2 
\left(Q(x,y)+\frac 1{xyt}\right)\big(S(x)-G(x) -  S(y)+G(y)
\big) - t\  \frac{y-x}{x^2y^ 2}.
\] 
Using the expressions for
$G(y)$ and $S(y)$, we observe that $C(x,y)$
is a formal power series in $t$ with coefficients in
$\qs[x,y]$. This series,  multiplied by the polynomial
$K(x,y)$, decouples as $A(x)-A(y)$. The following lemma shows that this is
impossible, unless $C(x,y)=0=A(x)-A(y)$. Thus we conclude that
$A(x)$ is independent of $x$, which was a consequence of the invariant
lemma in our first approach.

\begin{Lemma}\label{lem:generic-inv-lemma} Consider a quadrant model
  and its kernel $K(x,y)$. If there are
  series $A(x),B(y)$ and $C(x,y)$ in $\rs[[x,t]]$, $\rs[[y,t]]$ and
  $\rs[[x,y,t]]$, respectively, such that $A(x)-B(y)=K(x,y)C(x,y)$, then
  $A(x)=B(y)\in\rs[[t]]$ and $C(x,y)=0$.  
\end{Lemma}
\begin{proof}
   We define a total order on monomials  $t^n x^i y^j$, for $(n,i,j)\in \ns^3$, by taking the lexicographic order on $(n,i,j)$. For a series $S$, we denote by $\min(S)$ the smallest monomial occurring in $S$. Then
$\min K(x,y)=xy$. Assume $C(x,y)\not =0$, and let $M$ be its minimal
monomial. Then $xyM$ is the minimal monomial of $K(x,y)C(x,y)$, and
should thus occur in $A(x)-B(y)$, which is impossible.
\end{proof}

\medskip
We now adapt this to the reverse Kreweras model, with steps
$\rightarrow, \uparrow, \swarrow$.  The kernel is
\[ 
K(x,y)=t(1+x^2y+xy^2)-xy,
\] 
and the functional equation reads 
\[ 
K(x,y)Q(x,y)=R(x)+S(y)-S(0)-xy
\] 
where now $R(x)=tQ(x,0)=S(x)$.  The invariants and decoupling functions
can be taken as
\[ 
I(x)= t {x^2}- x -\frac t x =J(x)
\] 
and 
\[ 
F(x)=-\frac{x^2}2 +\frac{x}{2t}-\frac 1 {2x}=G(x).
\] 
 The decoupling property stems from
\[ 
G(x)+G(y)=xy-K(x,y) \frac{x+y}{2xyt},
\] 
and allows us to rewrite the functional equation as
\[  
\left(S(x)-\frac{S(0)}2 -G(x) \right) + \left( S(y)-\frac{S(0)}2-G(y) \right)= K(x,y) \left(
Q(x,y)+ \frac{x+y}{2xyt}\right).
\]  
Once multiplied by $\big(S(x)-G(x)\big) -\big(S(y)-G(y)\big)$, this reads
\begin{multline}\label{eq1}
\left(S(x)-\frac{S(0)}2-G(x) \right)^2 - \left( S(y)-\frac{S(0)}2-G(y)
\right)^2=\\
 K(x,y) \left(
Q(x,y)+ \frac{x+y}{2xyt}\right)\big((S(x)-G(x)) -(S(y)-G(y))\big).
\end{multline}
 This should be compared to the invariant property
\beq\label{eq2}
J(x)-J(y)=K(x,y) \frac{x-y}{xy}.
\eeq
We now cancel poles at $y=0$ (and $t=0$) by considering the series
\[ 
A(y):= 4t^2\left(S(y)-\frac{S(0)}2-G(y) \right)^2 -J(y)^2+2tS(0)J(y).
\] 
A linear combination of~\eqref{eq1} and~\eqref{eq2} gives
\[ 
A(x)-A(y)=K(x,y) C(x,y),
\] 
with 
\[ 
C(x,y)=4t^2\left(
Q(x,y) +\frac{x+y}{2xyt}\right)\left(S(x)-G(x) -S(y)+G(y)\right)
-\frac{x-y}{xy} \left( J(x)+J(y) -2tS(0)\right).
\] 
Using the expressions {for}
$G(y)$ and $S(y)$, we observe that $C(x,y)$ is a series in $t$ with coefficients in
$\qs[x,y]$. We conclude as in Kreweras' case that
$C(x,y)=0=A(x)-A(y)$, so that $A(y)$ is independent of $y$. 

{By expanding
the series $A(y)$ around $y=0$, we obtain an equation of the
form~\eqref{cateq}, as for the $7$ other models of
Figure~\ref{fig:alg_models}}:
\beq\label{eqinv-revK}
 4t^2\left(S(y)-\frac{S(0)}2-G(y) \right)^2 -J(y)^2+2tS(0)J(y)
= {t}^{2}S ( 0 )  ^{2}+4 {t}^{2} S'(0)-4\,t.
\eeq

We have not checked whether this ``substitution-free'' invariant
  lemma works for all models of Figure~\ref{fig:alg_models}.

\subsection{Effective solution of algebraic models}
\label{sec:effective}
At this stage, for each of the eight models of
Figure~\ref{fig:alg_models}, we have obtained an equation of the
form
\begin{equation}
\label{eq:pol_eq}
\Pol(Q(0,y), A_1, \ldots, A_k,t,y)=0,
\end{equation}
where the series $A_i$ are  derivatives of
$Q(0,y)$ with respect to $y$, evaluated at $y=0$ or $y=-1$.  Their exact
forms are given in Appendix~\ref{app:alg}. It remains to apply the general
procedure of~\cite{mbm-jehanne} to solve them. This is also detailed
in the Appendix, and a {\sc Maple} session supporting the calculations
is available on the authors' webpages. These calculations are of course
heavier when {the number} $k$ in~\eqref{eq:pol_eq} is large: the most complicated models turn out to be
Gessel's model and the last weighted model, for which
$k=3$ (we recall that this model was only \emm conjectured, to be
  algebraic~\cite{KaYa-15}). For the reverse and double Kreweras models, and for the third
weighted model, {we have} $k=2$; while for Kreweras' model  and for the
first two weighted models we have $k=1$.

In all cases the solution is (as already claimed) algebraic. In
particular, the \gf\ $Q(0,0)$ of excursions has degree $3$, $3$, $4$, $8$, $6$,
$3$, $3$, $8$ over $\qs(t)$, if we scan models from left to right in
Figure~\ref{fig:alg_models}. It is also worth noting that the
minimal polynomial of $Q(0,0)$ has genus zero (so that the
corresponding curve has a rational parametrization), except for the
last weighted model, which had never been solved so far: 
\[ 
Q(0,0)= \frac{-1-6t+\sqrt{Z}}{2t^2},
\] 
where $Z=1+12t+40t^2+O(t ^3)$ satisfies a quartic equation of genus $1$:
\begin{multline*}
   27\,{Z}^{4}- 
18 \left( 10000\,{t}^{4}+9000\,{t}^{3}+2600\,{t}^{2}+240\,t+1 \right)
{Z}^{2}
\\
+8 \left( 10\,{t}^{2}+6\,t+1 \right) 
 \left( 102500\,{t}^{4}+73500\,{t}^{3}+14650\,{t}^{2}+510\,t-1
 \right) Z\\
= \left( 10000\,{t}^{4}+9000\,{t}^{3}+2600\,{t}^{2}+240\,t+1
 \right) ^{2}.
\end{multline*}

\section{An analytic invariant method}
\label{sec:analysis}
In this section we  move to an analytic world, and consider
$Q(x,y)\equiv Q(x,y;t)$ as a function of three complex variables.
We will use several important results from  the analytic approach of quadrant walks, developed first in a probabilistic
framework~\cite{FIM-99}, and then in an enumerative one~\cite{Ra-12}. For the
reader's convenience, we will recall all relevant statements.

The main result in this section, Theorem~\ref{thm:9models}, tells that, for each decoupled model with an infinite
group (shown in Table~\ref{tab:decoupling_functions-infinite}), the series $Q(0,y)$ has a rational expression in terms of $t$,~$y$, 
and an {explicit} function $w(y)$, which will be seen as a \emm weak, invariant.
So far, only \emm integral,   expressions, also involving $w(y)$, were known.
We establish  some preliminary technical
results in Section~\ref{sec:prel}, introduce $w(y)$ in Section~\ref{subsec:weak_invariants}, and
prove our main theorem  in Section~\ref{subsec:analinv}. The readers
who are familiar with the analytic approach to quadrant walks will recognize
in the weak invariant $w(y)$ the \emm conformal gluing function, that is central
in this approach. We conclude in Section~\ref{sec:weak-finite}
  by showing that this analytic approach applies as well to the four
  decoupled   models with a finite group that we solved in an
  algebraic fashion in the previous section.
%
\subsection{Preliminaries}
\label{sec:prel}

Observing that the coefficients of $Q(x,y;t)$ satisfy 
\[ 
\left\vert\sum_{i,j\ge 0} q(i,j;n) x^i y^j \right\vert\leq \vert\cS\vert^n \max(1, \vert x\vert^n) \max(1,
\vert y\vert^n) ,
\]  
we see that $Q(x,y;t)$ is analytic {in} $\{ \vert t\vert\max(1, \vert x\vert) \max(1,\vert y\vert) <1/\vert\cS\vert\}$
(at least), and that this domain is a neighborhood of  the
polydisc $\{\vert x\vert\leq 1, \vert y\vert\le1,  \vert t\vert<1/\vert\cS\vert\}$.

 The roots
$Y_{0,1}$ of the kernel (now called \emm branches,) are 
\[ 
    Y_{0,1}(x)=\frac{- b(x)\pm\sqrt{ b(x)^2-4 a(x) c(x)}}{2 a(x)},
\] 
where $a, b$ and $c$ are defined in~\eqref{K-abc}.  The discriminant  $ d(x):=
 b(x)^2-4 a(x) c(x)$ has degree three or four,
hence there are four branch points 
$x_1, \ldots, x_4$ 
(depending on $t$), with
$x_4=\infty$ if $d(x)$ has degree three.   We define analogously 
the branches $X_{0,1}(y)$ and  their four branch 
points $y_1,\ldots ,y_4$.
One key difference with the  formal framework adopted so far is the following:
\begin{center}
\emph{In this section, $t$ is a fixed real number in $(0,1/\vert \mathcal S\vert)$.}  
\end{center}
Moreover, we only consider  non-singular, {unweighted} models.

\begin{Lemma}[{\bf{Properties of the branch points~\cite[Sec.\ 3.2]{Ra-12}}}]
\label{lem:branch_points}
The branch points $x_\ell$ are real and distinct. Two
of them (say $x_1$ and $x_2$) are in the open unit disc, with
$ x_1 < x_2$ and $x_2>0$. The other two  (say
$x_3$ and $x_4$) are outside the closed unit disc, with 
$x_3>0$ and  $ x_3 <  x_4$ if $x_4>0$. The discriminant
$d(x)$ is negative on $(x_1,x_2)$ and $(x_3,x_4)$, where if $x_4<0$, 
the set 
$(x_3,x_4)$ stands for the union of intervals $(x_3,\infty)\cup
(-\infty,x_4)$.

Of course, analogous results hold for the branch points $y_\ell$.
\end{Lemma}

Figure~\ref{fig:d} illustrates schematically the two cases $x_4>0$ and
$x_4<0$.

\begin{figure}
  \scalebox{0.5}{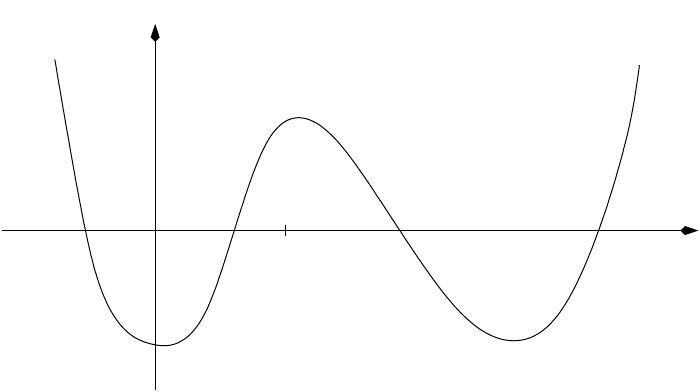}\hskip 10mm \scalebox{0.5}{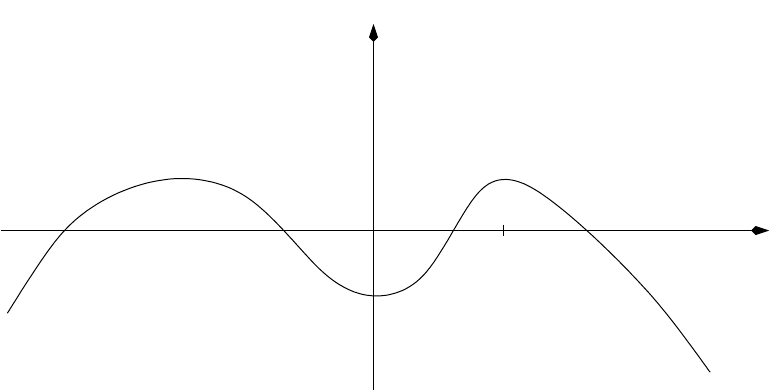}
  \caption{Plot of $d(x)$ for $x$ real: the main two possibilities,
    depending on the sign of $x_4$. Note that $x_1$ may be
    non-negative, and $x_4$ may be $+\infty$.}
  \label{fig:d}
\end{figure}

The branches $Y_{0,1}$ are meromorphic on $\mathbb C\setminus
([x_1,x_2]\cup [x_3,x_4])$. 
On the cuts $[x_1,x_2]$ and  $[x_3,x_4]$,
the two branches $Y_{0,1}$ 
still exist and are complex conjugate (but possibly infinite at
$x_1=0$ as
discussed in the next lemma). {At the branch points $x_i$, we have
$Y_0(x_i)=Y_1(x_i)$ (when finite), and we denote this value by
$Y(x_i)$.} A key object
in our definition of weak invariants 
is the  curve~$\mathcal L$  (depending on~$t$)
defined by
\[ 
     \mathcal L =Y_0([x_1,x_2])\cup Y_1([x_1,x_2])=\{y\in \mathbb C:
     K(x,y)=0 \text{ and } x\in[x_1,x_2]\}. 
\] 
By construction, it is symmetric with respect to the real axis.

We denote by $\mathcal G_\mathcal L$ 
the domain delimited by $\mathcal L$ and avoiding  the real point at
$+\infty$.  See Figure~\ref{fig:curves} for examples. 

\begin{figure}[ht]
\vspace{-10mm}
  \includegraphics[width=7.2cm,height=7.2cm]{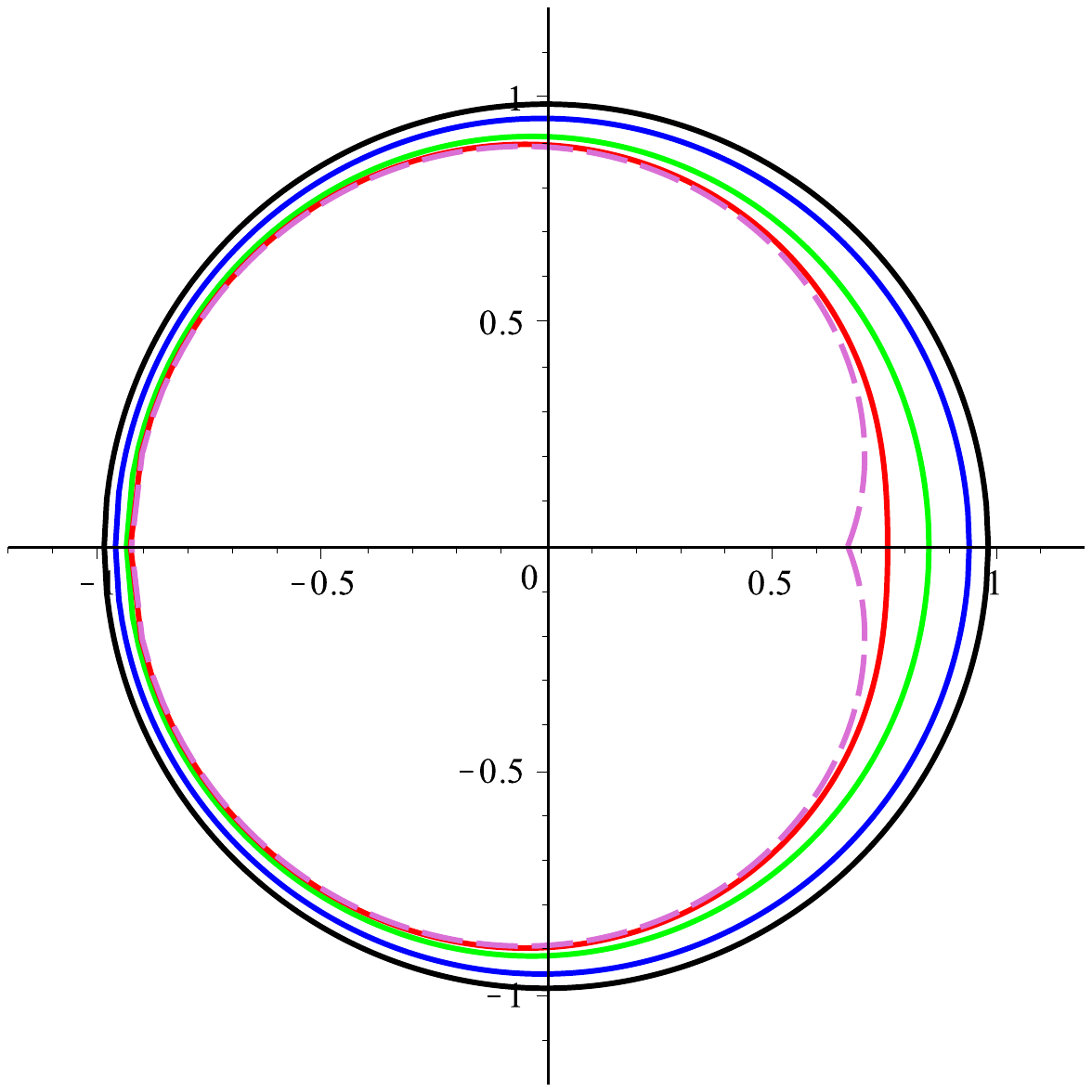} 
 \hspace{-1mm}
\includegraphics[width=7.2cm,height=7.2cm]{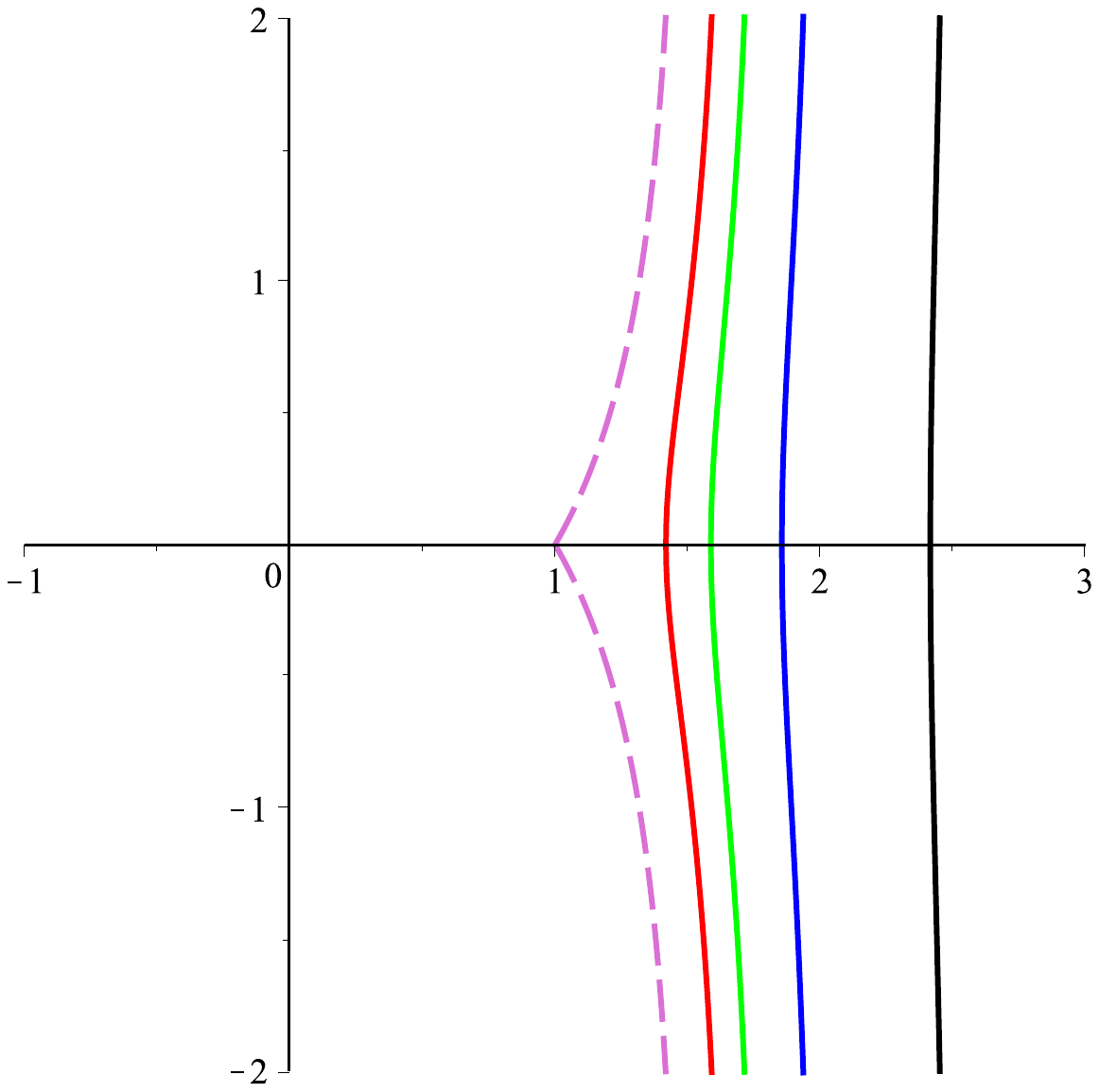}
\vspace{-10mm}
\caption{The curves $\mathcal L$ for model ${\#3}$
of Table~\ref{tab:decoupling_functions-infinite}
  (for $t=0.03$,
$0.1$,
$0.2$, 
 $0.25=1/\vert\cS\vert$ and $0.263185\ldots$
as one moves closer to the origin) and for the reverse Kreweras model
(second model in Table~\ref{tab:decoupling_functions-finite}; from  right to left,
$t=0.2$,
$0.25$,
$0.28$,
$0.3$
and $1/3=1/\vert\cS\vert$). The dashed curve is obtained for a value
$t_c\ge 1/\vert\cS\vert$, where $\cL$ stops being smooth,
but here we only consider values of $t$ less than $1/\vert\cS\vert$.
}
\label{fig:curves}
\end{figure}

The following lemma is proved in~\cite{FIM-99} in the probabilistic
  framework, that is,  when $t=1/\vert \mathcal S\vert$ (see in
  particular Thm.~5.3.3 and its proof). We show that it 
  holds for all $t\in(0,1/\vert \mathcal S\vert)$ as well.

\begin{Lemma}[\bf{Properties of  the curve $\cL$}]
\label{Lem:curve_L_0_infty}
 The curve $\cL$ is symmetric with respect to 
 the real axis. It intersects this
 axis at $Y(x_2)>0$. 

If $\cL$ is unbounded, $Y(x_2)$ is the only
 intersection point. This occurs if  and only if neither  $(-1,1)$ nor $(-1,0)$ belong to $\cS$. In this case,
$x_1=0$ and the only point of $[x_1, x_2]$ {where at least
one branch $Y_i(x)$ is infinite is $x_1$ (and then both branches are
infinite there). }

Otherwise, the curve $\cL$ goes through a second real point, namely
$Y(x_1) \leq 0$. The limit case  $Y(x_1)=0$ occurs if and only if neither
$(-1,-1)$ nor $(-1,0)$ belong to  $\cS$. In this case,
$x_1=0$. 

Consequently, the point $0$ is either in the domain  $\mathcal G_\cL$ or on the curve $ \cL$.  The domain  $\mathcal G_\cL$ also
contains the (real)
branch points $y_1$ and $y_2$, of modulus less than $1$. The other two
branch points, $y_3$ and $y_4$, are in the complement of $\cG_\cL \cup
\cL$. The domain  $\mathcal G_\cL$ coincides with the region denoted $\mathcal G
Y([x_1(t),  x_2(t)];t)$ in~\cite[Lem.~2]{Ra-12}.
\end{Lemma}
\begin{proof}
Since $d(x)<0$ in $(x_1, x_2)$, the curve $\cL$ intersects the real
axis at two points at most, namely $Y(x_1)$ and $Y(x_2)$.
 Recall from~\eqref{K-abc} that
\[ 
a(x)=tx\sum_{(i,1)\in \cS}x^i, \qquad b(x)=tx\sum_{(i,0)\in \cS}x^i-x \qquad\hbox{while} \qquad
c(x)=tx\sum_{(i,-1)\in \cS}x^i.
\] 
We begin with the polynomial $a(x)$, which is (at most) quadratic.  If
$a(x)=0$ for some real $x$, then $d(x)=b(x)^2\ge 0$, hence the sign of $a(x)$ is
constant on  the interval $(x_1,x_2)$. Since $a(x_2)>0$ (because
$x_2>0$,  see Lemma~\ref{lem:branch_points}), we also have $a(x_1)\ge
0$.

Now consider the polynomial $b(x)$, which is also quadratic at most. We
have $b(0)\ge 0$ and $b(1)<0$ (by our choice of $t$), hence $b(x)$ has 
one root $x_b$ in $[0,1)$: exactly one, since if $b(x)$ is quadratic,
  it must have a root larger than $1$ because $t>0$. Moreover,
  $d(x_b)=-4a(x_b)c(x_b)\leq 0$, hence $x_b\in [x_1, x_2]$. In fact $x_b\in[x_1,x_2)$
since $x_2$ is positive and thus satisfies
$0<4a(x_2)c(x_2)=b(x_2)^2$.
Since $x_b< x_2<1$, we have $b(x_2)<0$ hence
$Y(x_2)=-b(x_2)/(2a(x_2))>0$.
Similarly, since $x_1\leq x_b$, we have $b(x_1)\ge 0$. 

If $a(x_1)=0$, the condition $d(x_1)=0$ implies that $b(x_1)=0$ as
well. Hence $x_1$ coincides with $x_b$, which is non-negative; but
$a(x_b)=0$ then forces $x_b=0$. Thus $b(0)=0=a(0)$, which is
equivalent to saying that neither $(-1,0)$ nor $(-1,1)$ belong to
$\cS$. It is readily checked that in this case each $Y_i(x)$ tends to
infinity as $x\rightarrow 0^+$.

Now assume $a(x_1)>0$. Then $Y(x_1)=-b(x_1)/(2a(x_1))\leq 0$. The limit
case $Y(x_1)=0$ occurs when $b(x_1)=0$ and $c(x_1)=0$ (since
$d(x_1)=0$). Hence $x_1$ coincides again with $x_b$, which is
non-negative, and the condition $c(x_b)=0$ forces $x_b=0$. Thus
$b(0)=0=c(0)$,  which is
equivalent to saying that neither $(-1,0)$ nor $(-1,-1)$ belong to
$\cS$. It is readily checked that in this case $Y(0)=0$ indeed.

\medskip
It follows from the results established so far that the intersection
of the domain $\cG_\cL$ with the real axis is $(Y(x_1),Y(x_2))$, where by
convention $Y(x_1)=-\infty$ if $\cL$ is unbounded. Moreover, either
$Y(x_1)=0$ and thus $0\in\cL$, or $0\in (Y(x_1),Y(x_2))$. We now want to prove
that $Y(x_1)<y_1<y_2<Y(x_2)<y_3$, and $y_4<Y(x_1)$ if $y_4<0$. Let us begin with $Y(x_1)<y_1$, assuming
$Y(x_1)$ is finite (otherwise there is nothing to prove). Define $\widetilde d(y)$ as the counterpart for the
variable~$y$ of the discriminant $d(x)$, that is, $\widetilde d(y)= \widetilde b(y)^2 -4 \widetilde a(y) \widetilde c(y).$  We observe
that $\widetilde d(Y(x_1))\ge 0$: otherwise, the roots of $K(x,Y(x_1))$
would be complex conjugate or infinite, while one of them is $x=x_1$. Hence
$Y(x_1)$ cannot be in any of the intervals $(y_1, y_2)$ or $
(y_3,y_4)$. Since it is non-positive, as proved above, it is
necessarily less than or equal to $y_1$, and larger than or equal to $y_4$ if
$y_4<0$.

Similarly, $Y(x_2)$ cannot be in any of the intervals $(y_1, y_2)$ or $
(y_3,y_4)$. Since it is positive (as proved above), it is either
larger than or equal to $y_2$, or in $(0, y_1]$. It remains to exclude
the two cases $0<Y(x_2)\leq y_1$ and $0<y_4\leq Y(x_2)$. 

If $0<Y(x_2)\leq y_1$ then each function $X_i$ is continuous on the
interval $[Y(x_2), y_1]$. Let $X_i$ be the branch of $X$ satisfying
$X_i(Y(x_2))=x_2>0$. Since $X_i(y_1)\leq 0$, there exists  a real
number $y\in (Y(x_2),y_1]$, hence necessarily positive, such that
$X_i(y)=0$. That is, $K(0,y)=0=\widetilde c(y)$, which is impossible for $y$
positive.

The argument excluding the case $0<y_4\leq Y(x_2)$ is similar: in fact, replacing
the step set $\cS$ by $\overline \cS:=\{(i,-j): (i,j)\in \cS\}$ leaves
the $x_\ell$'s unchanged, replaces the set $\{y_\ell: 1\leq i \leq 4\}$ by
$\{1/y_\ell: 1\leq i \leq 4\}$, and finally replaces $Y_i$ by $1/Y_i$. With
these remarks at hand, one  realizes that if $0<y_4\leq Y(x_2)$ for
one model, then $0<Y(x_2)\leq y_1$ for the reflected one.

We still have to exclude the limit cases where $Y(x_\ell)$ would be one of the branch
points $y_i$. This would mean that the system 
$K(x,y)=d(x)=\widetilde d(y)=0$
has a solution. Writing $K(x,y)$ as in~\eqref{K-w}, and 
eliminating $x$ and $y$ between these three
equations, 
we obtain a polynomial in~$t$ and the weights $w_{i,j}$ that
must vanish. One can check that among the $79$ unweighted models
($w_{i,j}\in \{0,1\}$), those
that cancel this polynomial are exactly the $5$ singular models.

Finally, since $\cG_\cL$ contains $y_1$, it must coincide with the
component of $\cs\setminus \cL$ denoted $\mathcal G
Y([x_1(t),  \break  x_2(t)];t)$ in~\cite[Lem.~2]{Ra-12}.
\end{proof}

Among the models having  decoupling functions (Tables~\ref{tab:decoupling_functions-finite}
and~\ref{tab:decoupling_functions-infinite}), the only one for which $\cL$ goes through the point $0$
is model $\# 9$ in Table~\ref{tab:decoupling_functions-infinite}. The
only one for which  $\cL$ is unbounded is the reverse Kreweras model
(second model in Table~\ref{tab:decoupling_functions-finite}). In
fact, the method that we are going to present in this section to solve
models having a decoupling function is more elegant  when $\cL$ is
bounded: this is why three models in
Table~\ref{tab:decoupling_functions-infinite} differ from  the
original  classification of~\cite{BMM-10} by an $x/y$-symmetry (Figure~\ref{fig:switch_models}). We will still
illustrate {by} 
the case of reverse Kreweras walks what can be done when
$\cL$ is unbounded. Note that the  condition for unboundedness is that
$K(0,y)$ has no root (and then it equals $t$).
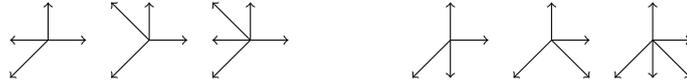
\begin{figure}[ht!]
\begin{center}
\begin{tikzpicture}[scale=.5] 
    \draw[->] (0,0) -- (-1,0);
    \draw[->] (0,0) -- (0,1);
    \draw[->] (0,0) -- (1,0);
    \draw[->] (0,0) -- (-1,-1);
  \end{tikzpicture}\hspace{2mm}
 \begin{tikzpicture}[scale=.5] 
    \draw[->] (0,0) -- (-1,1);
    \draw[->] (0,0) -- (-1,-1);
    \draw[->] (0,0) -- (1,0);
    \draw[->] (0,0) -- (0,1);
  \end{tikzpicture}\hspace{2mm}
  \begin{tikzpicture}[scale=.5] 
    \draw[->] (0,0) -- (1,0);
    \draw[->] (0,0) -- (0,1);
    \draw[->] (0,0) -- (-1,1);
    \draw[->] (0,0) -- (-1,0);
    \draw[->] (0,0) -- (-1,-1);
  \end{tikzpicture}\hspace{15mm}
 \begin{tikzpicture}[scale=.5] 
    \draw[->] (0,0) -- (0,-1);
    \draw[->] (0,0) -- (0,1);
    \draw[->] (0,0) -- (1,0);
    \draw[->] (0,0) -- (-1,-1);
  \end{tikzpicture}\hspace{2mm}
 \begin{tikzpicture}[scale=.5] 
    \draw[->] (0,0) -- (1,-1);
    \draw[->] (0,0) -- (-1,-1);
    \draw[->] (0,0) -- (1,0);
    \draw[->] (0,0) -- (0,1);
  \end{tikzpicture}\hspace{2mm}
  \begin{tikzpicture}[scale=.5] 
    \draw[->] (0,0) -- (1,0);
    \draw[->] (0,0) -- (0,1);
    \draw[->] (0,0) -- (1,-1);
    \draw[->] (0,0) -- (0,-1);
    \draw[->] (0,0) -- (-1,-1);
  \end{tikzpicture}  \end{center}
 \vskip -2mm  
\caption{Models \#1, \#2 and \#7 (left) are symmetric versions of
  models found in the original classification of~\cite{BMM-10} (right).}
  \label{fig:switch_models}
\end{figure}

We now turn to the properties of the function $S(y)=K(0,y)Q(0,y)$. It
is originally defined around $y=0$, and analytic (at least) in the
unit disc $\cD$. This disc contains the points $y_1$ and~$y_2$, and thus intersects
the domain $\mathcal G_\cL$ by Lemma~\ref{Lem:curve_L_0_infty}.

\begin{Proposition}[\bf{The function $\boldsymbol{S(y)}$}]
\label{prop:S-properties}
  The function $S(y)=K(0,y)Q(0,y)$ has an analytic continuation 
in $\cD \cup \mathcal G_\mathcal L$, with finite limits on
  $\cL$. Moreover,  for $x \in [x_1, x_2]\subset(-1,1)$ and $i\in \{0,1\}$, 
\beq
\label{eqker:xYi}
R(x)+S(Y_i)=xY_i+R(0).
\eeq
The function $S(y)$ is bounded on $\mathcal G_\mathcal L\cup\mathcal L$. 
\end{Proposition}
Note that it follows from~\eqref{eqker:xYi}  that for those values of $x$,
 \beq\label{eqker:Y01}
S(Y_0)-xY_0=S(Y_1)-xY_1,
\eeq
an identity that will be combined {with} the properties of decoupling functions.
Observe that~\eqref{eqker:xYi} and~\eqref{eqker:Y01} are analytic versions of~\eqref{eq:func_spec} and~\eqref{SYi}, respectively. They hold for any model, while their formal counterparts~\eqref{eq:func_spec} and~\eqref{SYi} require formal convergence properties.

\begin{proof}
   The first point (analyticity) is  Theorem~5
  in~\cite{Ra-12}. In order to prove the other statements, we need a
  more complete picture of the properties of $R$ and $S$, which can be
  found in~\cite{Ra-12}.

Let us define the curve $\cM$ as the counterpart of $\cL$ for the branches
$X_i$: that is, $\cM=X_0([y_1,y_2]) \cup
X_1([y_1,y_2])$. Define the domain $\cG_\cM$ as the counterpart of
$\cG_\cL$. Let $X_0$ be the branch of $X$ satisfying $\vert X_0(y)\vert\leq \vert
X_1(y)\vert$ for all $y\in \cs$ (see~\cite[Lem.~1]{Ra-12}), and
define $Y_0(x)$ analogously. Then  $X_0$ is a conformal map from $\cG_\cL
\setminus[y_1,y_2]$ to $\cG_\cM \setminus[x_1,x_2]$, with inverse
$Y_0$ (see~\cite[Lem.~3(ii)]{Ra-12}).

Moreover,   it is shown in the proof of~\cite[Thm.~5]{Ra-12} that 
\begin{itemize}
\item $R$ has an analytic continuation on the domain $\cD \cup
  \cG_\cM$, which is included in $\cD \cup \{ x: |Y_0(x)| <1\}$, 
\item analogously, $S$ has an analytic continuation on the domain $\cD \cup \cG_\cL$,
which is included in $\cD \cup \{ y: |X_0(y)| <1\}$, 
\item  with these continuations, the following identity holds on $y\in \cD \cup \cG_\cL$:
\beq\label{e2}
R(X_0)+S(y)=X_0y+R(0).
\eeq
\end{itemize}

With these results at hand, let us now prove that $S$ has finite
limits on $\cL$. Take $y_0\in \cL$. Then $y_0=Y_i(x_0)$ for some $i
\in \{0,1\}$ and $x_0\in
[x_1, x_2]$.  Let $y$ tend to $y_0$ in
$\cG_\cL$. We can write $y=Y_0(x)$, where $x\in \cG_\cM$ tends to $x_0$. Given that $X_0$ and $Y_0$ are inverse maps,~\eqref{e2} reads
\[ 
S(y)=xy+R(0)-R(x),
\] 
so that, as $x$ tends to $x_0$ and $y$ to $y_0$,
\[ 
S(y) \rightarrow x_0 y_0+R(0)-R(x_0),
\] 
by continuity of $R(x)$ in $\cD$. Hence $S$ has finite limits on
$\cL$. Denoting the right-hand side by $S(y_0)$, 
this also establishes~\eqref{eqker:xYi}, since we can take for $y_0$ any
$Y_i(x_0)$ with $x_0\in [x_1, x_2]$.

It remains to prove that $S$ is bounded on $\cG_\cL \cup \cL$. If
$\cG_\cL$ is bounded, there is nothing more to prove. Otherwise, we
know from Lemma~\ref{Lem:curve_L_0_infty} that neither $(-1,1)$ nor $(-1,0)$ are in
$\cS$. Then $(-1,-1)$ and $(0,1)$ must be in $\cS$, and it is easy to check that one of the branches $X(y)$ is
asymptotic to $-1/y^2$ as $y\rightarrow -\infty$, while the other
 tends either to a non-zero constant, or to infinity. Since $X_0$ is
defined to be the ``small'' branch, we conclude that $X_0(y)\sim -1/y^2$ at
infinity. Returning to~\eqref{e2}, this implies that $S(y)$ tends to
$0$ as $y$ tends to infinity in $\cG_\cL$, and completes the proof of
the proposition.
\end{proof}

\subsection{Weak invariants}
\label{subsec:weak_invariants}

\begin{Definition}
\label{def:weak-inv} 
 A function $I(y)$ 
is a \emph{weak invariant} of a quadrant model
$\cS$ if: 
\begin{itemize}
\item it is meromorphic in the domain $\mathcal G_\mathcal L$,
and  admits finite limit values on the curve $\mathcal L$,
\item for any $y\in\mathcal L$, we have $I(y)=I(\overline{y})$,
\end{itemize}
where now the bar denotes the complex conjugate.
\end{Definition}

The second condition also reads $I(Y_0)=I(Y_1)$ for
$x\in[x_1,x_2]$,  because two conjugate points $y$ and $\overline{y}$ of the
curve $\mathcal L$ are the (complex conjugate) roots of $K(x,y)=0$, for
some $x\in[x_1,x_2]$. This condition is thus indeed a weak form of the invariant
condition of 
Lemma~\ref{lem:xy-inv}. Hence, if the model admits a rational invariant $I(y)$ in the
sense of Lemma~\ref{lem:xy-inv}, having no pole on $\cL$, then $I(y)$ is also a weak
invariant.
However, the above definition is less demanding,  and it
turns out that every non-singular quadrant model admits a 
(non-trivial)
weak invariant, which we now describe.

This invariant, traditionally denoted $w(y)$ 
%
in the analytic approach
to quadrant problems~\cite{FIM-99,Ra-12}, is in addition injective in  $\mathcal
G_\mathcal L$. In analytic terms, this third condition makes it a \emm
conformal gluing function, for the domain $\mathcal G_\mathcal
L$. Explicit expressions for
conformal gluing functions are known in a
number of cases (when the domain is an ellipse, a polygon, etc.).  In
our case the bounding curve $\mathcal L$ is a quartic
curve,
and  $w$ can be expressed in terms of Weierstrass' elliptic
functions~(see~\cite[Sec.~5.5.2.1]{FIM-99} or~\cite[Thm.~6]{Ra-12};
note that in our paper we exchange the roles played by $x$ and $y$ in
these two references): 
\beq
\label{eq:expression_gluing}
     w(y;t)\equiv
     w(y)=\wp_{1,3}\Big(-\frac{\omega_1+\omega_2}{2}+\wp_{1,2}^{-1}(f(y))\Big), 
\eeq
where the various ingredients of this expression
are as follows. First, $f(y)$ is a
simple rational function of $y$ whose
coefficients are algebraic functions of  $t$:
\beq\label{eq:def_f}
     f(y) = \left\{\begin{array}{ll}
     \displaystyle \frac{\widetilde d''(y_4)}{6}+\frac{\widetilde d'(y_4)}{y-y_4} & \text{if } y_4\neq \infty,\medskip\\
     \displaystyle\frac{\widetilde d''(0)}{6}+\frac{\widetilde d'''(0)y}{6}& \text{if } y_4=\infty,
     \end{array}\right.
\eeq 
where the $y_\ell$'s are the branch points of the functions
  $X_{0,1}$, and $\widetilde
d(y)= \widetilde b(y)^2 -4 \widetilde a(y) \widetilde c(y)$ as
before.

The next ingredient is  Weierstrass' elliptic function $\wp$,
with  periods $\omega_1$ and $\omega_2$:
\beq\label{w-def}
\wp(z)\equiv \wp(z, \om_1, \om_2) = \frac 1 {z^2} +\sum_{(i,j) \in \zs^2
  \setminus\{(0,0)\}} \left( \frac 1
  {(z-i\om_1-j\om_2)^2}-\frac1{(i\om_1+j\om_2)^2}\right).
\eeq
Then $\wp_{1,2}(z)$ (resp.\ $\wp_{1,3}(z)$) is the Weierstrass function
with periods $\om_1$ and $\om_2$ (resp.\ $\om_1$ and $\om_3$) defined by:
\beq
\label{eq:expression_periods}
   \omega_1 = i\int_{y_1}^{y_2} \frac{\text{d} y}{\sqrt{- \widetilde d(y)}},\qquad
     \omega_2 = \int_{y_2}^{y_3} \frac{\text{d} y}{\sqrt{ \widetilde d(y)}},\qquad
     \omega_3 = \int_{Y(x_1)}^{y_1} \frac{\text{d} y}{\sqrt{ \widetilde d(y)}}.
\eeq
These definitions make sense due to
the properties  of the $y_i$'s and $Y(x_i)$'s  (see
Lemmas~\ref{lem:branch_points} and~\ref{Lem:curve_L_0_infty}). 
If $Y(x_1)$ is infinite (which happens if
  and only if neither $(-1,0)$ nor $(-1,1)$ are in $\cS$), the
  integral defining $\om_3$ starts at $-\infty$.
Note that $\omega_1\in i\mathbb R_+$ and $\omega_2,\omega_3\in \mathbb R_+$.

 Finally, as the Weierstrass function is not
injective on $\mathbb C$, we need to clarify our definition of
$\wp_{1,2}^{-1}$ in~\eqref{eq:expression_gluing}. The function
$\wp_{1,2}$ is two-to-one on the fundamental parallelogram
$[0,\omega_1)+[0,\omega_2)$ (because $\wp(z)=\wp(-z+\om_1+\om_2)$),
but is one-to-one when restricted to a half-parallelogram ---
more precisely,  when restricted to the open
  rectangle $(0,\om_1)+(0,
  \om_2/2)$ together with the three boundary segments  $[0, \om_1/2]$,
   $[0, \om_2/2]$ and $\om_2/2+[0, \om_1/2]$. 
  We choose the
determination of $\wp_{1,2}^{-1}$ in this set.

\begin{Proposition}[\bf{The function $\boldsymbol{w(y)}$}]
\label{Prop:properties_w}
The function $w$ defined by~\eqref{eq:expression_gluing} is
a weak invariant, in the sense of Definition~\ref{def:weak-inv}. It is moreover
injective on $\mathcal G_\mathcal L$, and has in this domain  a unique
(and simple) pole, located at $y_2$.
The function $w$ admits a meromorphic continuation on $\mathbb
C\setminus [y_3,y_4]$.
\end{Proposition}
\begin{proof} See Theorem~6 and Remark~7 in~\cite{Ra-12}.
\end{proof}

In fact, $w(y)$ is a rational function of $y$ if $\cS$ is one of the $23$ models with a
finite group, except for the $4$ algebraic models (Figure~\ref{fig:alg_models},
left), where it is algebraic (see~\cite[Thm.~2
and Thm.~3]{Ra-12}). We refer to Section~\ref{sec:rational_weak} for a further discussion of
the connection between the weak invariant $w$ and the rational
invariant $J$ in the finite group case. 
In the infinite group case, $w(y)$ is not algebraic, nor even
D-finite w.r.t.\ to $y$, see~\cite[Thm.~2]{Ra-12}. However, we will
prove in Theorem~\ref{prop:w-DA} 
that it is D-algebraic  in $y$, and also in  $t$.

\subsection{The analytic  invariant lemma --- Application to quadrant walks}
\label{subsec:analinv}
We now state an  analytic counterpart of Lemma~\ref{lem:inv-gessel}, which applies to the weak invariants of Definition~\ref{def:weak-inv}.
\begin{Lemma}[\bf{The analytic invariant lemma}]\label{lem:inv-analytic}
  Let $\cS$ be a non-singular quadrant model and $A(y)$ a weak
  invariant for this model. If $A$ has no pole in $\mathcal G_\mathcal
  L$ (and, in the case of a non-bounded curve $\mathcal L$, if $A$ is
  in addition bounded at $\infty$), then it is independent of $y$. 
\end{Lemma}
\begin{proof}
 This is proved in~\cite[Ch.~3]{Li-00},
in Lemma~1 (resp.~Lemma~2) for
 the bounded (resp.\ unbounded) case.
\end{proof}

Our main result tells that, for each decoupled model with an infinite
group (Table~\ref{tab:decoupling_functions-infinite}), the series $Q(0,y)$ has a rational expression in terms of $t,
y$, the function $w(y)$ and some of its specializations. Moreover,
this expression is uniform for the first $8$ models of the table. The $9$th one stands
apart, and this is related to the fact, noted after Lemma~\ref{Lem:curve_L_0_infty}, that
the curve $\cL$ contains the point $0$ in this case; equivalently, $K(0,y)=ty^2$.

\renewcommand\Tstrut{\rule{0pt}{2.5ex}}         

\renewcommand\Bstrut{\rule[-1.4ex]{0pt}{2.0ex}}
\begin{table}[h]
   \begin{tabular}{c|ccccccccc}
model     &\#1&\#2&\#3&\#4&\#5&\#6&\#7&\#8&\#9\\
          &\begin{tikzpicture}[scale=.45] 
    \draw[->] (0,0) -- (0,1);
    \draw[->] (0,0) -- (1,0);
    \draw[->] (0,0) -- (-1,0);
    \draw[->] (0,0) -- (-1,-1);
  \end{tikzpicture}& \begin{tikzpicture}[scale=.45] 
    \draw[->] (0,0) -- (0,1);
    \draw[->] (0,0) -- (1,0);
    \draw[->] (0,0) -- (-1,1);
    \draw[->] (0,0) -- (-1,-1);
  \end{tikzpicture} &     \begin{tikzpicture}[scale=.45] 
    \draw[->] (0,0) -- (0,1);
    \draw[->] (0,0) -- (1,1);
    \draw[->] (0,0) -- (0,-1);
    \draw[->] (0,0) -- (-1,0);
  \end{tikzpicture}&\begin{tikzpicture}[scale=.45] 
    \draw[->] (0,0) -- (0,1);
    \draw[->] (0,0) -- (1,0);
    \draw[->] (0,0) -- (1,-1);
    \draw[->] (0,0) -- (-1,0);
  \end{tikzpicture}&\begin{tikzpicture}[scale=.45] 
    \draw[->] (0,0) -- (0,1);
    \draw[->] (0,0) -- (1,0);
    \draw[->] (0,0) -- (1,1);
    \draw[->] (0,0) -- (-1,-1);
    \draw[->] (0,0) -- (-1,0);
  \end{tikzpicture}&\begin{tikzpicture}[scale=.45] 
    \draw[->] (0,0) -- (0,1);
    \draw[->] (0,0) -- (0,-1);
    \draw[->] (0,0) -- (1,1);
    \draw[->] (0,0) -- (-1,-1);
    \draw[->] (0,0) -- (-1,0);
  \end{tikzpicture} & \begin{tikzpicture}[scale=.45] 
    \draw[->] (0,0) -- (-1,1);
    \draw[->] (0,0) -- (-1,0);
    \draw[->] (0,0) -- (1,0);
    \draw[->] (0,0) -- (-1,-1);
    \draw[->] (0,0) -- (0,1);
  \end{tikzpicture}& \begin{tikzpicture}[scale=.45] 
    \draw[->] (0,0) -- (1,1);
    \draw[->] (0,0) -- (0,-1);
    \draw[->] (0,0) -- (1,0);
    \draw[->] (0,0) -- (-1,0);
    \draw[->] (0,0) -- (0,1);
  \end{tikzpicture}& \begin{tikzpicture}[scale=.45] 
    \draw[->] (0,0) -- (1,0);
    \draw[->] (0,0) -- (0,-1);
    \draw[->] (0,0) -- (0,1);
    \draw[->] (0,0) -- (-1,1);
    \draw[->] (0,0) -- (1,-1);
  \end{tikzpicture} \\
\hline
$p$& 0&0&0&0& $-1$& 0&0&0&0 \Tstrut \Bstrut\\
\hline
$r$& $-1$ & $-1$ &$-1$ &1 &$-\frac{1+t}t$ &$-1$ &$-1$ &$-1$ & $1/t$
\Tstrut\Bstrut
\\
\hline
$\alpha$ & $-1$ & $\pm i$ & 0 &0& $-1$ & $-1$ & $j, j^2$ & 0 & 0
\Tstrut\Bstrut\\
\hline
$g_0$ & 1& 0&0&0& 1&1&1&0&0 \Tstrut
  \end{tabular}
\vskip 3mm
 \caption{Values of $p$, $r$, $\alpha$ and $g_0$ in
   Theorem~\ref{thm:9models}. We denote  $j= e^{2i \pi/3}$.}
\label{tab:prg}
\end{table}

\begin{Theorem}
\label{thm:9models}
Let $\cS$ be one of the first $8$ models of
Table~\ref{tab:decoupling_functions-infinite}, with
associated decoupling function $G(y)$.
Let $p$ be the unique pole of $G$, and let $r$ be the residue of $G(y)$ at
$y=p$. Finally, let $\alpha$ be any root of $K(0,y)$, and let $g_0$ be
the constant term of $G(y)$ in its expansion around $y=\alpha$. 
Then the pole $p$ belongs to the domain $\mathcal G_\cL$, the point $\alpha$ belongs to
$\cG_\cL\cup \cL$ and for $y \in \mathcal G_\cL$,  the series $S(y)=K(0,y)Q(0,y)$ is
given by:
\beq\label{S-uniform}
S(y)=G(y) -g_0-\frac{r w'(p) }{w(y)-w(p)} +
\begin{cases}
\displaystyle  \frac{rw'(p)}{w(\alpha)-w(p)}& \hbox{if } \alpha\not=p 
\quad\hspace{2.2mm} (\hbox{models } 1, 2, 6, 7),
\\
\displaystyle - \frac r 2 \frac{w''(p)}{w'(p)} & \hbox{otherwise}\quad
(\hbox{models }
    3, 4, 5, 8),\\
\end{cases}
\eeq
where $w(y)$ is the weak invariant defined by~\eqref{eq:expression_gluing}.
The values of $p$, $r$, $\alpha$  and $g_0$ are made explicit in Table~\ref{tab:prg}.
For instance, for models  \#1 and \#6, which have decoupling function
$G(y)=-1/y$, 
\[ 
  S(y)=-\frac{1}{y}-1+ \frac{w'(0)}{w(y)-w(0)}-\frac{w'(0)}{w(-1)-w(0)},
\] 
while for models \#3 and \#8, with decoupling
function $G(y)=-y-1/y$, one has:
\[ 
     S(y)=-\frac{1}{y}-y+ \frac{w'(0)}{w(y)-w(0)}+\frac{1}{2}\frac{w''(0)}{w'(0)}.
\] 

For the 9th model of Table~\ref{tab:decoupling_functions-infinite},
\beq\label{ninth}
   S(y)=G(y)+ \frac{1}{2}\frac{w''(0)}{w(y)-w(0)}
   +\frac{1}{12}\frac{w^{(4)}(0)}{w''(0)}-\frac{1}{t^2},
\eeq
where $G(y)$ is given in Table~\ref{tab:decoupling_functions-infinite}, and $w(y)$ is defined by~\eqref{eq:expression_gluing}.
\end{Theorem}

\noindent{\bf Remarks}\\
{\bf 1.} The expression of
$S(y)$ in terms of $w(y)$ is the same in cases \#1 and \#6, and in
cases \#3 and \#8 as well. But of course the value of $w(y)$ (given by~\eqref{eq:expression_gluing}) depends
on the details of the model.
\\
{\bf 2.} For models \#2 and \#7, Table~\ref{tab:prg} shows that we have a choice for the
value of $\alpha$: for the first model $\alpha=\pm i$, for the second
one $\alpha=j$ or $j^2$.  But one easily checks that $g_0$ is the same
(namely~$0$, resp.\ $1$) for both choices of $\alpha$, and moreover $w$ takes the
same value at both points~$\alpha$. This comes from the fact that the
two possible values of $\alpha$ are the (complex conjugate) values $Y_0(0)$ and $Y_1(0)$, and that $w(y)$ is an invariant. Hence
both choices of $\alpha$ give the same expression {for} 
$S(y)$.
\\
{\bf 3.} The theorem states that our expressions {for}
$S(y)$ are valid in
$\cG_\cL$. But combined with Proposition~\ref{Prop:properties_w}, they imply
that $S$ can be meromorphically continued to the whole of $\mathbb
C\setminus [y_3,y_4]$.\\
{\bf 4.} The above expressions for 
$S(y)$ differ  from those
obtained in the past using the complex analytic approach of~\cite{Ra-12} by the fact
that they do not involve any integration. This opens the way to
D-algebraicity, as proved in the next section. The connection between
our expressions and the earlier ones is discussed  further in Section~\ref{sec:integrals}.

\medskip
We begin with a separate lemma dealing with the location of
$\alpha$. The case $\alpha=0$ being already addressed in Lemma~\ref{Lem:curve_L_0_infty},
we focus here on the other cases.

\begin{Lemma}
\label{Lem:location_particular_points_L}
For models \#1, \#5 and \#6, the point $\alpha=-1$ belongs to the domain
$\mathcal G_\mathcal L$. For model \#2 (resp.\ \#7), the points
$\alpha=\pm i$
(resp.\ $\alpha=j,j^2$) are located on the curve $\mathcal L$,
and equal to $Y_{0,1}(0)$.
\end{Lemma}

\begin{proof}
We begin with  models \#2 and~\#7.  We note that $Y_{0,1}(0)=\pm i$ 
 (resp.\ $Y_{0,1}(0)=j,j^2$), while the discriminant $d(x)$ is negative at
$x=0$ in both cases. {Due to}
Lemma~\ref{lem:branch_points}, this implies that $0\in
(x_1, x_2)$, so that  $\pm i$ (resp.\ $j,j^2$) indeed belong to the
curve $\mathcal L=Y([x_1,x_2])$.

Now consider the remaining three models, with $\alpha=-1$. Recall that the curve $\cL$ is bounded (Lemma~\ref{Lem:curve_L_0_infty}), symmetric
with respect to the real axis, and
intersects this axis at exactly two points, namely $Y(x_1)=-b(x_1)/(2a(x_1))$ and
$Y(x_2)=-b(x_2)/(2a(x_2))>0$. Hence we want to
prove that, for models \#1, \#5 and \#6, we have $Y(x_1)<-1$.
Recall that we have shown in the proof of
Lemma~\ref{Lem:curve_L_0_infty} that  $b(x_1)\ge 0$. 

We proceed with a  case-by-case analysis. For models  \#1 and \#6, one
has $a(x)=xc(x)$, hence $d(x_1)=b(x_1)^2-4x_1c(x_1)^2=0$ implies that $x_1\ge
0$, and in fact $x_1>0$ since $b(0)\not =0$ for these models. In particular,
$a(x_1)>0$, $c(x_1)>0$ and 
\begin{equation*}
     Y(x_1)=-\frac{b(x_1)}{2a(x_1)}=-\frac 1 {\sqrt{x_1}}.
\end{equation*}
This is indeed less than $-1$ as $x_1 <1$ (see Lemma~\ref{lem:branch_points}).

For model \#5, one has $a(x)=x(1+x)c(x)$, hence $d(x_1)=b(x_1)^2-4x_1(1+x_1)c(x_1)^2=0$,
which implies similarly that $x_1>0$ (recall that $x_1>-1$).  Now
\[ 
 Y(x_1)=-\frac{b(x_1)}{2a(x_1)}=-\frac 1 {\sqrt {x_1(1+x_1)}}.
\] 
Hence we need to prove that $z:=x_1(1+x_1)<1$. The function $z\equiv z(t)$ is
quartic over $\qs(t)$, and its four branches at $t=0$ are
\begin{align*}
x_1(1+x_1)&=t-2t^{3/2}+3t^2+ O(t^{5/2}),\\
x_2(1+x_2)&=t+2t^{3/2}+3t^2+ O(t^{5/2}), \\
x_3(1+x_3)&=\frac 1{t^2}-\frac 3 t-2+O(\sqrt t),\\
x_4(1+x_4)&=\frac 1{t^2}+\frac 3 t-2+O(\sqrt t).
\end{align*}
A careful study of $z(t)$ shows that it increases between $t=0$ and
$t=1/|\cS|=1/5$, with maximum value $z\simeq 0.09$ at $t=1/5$. In
particular, $z<1$. We omit
the details, but illustrate these facts by the plot of the two small
branches of $z(t)$ in Figure~\ref{fig:z}.
\begin{figure}[htb]
  \centering
  \includegraphics[height=40mm]{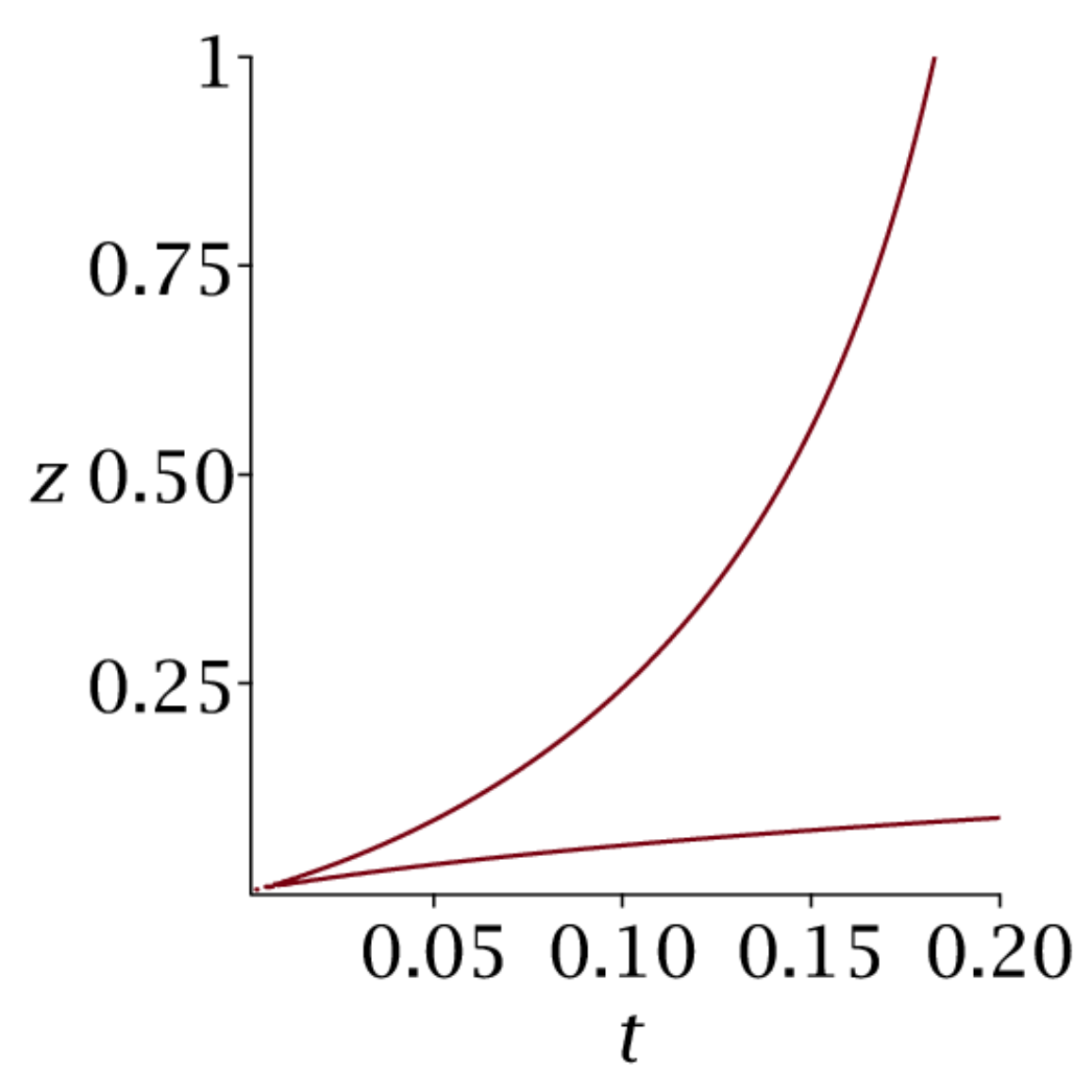}
  \caption{A plot of $x_1(1+x_1)$ (bottom) and $x_2(1+x_2)$ (top) for
    $t\in [0, 1/5]$ in model \#5.}
  \label{fig:z}
\end{figure}
\end{proof}

\begin{proof}[Proof of Theorem~\ref{thm:9models}]
Let $\cS$ be one of the first eight models of
Table~\ref{tab:decoupling_functions-infinite}, and $G(y)$ the
associated decoupling function. By Lemma~\ref{lem:dec-weak}, the identity
\begin{equation*}
xY_0-xY_1=G(Y_0)-G(Y_1)
\end{equation*}
holds at the level of \fps\ (in $t$, with rational coefficients in $x$).
Returning to the analytic framework where $t$ is fixed in $(0,
1/|\cS|)$,  this identity 
holds for any $x\in[x_1, x_2]$ (recall from
Lemma~\ref{Lem:curve_L_0_infty} that the curve $\cL=Y([x_1,x_2])$ is bounded). By~\eqref{eqker:Y01}, any such $x$
thus satisfies $L(Y_0)=L(Y_1)$, where $L(y):=S(y)-G(y)$. Is $L(y)$ a
weak invariant, in the sense of Definition~\ref{def:weak-inv}? 
By Proposition~\ref{prop:S-properties}, this holds if and only if $p\not \in \cL$.
 But this is true
for $p=0$ by  Lemma~\ref{Lem:curve_L_0_infty}, and  for
$p=-1$ (and model \#5) by
Lemma~\ref{Lem:location_particular_points_L}. In both cases, $p$ is in
fact in $\cG_\cL$, and  is the only pole
of $L(y)$ in this domain. Since it is simple in $G$, it is also simple
in $L$. Moreover, it is distinct from the pole $y_2$ of $w$, since $y_2>0$
(Lemma~\ref{lem:branch_points}) while $p\in \{0,-1\}$.

Consider now the function $-\frac{r w'(p)}{w(y)-w(p)}$, where $r$ is
the residue of $G$ at $p$. By Proposition~\ref{Prop:properties_w}, this is also a weak
invariant, with a single pole in $\mathcal G_\cL$, found at $y=p$
(note that $w'(p)$ cannot vanish since $w$ is injective on $\cG_\cL$). Its
residue at $p$ is $-r$. Then the invariant lemma
(Lemma~\ref{lem:inv-analytic}) implies that this function differs from
$L(y)$ by a constant~$c$:
\beq\label{SGw}
S(y)-G(y) =c- \frac{r w'(p)}{w(y)-w(p)}.
\eeq
Since both functions have finite limits  {on}
$\cL$, this identity holds
on $\cL$ as well.  To conclude the proof of~\eqref{S-uniform}, it suffices
to determine the constant $c$. Let $\alpha$ be any root of $K(0,y)$
(see Table~\ref{tab:prg}). Since $\vert\alpha\vert\leq 1$, $Q(0,y)$ and $S(y)=K(0,y)Q(0,y)$ are
analytic in a neighborhood of $\alpha$, as explained at the beginning
of this section, and $S$ vanishes at this point. It remains to  expand the above identity 
at $y=\alpha$, up to the order $O(y-\alpha)$, to determine the value of $c$
as given in the theorem.

\medskip

We now examine the ninth model, which differs from the
first eight  for
two (related) reasons. First, its decoupling function $G$ still has a unique
pole (at $p=0$) but this pole has order two (see Table
\ref{tab:decoupling_functions-infinite}). Moreover, the curve
$\mathcal L$ goes through $0$ (Lemma
\ref{Lem:curve_L_0_infty}). However, the idea of the proof is the same
as for the first eight models: we will prove that $1/(w(y)-w(p))$ also
has a double pole at $0$, and that subtracting a multiple of this
function from $S(y)-G(y)$ yields a pole-free invariant --- which is
unexpected as there might remain a pole of order $1$. 

So let us examine the function $w(y)$ near $y=0$.
Proposition~\ref{Prop:properties_w} implies  that $w$ is analytic in a
neighborhood of $0\in \cL$.  Let us  write $w(y)=\sum_{k\geq0}w_{k}y^{k}$ in this
neighborhood. Solving $K(x,y)=0$ in the neighborhood of $x=0$ gives
for $Y_0$ and $Y_1$ the following expansions, valid in $[0,x_2]$
(recall that $x_1=0$):
\[  
     Y_{0,1}(x)=\pm i\sqrt{x}+\frac{x}{2t}\mp i\frac{x^{3/2}}{8t^2}
+ O(x^2).
\] 
Since $w$ is an invariant, we have $w(Y_0)=w(Y_1)$ for $x\in [0,x_2]$,
which forces $w_1=0$, $w_3=w_2/t$,  and further identities relating the coefficients
$w_k$. Moreover, $w_2\not = 0$: given the form $w(y)=\wp(\cZ(y))$
  of~\eqref{eq:expression_gluing}, and the fact that $\cZ'(0)$ cannot vanish (this
  follows from the identity~\eqref{DE-Z} proved below in Section~\ref{sec:DA}), having $w_1=w_2=0$ would mean that $\wp'$ has a
  multiple root, namely $\cZ(0)$, which is never true for a Weierstrass
  function. Hence, around $y=0$ we have: 
\beq\label{model9-w}
\frac {w_2}{w(y)-w(0)}=\frac1{y^2 }-\frac 1{t
  y}-\frac{w_4}{w_2}+\frac1{t^2}+ O(y).
\eeq
Let us now compare this to the expansion of $S(y)-G(y)$ near $y=0$,
recalling that $S(0)=0$:
\[ 
S(y)-G(y)=\frac 1 {y^2} -\frac1{ty} +O(y).
\] 
This shows that $S(y)-G(y)- w_2/(w(y)-w(0))$ is an invariant with no
pole on $\cG_\cL$. Applying Lemma~\ref{lem:inv-analytic} implies that this function is
constant in $\cG_\cL$, and the above expansions give this value as
$w_4/w_2-1/t^2$, as stated in~\eqref{ninth}.
\end{proof}

\subsection{The finite group case}
\label{sec:weak-finite}

Our approach using weak invariants is robust, and applies
to the four (unweighted) decoupled models with a finite group
already solved in Section~\ref{sec:extensions} via an algebraic approach.

The method is exactly the same as in Section~\ref{subsec:analinv},
as long as the polynomial $K(0,y)$ has at least one root $\alpha$.
Recall indeed that we 
used $\alpha$  to identify the constant $c$ in~\eqref{SGw}. As explained below
the proof of Lemma~\ref{Lem:curve_L_0_infty}, the existence of~$\alpha$ is equivalent to \emph{the boundedness of the curve~$\cL$,}
which holds for Kreweras' model, the double
Kreweras' model and Gessel's model (Figure~\ref{fig:alg_models}). In all three cases,
the function $S(y)$ is still given by~\eqref{S-uniform}, where the
decoupling function $G(y)$ is given in
Table~\ref{tab:decoupling_functions-finite}, the weak invariant $w(y)$
   by~\eqref{eq:expression_gluing}, and the various constants
by Table~\ref{tab:prg-finite}. 
The connection  between the weak invariant $w(y)$ and the rational invariant $J(y)$ of Table~\ref{tab:ratinv} will be made explicit in Section~\ref{sec:rational_weak}.

For the above mentioned three models, an alternative is to proceed as in Sections~\ref{sec:gessel} and~\ref{sec:extensions}, up to the point where we construct a series $A(y)$ satisfying the conditions of Lemma~\ref{lem:inv} (for instance $(L(y)-L(0))J(y)-C_1L(y)^3-C_2L(y)^2-C_3L(y)$ in Section~\ref{sec:gessel}), and then apply the analytic invariant lemma (Lemma~\ref{lem:inv-analytic}) rather than the formal one  (Lemma~\ref{lem:inv}) to conclude that this series depends on $t$ only.

\begin{table}[h]
   \begin{tabular}{c|ccc|}
model     & \begin{tikzpicture}[scale=.2] 
    \draw[->] (0,0) -- (1,1);
    \draw[->] (0,0) -- (-1,0);
    \draw[->] (0,0) -- (0,-1);
        \draw[-] (0,-1) -- (0,-1) 
;
  \end{tikzpicture} &     \begin{tikzpicture}[scale=.2] 
      \draw[->] (0,0) -- (0,-1);
    \draw[->] (0,0) -- (1,1);
    \draw[->] (0,0) -- (-1,0);
    \draw[->] (0,0) -- (0,1);
    \draw[->] (0,0) -- (-1,-1);
    \draw[->] (0,0) -- (1,0);
    \draw[-] (0,-1) -- (0,-1) 
;
  \end{tikzpicture} &   \begin{tikzpicture}[scale=.2] 
      \draw[->] (0,0) -- (-1,-1);
    \draw[->] (0,0) -- (1,1);
    \draw[->] (0,0) -- (-1,0);
    \draw[->] (0,0) -- (1,1);
    \draw[->] (0,0) -- (1,0);
    \draw[-] (0,-1) -- (0,-1)
;
  \end{tikzpicture} \\
\hline
$p$& 0&0& $-1$\\
\hline
$r$& $-1$ & $-1$ &$-1/t$\\
\hline
$\alpha$ & $0$ &$-1$ &  $-1$\\
\hline
$g_0$ & 0&1&0
  \end{tabular}

\vskip 3mm
 \caption{Values of $p$, $r$, $\alpha$ and $g_0$ for three
   algebraic models.}
\label{tab:prg-finite}
\end{table}

Let us now examine what happens in a model for which $K(0,y)=t$, 
and solve the reverse Kreweras model (Figure~\ref{fig:curves}, right). We follow the proof of
Theorem~\ref{thm:9models}. By Proposition~\ref{prop:S-properties}, the function $L(y)=S(y)-G(y)=tQ(0,y)+1/y$
is meromorphic in $\cG_\cL$, with a unique pole at $0$, and is bounded
at infinity. It is thus an invariant.  The analytic invariant
lemma tells us that  
\beq\label{RK-weak}
tQ(0,y)=S(y)=G(y) +\frac{w'(0)}{w(y)-w(0)} +c
\eeq
for some constant $c$. 
Since $K(0,y)=t$ has no
root (in $y$), we cannot use the same trick as in the proof of
Theorem~\ref{thm:9models} to determine $c$.
But we can expand the above identity, first at $y=0$, which
gives
\[ 
  S(0)=-\frac{w''(0)}{2w'(0)}+c,
\] 
and then at the unique point $y_c\in (0,1)$ such that $K(y,y)=0$
(Figure~\ref{fig:Kyy}; this point is always in~$\cG_\cL$ since
$Y(x_2)= 1/\sqrt{x_2}>1$), which gives
\[ 
S(y_c)= G(y_c) +\frac{w'(0)}{w(y_c)-w(0)} +c.
\] 
 Now applying~\eqref{eqker:xYi} with $x=y_c$, and using the $x/y$
 symmetry of the model, we find that
 $2S(y_c)=y_c^2+S(0)$. We  now use the above expressions of $S(0)$
   and $S(y_c)$ to determine $c$ in terms of $w$.
Returning to~\eqref{RK-weak}, and using $G(y)=-1/y$, this finally gives 
\[
  S(y)=- \frac 1 y
  +\frac{w'(0)}{w(y)-w(0)} +y_c^2+\frac 2
  {y_c}-\frac{w''(0)}{2w'(0)}-\frac{2w'(0)}{w(y_c)-w(0)}
\] 
for the reverse Kreweras model.

\begin{figure}[htb]
  \centering
  \includegraphics[height=40mm]{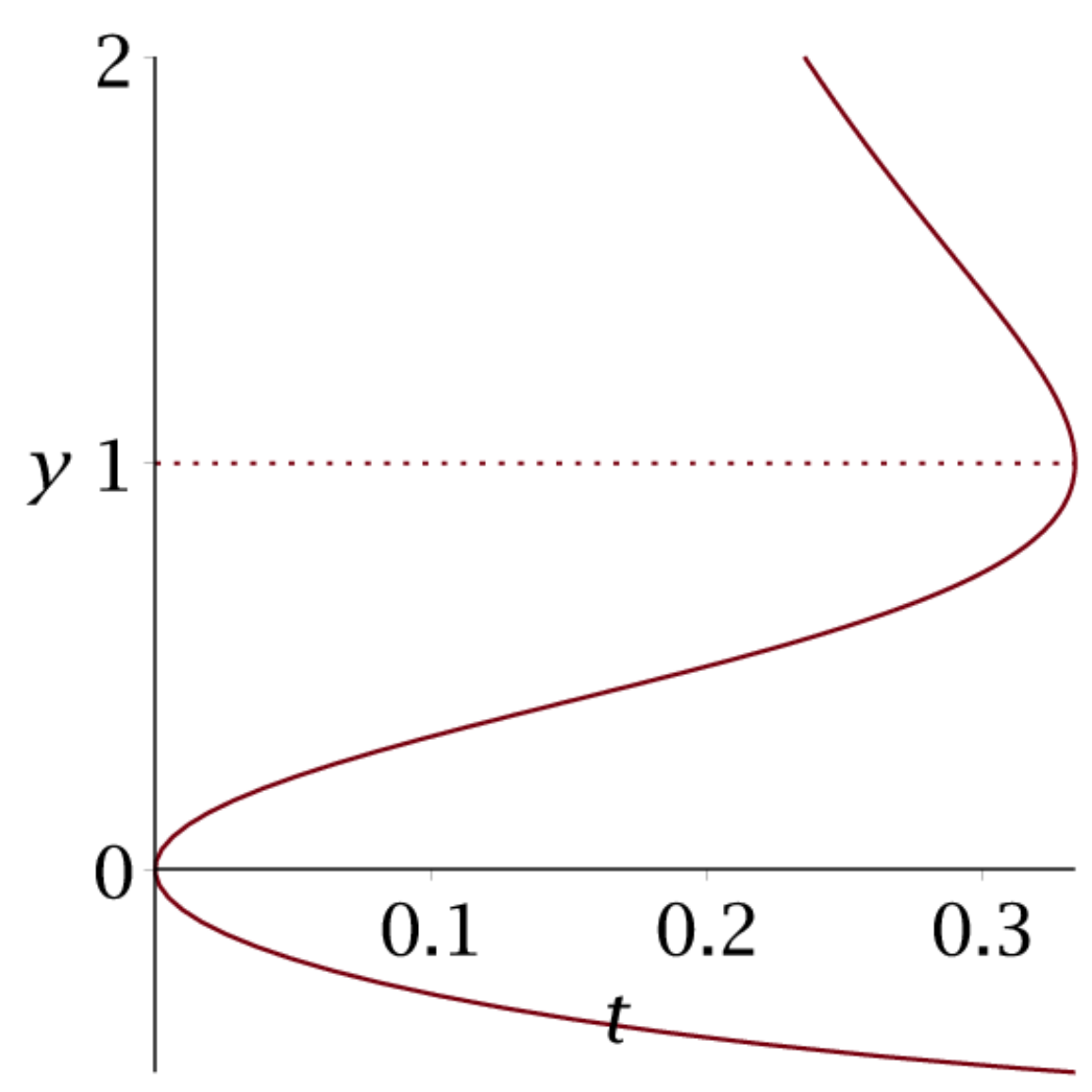}
  \caption{A plot of the three branches of $K(y,y;t)=0$ against $t$ for $t\in
    [0,1/3]$.}
  \label{fig:Kyy}
\end{figure}

 We expect our analytic method to be also applicable to
  the four weighted algebraic models of Figure~\ref{fig:alg_models}, right, provided one develops the counterpart of~\cite{Ra-12} for steps with (real and positive) weights.

\section{Differential algebraicity}
\label{sec:DA}
As recalled in the introduction,  quadrant walks have a D-finite \gf\ if and only if the
associated group is finite
--- we can now say, if and only if they admit  rational invariants
(Theorem~\ref{thm:finite-invariant}).  
Here we will show that the $9$ models with an infinite group that have
decoupling functions
still satisfy \emm polynomial, differential equations. 
This property will be derived from the new expression for the generating function of these 
models that we obtained in Section~\ref{sec:analysis} by 
the analytic invariant approach.

\begin{Theorem}
\label{thm:DA}
For any of the $9$  models of Table~\ref{tab:decoupling_functions-infinite}, 
the generating function $Q(x,y;t)$ is
\emm differentially algebraic, (or: \emm D-algebraic,) in
$x,y,t$. That is, it
satisfies \emm three, polynomial differential equations with coefficients in $\qs$: one in $x$,
one in $y$ and one in $t$. 
\end{Theorem}

As discussed in the introduction, this is the first
  D-algebraicity result for (some) non-D-finite quadrant models~\cite{BeBMRa-FPSAC-16}, and  the 47 
  other non-D-finite models have been proved to be
  hypertranscendental since then~\cite{DHRS-17,DHRS-sing,DH-t}.

\subsection{Generalities}\label{sec:genDA}
We consider an abstract differential field $\GK$ of characteristic $0$
{equipped with one or several} derivations. 
Typical examples occurring in this section are:
\begin{itemize}
\item the field of meromorphic functions in $k$ variables $x_1,
  \ldots, x_k$ over a complex domain $D$ of $\cs^k$, equipped with
 the  {derivations}  $\partial/\partial x_i$, 
\item the quotient field of the (integral) ring of formal power series
  in the variables $x_1, \ldots, x_k$ with coefficients in $\qs$,
  equipped with   the {derivations}   $\partial/\partial x_i$,
\item at the end of the section, the field of Laurent series in $t$
  with rational coefficients in $x$ and $y$, equipped with the three
   {derivations}    $\partial /\partial t$, $\partial /\partial x$ and
  $\partial /\partial y$.
\end{itemize}

\begin{Definition}\label{def:DA}
{Let $\GK$ be a differential    field, with     a derivation $\delta$.}
An element  $F$ of $\,\GK$ is
  \emph{$\delta$-algebraic}, if there exists a non-zero polynomial $P(x_0, x_1,
  \ldots, x_d)$ with coefficients in $\qs$ such that
\[ 
P(F, \delta F, \ldots, \delta^{(d)}F)=0.
\] 
\end{Definition}
When $F$ is a function or a series involving $k$ variables $x_1,
\ldots, x_k$, as in the above
 examples, we say that $F$ is \emph{differentially algebraic in $x_i$},
(or \emph{DA in $x_i$}) if it
is $\partial /\partial x_i$-algebraic. We say that~$F$ is \emph{globally DA}, (or \emph{DA}, for short)
if it is DA in each of its variables.

It may be surprising that we do not allow polynomial coefficients in
the definition of DA series or functions. In fact, this would not
enlarge the DA class: indeed, imagine for instance that the series (or function)
$F(x,y)$ satisfies a non-trivial equation
\[ 
P( x,y, F(x,y), \ldots, F^{(d)}(x,y))=0,
\] 
where  the derivatives are taken with respect to $x$, {and the
  variable $y$ actually occurs}. Then
differentiating with respect to $x$ gives another differential
equation (DE).
If it does not involve~$y$, then we have found a DE
  that is free from $y$. Otherwise, we can eliminate $y$ between this
  new equation and the above   one to obtain a DE free from $y$
(the resultant will involve $F^{(d+1)}(x,y)$, and thus cannot be trivially zero).
With one more differentiation, we can similarly construct a DE free from $x$ (and
$y$) and conclude that $F$ is DA in $x$. 

An important subclass of DA series (or functions) consists of \emm differentially
finite, series (or functions): We say that $F$ is
  \emm D-finite, in  $x_i$ (for short: DF in $x_i$) if there exist
   polynomials $P_j (x_1, \ldots, x_k) $ in $\qs[x_1, \ldots, x_k]$,
   for $0\leq j\leq d$,
   not all zero, such that
\[ 
P_d(x_1, \ldots, x_k)F^{(d)}+ \cdots+ P_1(x_1, \ldots, x_k)F' + P_0(x_1, \ldots, x_k)F =0,
\] 
where the derivatives are taken with respect to $x_i$. We say that $F$
is \emm globally differentially finite, (or
D-finite, or DF) if it is DF in each $x_i$. Finally, a simple subclass of DF
series (or functions) consists of \emm algebraic, elements, that is,
series or functions satisfying a non-trivial polynomial equation
\[ 
P(x_1, \ldots, x_k, F)=0
\] 
with coefficients in $\qs$.

The notions of D-finite and D-algebraic series/functions are standard~\cite{lipshitz-df,lipshitz-diag,ritt-vol2,stanley-vol2},
but D-finite series, having a lot of structure, seem to be discussed more
often than DA series, at least in the combinatorics literature. Note
that 
if a series is DF (resp.\ DA), the function that it defines in (say) its
polydisc of convergence is also DF (resp.\ DA).  Conversely, any
differential equation satisfied in the neighborhood of some point
$a=(a_1, \ldots, a_k)$ by a function~$F$ analytic around $a$ holds at the level of power series for the
series expansion of $F$ around $a$. 

We will use a number of  closure properties. Some of them can be
stated in the context of an abstract differential field, using the
following proposition.

\begin{Proposition}\label{prop:DA-equiv}
  Let $\GK$ be a differential field of characteristic $0$, with
  {a} derivation $\delta$. Let $F\in \GK$.  The following statements are equivalent:
  \begin{enumerate}[label={{\rm(\arabic{*})}},ref={{\rm(\arabic{*})}}]
  \item\label{it1:DA-equiv}$F$ is $\delta$-algebraic,

\item\label{it2:DA-equiv}there exists $d\in \ns$ such that all $\delta$-derivatives of $F$  belong to
  $\qs(F, \delta F, \ldots, \delta^{(d)}F)$,
 
\item\label{it3:DA-equiv}there exists a field extension $K$ of $\qs$ of finite
  transcendence degree  that contains $F$ and all its $\delta$-derivatives,
  
 \end{enumerate}
\end{Proposition}
\begin{proof}
  $\ref{it1:DA-equiv}\Rightarrow \ref{it2:DA-equiv}$. Take a DE for $F$ of minimal order, and minimal
  total degree among DEs of minimal order:
\[ 
P(F, \ldots, \delta^{(d)}F)=0.
\] 
Applying $\delta$ gives:
\[ 
\left(\delta^{(d+1)}F\right) P_1(F, \ldots, \delta^{(d)}F) + P_2(F, \ldots, \delta^{(d)}F)=0
\] 
for some polynomials $P_1$ and $P_2$. The total degree of $P_1$ is
less than the total degree of $P$, and thus by minimality of $P$,
$P_1(F, \ldots, \delta^{(d)}F)$ is non-zero. Property~\ref{it2:DA-equiv} then
follows by induction on the order of the derivative.

\medskip
$\ref{it2:DA-equiv}\Rightarrow \ref{it3:DA-equiv}$. The field $K=\qs(F, F', \ldots, \delta^{(d)}F)$
 contains all derivatives of $F$ and has transcendence
degree at most $d+1$  (recall that $\qs(x_1, \ldots, x_k)$ has
transcendence degree $k$).

\medskip 
$\ref{it3:DA-equiv} \Rightarrow \ref{it1:DA-equiv}$. If $K$ has transcendence degree $d$, then the $d+1$
functions $F, F',
\ldots, \delta^{(d)}F$ are algebraically dependent over $\qs$. 
\end{proof}

The following closure properties easily follow.
\begin{Corollary}\label{cor:closure}
  The set of $\delta$-algebraic elements of $\GK$ forms a field. This field
  is closed under $\delta$, and in fact under any 
 derivation $\partial$ that
  commutes with $\delta$.
\end{Corollary}
\begin{proof}
Assume that  $F$ and $G$ are $\delta$-algebraic. Say that all derivatives
of $F$ belong to $\qs( F, \ldots, \delta^{(d)}F)$, and all
derivatives of $G$ belong to $\qs( G, \ldots, \delta^{(e)}G)$. Then all
derivatives of $F+G$ and $FG$ belong to $\qs( F, \ldots,
\delta^{(d)}F,G, \ldots, \delta^{(e)}G)$, so that $F+G$ and $FG$ are
$\delta$-algebraic by Proposition~\ref{prop:DA-equiv}\ref{it3:DA-equiv}.  Similarly, all derivatives
of $1/F$ belong to $\qs( F, \ldots, \delta^{(d)}F)$, so that $1/F$ is
$\delta$-algebraic. The closure under $\delta$ of the field of
$\delta$-algebraic elements is obvious by Proposition~\ref{prop:DA-equiv}\ref{it2:DA-equiv}. Finally, if~$\partial$ is
another derivation commuting with $\delta$, then $\partial F$
  satisfies the same DE as $F$, and is thus $\delta$-algebraic.
\end{proof}
Specialized to series or  functions in $k$ variables $x_1,
\ldots, x_k$, the above
corollary implies that
$F+G$, $FG$, $1/F$, $\partial F/\partial x_i$ are DA as soon as $F$
and $G$ are DA. We will need one final closure property, involving composition.

\begin{Proposition}\label{prop:closure-comp}
If $F(y_1, \ldots, y_r)$ is a DA series (or   function) 
of $r$ variables, $G_1(x_1,\ldots, x_k), \ldots$,
$G_r(x_1, \ldots, x_k)$ are DA in all $x_i$'s, and the composition
  $H:=F(G_1, \ldots, G_r)$ is well defined, then $H$ is DA in the $x_i$'s.  
\end{Proposition}
\begin{proof}
Let us prove that $H$ is DA in $x_1$. If, for $1\leq i \leq r$,  all $y_i$-derivatives of $F(y_1, \ldots, y_r)$ can be expressed
rationally in terms of the first $d_i$ derivatives, and all
$x_1$-derivatives of $G_j$ in terms of the first $e_j$ ones, then all
$x_1$-derivatives of $H$ can be expressed rationally in terms of 
\begin{itemize}
\item the functions $\partial^{a} G_j/\partial x_1^{a}$, for $1\leq j \leq r$ and $0\leq a <  e_j$, 
\item and the functions 
\[ 
\frac{\partial^{c_1+ \cdots + c_k} F}{\partial y_1^{c_1} \cdots \partial y_r^{c_r}}(G_1,
\ldots, G_k)
\] 
for $0\leq c_i < d_i$, $1\leq i \leq r$. We then apply Proposition~\ref{prop:DA-equiv}\ref{it3:DA-equiv}.\qedhere
\end{itemize}
\end{proof}
D-finiteness is not preserved in general by composition,  but we still
have the following result  (see~\cite{lipshitz-df} for a
proof in the series setting).
\begin{Proposition}\label{DF-alg}
If $F(y_1, \ldots, y_r)$ is a D-finite series (or   function) 
of $r$ variables, $G_1(x_1,\ldots, x_k)$, $\ldots,
G_r(x_1, \ldots, x_k)$ are algebraic in the $x_i$'s, and the composition
  $H:=F(G_1, \ldots, G_r)$ is well defined, then $H$ is D-finite in the $x_i$'s.
\end{Proposition}

\subsection{The Weierstrass elliptic function}
It is well known that the Weierstrass function $\wp(z, \om_1, \om_2)$
defined by~\eqref{w-def} is DA in~$z$. What may be less known is
that it is DA in its periods $\om_1$ and $\om_2$ as well.   Since we
could not find any reference in the literature, we will sketch a
proof. 
\begin{Proposition}\label{prop:P-DA}
   The function   $\wp(z, \om_1, \om_2)$ is DA in $z$, $\om_1$ and $\om_2$.
\end{Proposition}
\begin{proof}
 We refer to~\cite{jones-singerman,ww} for generalities on the Weierstrass function
 (but we draw the attention of the reader on the fact that the
 periods are $2\om_1$ and $2\om_2$ in~\cite{ww},
instead of $\om_1$ and $\om_2$ (or $\om_3$) in our paper).
The following differential equation is well known to hold:
\[ 
 \wp _z(z, \om_1,\om_2)^2=4 \wp(z,\om_1,\om_2)^3-g_2(\om_1, \om_2)
\wp(z,\om_1,\om_2) -g_3(\om_1, \om_2),
\] 
which we shorten as
\beq\label{diff-p}
\wp _z^2=4 \wp^3-g_2\wp -g_3,
\eeq
where $g_2$ and $g_3$ (also called \emm invariants, in the
elliptic terminology!) depend on the periods only.

We now use the connection between the Weierstrass function and
Jacobi's theta function:
    \[
          \theta(z;\tau)=\sum_{n\in \zs} e^{i(2n+1)z +i\pi \tau (n+1/2)^2}.
        \]
 Indeed, 
    \beq\label{P-theta}
      \wp(z, \pi, \pi \tau)= - \frac {\rm{d}}{ {\rm d} z} \left( \frac{\theta_z(z;
          \tau)}{\theta(z;\tau)}\right) + \frac{\theta_{zzz}(0;
          \tau)}{3 \theta_z(0;\tau)}.
      \eeq
This can be easily proved by observing that the right-hand side of~\eqref{P-theta} has
periods $\pi$ and $\pi\tau$, and behaves like $1/z^2+O(z^2)$ around
zero --- two properties that characterize $\wp(z, \pi, \pi \tau)$.
    
        It follows from this and~\eqref{diff-p} that $\theta(z;\tau)$
        is DA in $z$. Hence $\theta(z;\tau)$ and its $z$-derivatives
        span a field of finite transcendence degree over $\qs$. By
        the heat equation~\cite[Sec.~21.4]{ww},
        \[
          \theta_\tau(z;\tau)= -\frac {i\pi}{4} \theta_{zz}(z;\tau),
        \]
the $\tau$-derivatives of $\theta$ are also contained in a field of
finite transcendence degree, and $\theta$ is thus DA in~$\tau$. The same
holds for any $z$-derivative of $\theta$, upon differentiating the heat equation.

It now follows from~\eqref{P-theta} and Corollary~\ref{cor:closure}
that $\wp(z, \pi, \pi \tau)$ is DA in $\tau$ as
well. Finally, since
\[
  \wp(z, \om_1, \om _2)= \frac {\pi^2} {\om_1^2}\,  \wp\!\left( \frac {\pi z}{\om_1}, \pi,
    \pi \frac{\om_2}{\om_1}\right),
\]
we conclude by Proposition~\ref{prop:closure-comp} that $\wp$ is DA in $\om_1$ and $\om_2$ as well.        
\end{proof}

\subsection{The weak invariant $\boldsymbol w$}
We now consider a non-singular, unweighted quadrant
model. Recall the expression~\eqref{eq:expression_gluing} of the weak invariant $w(y;t)$,
valid for $t\in (0,1/\vert\cS\vert)$ and $y$ in the complex domain $\GL$,
which depends on $t$
(Proposition~\ref{Prop:properties_w}). From now on we will
  often insist on the dependence   in $t$ {in the notation} of our functions, 
writing for 
instance $\om_1(t)$ rather than   $\om_1$.

\begin{Theorem}\label{prop:w-DA}
  For any non-singular model, the weak invariant $w(y;t)$
  defined by~{\rm\eqref{eq:expression_gluing}} can be extended
  analytically to a domain of $\cs^2$ where it is
   D-algebraic in $y$ and~$t$.
\end{Theorem}                

Recall that
\[ 
w(y;t)= \wp\left( \cZ(y;t), \om_1(t), \om_3(t)\right),
\] 
where the periods $\om_1$ and $\om_3$ are given
by~\eqref{eq:expression_periods} and the first argument  of $\wp$ is
\beq\label{Z-def}
\cZ(y;t)=-\frac{\omega_1(t)+\omega_2(t)}{2}+\wp_{1,2}^{-1}(f(y;t)).
\eeq
We will argue by composition of DA functions. We have already proved
in the previous subsection that $\wp$ is DA in its three
arguments. Our next objective will be to prove that $\om_1$ and
$\om_3$ are DA in $t$ (and in fact D-finite, see Lemma~\ref{lem:omega-DF}). We will then proceed
with the bivariate function $\cZ$, which is also D-finite (Lemma~\ref{lem:Z-DF}).

As a very first step, we consider the branch points $x_\ell$ of $Y_{0,1}$, and the branch points $y_\ell$ of $X_{0,1}$.
 \begin{Lemma}\label{lem:xi-holom}
The functions $x_1, x_2, x_3$ are algebraic functions of $t$, and so is $x_4$ when it is finite.
They are analytic and distinct in  a neighborhood of the interval
$(0,1/\vert\cS\vert)$.

The same holds for the branch points $y_\ell$.
 \end{Lemma}
  \begin{proof}
The $x_i$'s are the roots of $d(x)=b(x)^2-4a(x)c(x)$. This is a cubic
or quartic polynomial in $x$, with coefficients in $\qs[t]$. Its discriminant does not vanish in
$(0,1/\vert\cS\vert)$ since the $x_i$'s are distinct on this interval (see Lemma~\ref{lem:branch_points}). Its dominant coefficient is easily checked
to be $c\,t^2$, for some $c\not=0$. Since the singularities of algebraic
functions are found among the roots of the discriminant and of the
dominant coefficient, we conclude that the $x_i$'s are non-singular on
$(0,1/\vert\cS\vert)$, and thus (since their singularities are isolated) in a
complex neighborhood of this segment.
\end{proof}

 \begin{Lemma}\label{lem:omega-DF}
   The functions $\om_1/i, \om_2$ and $\om_3$, defined by~\eqref{eq:expression_periods} for $t \in (0,
    1/\vert\cS\vert)$, are real and positive. They can be
    extended  analytically in a complex neighborhood of $(0,
    1/\vert\cS\vert)$, where  they are D-finite.
  \end{Lemma}
  \begin{proof}
 The periods $\om_i$ are  expressed in~\eqref{eq:expression_periods} as
elliptic integrals. Using the classical reduction to Legendre
forms~\cite[Sec.~22.7]{ww}, we can express them in terms of complete
and incomplete elliptic integrals of the first kind, defined
respectively, for $k\in (-1,1)$ and $v\in (-1,1)$, by
\[ 
     K(k)=\int_0^1\frac{\text{d}y}{\sqrt{1-y^2}\sqrt{1-k^2y^2}}
\] 
and
\beq\label{def:F}
     F(v,k)=\int_0^v\frac{\text{d}y}{\sqrt{1-y^2}\sqrt{1-k^2y^2}}.
\eeq
Then we claim that:
\beq
\label{after_change_1}
     \om_1 = i \alpha \ K\left(\sqrt{\frac{(y_2-y_1)(y_4-y_3)}{(y_3-y_1)(y_4-y_2)}}\right),
\qquad \om_2 = \alpha\
K\left(\sqrt{\frac{(y_3-y_2)(y_4-y_1)}{(y_3-y_1)(y_4-y_2)}}\right)
\eeq
and
\beq\label{eq:new_o_3_F}
     \om_3 = \alpha\ 
     F\left(\sqrt{\frac{(y_1-Y(x_1))(y_4-y_2)}{(y_2-Y(x_1))(y_4-y_1)}},\sqrt{\frac{(y_3-y_2)(y_4-y_1)}{(y_3-y_1)(y_4-y_2)}}\right),
\eeq
where the prefactor $\alpha$  is an algebraic function of $t$, which
depends  on the degree ($3$
or $4$) of $\widetilde d(y)$, and of
its dominant coefficient $\widetilde d_3$ or $\widetilde d_4$:
\beq
\label{eq:def_alpha}
\alpha=\begin{cases}
 \displaystyle  \frac 2 {\sqrt{\widetilde d_3(y_1-y_3)}}& \hbox{if } \widetilde d_4=0,\\
 \displaystyle \frac 2 {\sqrt{\widetilde d_4(y_3-y_1)(y_4-y_2)}}& \hbox{otherwise}.
\end{cases}
\eeq
The dominant  coefficient  $\widetilde d_3$ or $\widetilde d_4$ is always of the
form $\varepsilon c t^2$, with $\vareps=\pm1$  and $c\in
\{1,3,4\}$. 
The sign $\vareps$ equals $+1$ if and only if $y_4$ is finite and
positive, so that $\alpha$ is always real and positive (see
Lemma~\ref{lem:branch_points} for the properties of the $y_i$'s).

To obtain the above expressions for the periods, one starts from their
original expressions in terms of $\widetilde d(y)$ (see~\eqref{eq:expression_periods}) and performs
the following change of variable in the integrand (for
$\om_1$, $\om_2$ and $\om_3$ respectively):
\[ 
z= \sqrt\frac{(y-y_1)(y_2-y_4)}{(y-y_4)(y_2-y_1)},\qquad
z= \sqrt\frac{(y-y_2)(y_3-y_1)}{(y-y_1)(y_3-y_2)},\qquad
z= \sqrt\frac{(y-y_1)(y_4-y_2)}{(y-y_2)(y_4-y_1)}.
\] 
The calculation is then straightforward.

If $\widetilde d_4=0$, that is, $y_4=\infty$, then the argument of $K$ in~\eqref{after_change_1} reduces to
$\sqrt{(y_2-y_1)/(y_3-y_1)}$ (resp.\ $\sqrt{(y_3-y_2)/(y_3-y_1)}$) in
the expression of 
$\om_1$ (resp.\ $\om_2$). Similarly, the arguments of
$F$ in the expression~\eqref{eq:new_o_3_F} for
$\om_3$ are replaced by their limits
as $y_4\rightarrow \infty$.
Observe that Lemma~\ref{lem:branch_points} implies that the ratios
\[ 
\frac{(y_2-y_1)(y_4-y_3)}{(y_3-y_1)(y_4-y_2)}\qquad \hbox{and} \qquad \frac{(y_3-y_2)(y_4-y_1)}{(y_3-y_1)(y_4-y_2)}
\] 
are positive. Since they sum to $1$, they both belong to $(0,1)$, so
that the values of $K$ are well defined in~\eqref{after_change_1}. A
similar argument, relying on Lemma~\ref{Lem:curve_L_0_infty}, proves that the first argument of $F$ in~\eqref{eq:new_o_3_F} lies
in $(0,1)$. The second argument already appears in~\eqref{after_change_1}.

\medskip
The function $K(k)$ has a convergent expansion of radius $1$ around
$k=0$:
\[ 
K(k)= \frac \pi 2 \sum_{n\ge 0} {2n\choose n}^2 \left( \frac k
    4\right)^{2n}.
\] 
It is D-finite, with differential equation
\[ {\displaystyle {\frac {\mathrm {d} }{\mathrm {d}
      k}}\left[k(1-k^{2}){\frac {\mathrm {d} K(k)}{\mathrm {d}
        k}}\right]=kK(k)}
.\] 
Its only singularities are at $\pm1$, and it has an analytic continuation {on}  
$\cs\setminus((-\infty, -1) \cup (1,
+\infty))$. By Lemma~\ref{lem:xi-holom}, the  arguments
involved in the expressions~\eqref{after_change_1} of 
$\om_1$ and $\om_2$ still have modulus
less than 1 in some neighborhood  of $(0,1/\vert\cS\vert)$, where the
$\om_i$ are thus analytic. By Proposition~\ref{DF-alg}, these two
periods are also D-finite in $t$. 

Let us now return to the expression~\eqref{eq:new_o_3_F} for
$\om_3$.
The function $F(v,k)$ has an expansion around $(0,0)$ that converges
absolutely for $\vert v\vert<1$ and $\vert k\vert<1$:
\[ 
F(v,k)=\sum_{m,n\ge0} {2m\choose m}{2n\choose n}
\frac{k^{2n}}{4^{m+n}}\cdot \frac{v^{2m+2n+1}}{2m+2n+1}.
\] 
It is D-finite in each of its two variables (as a bivariate series and
thus as a function). Indeed,
\[ 
(1-v^2)(1-k^2v^2) \frac{\partial ^2 F}{\partial v^2} 
= v(1+k^2-2k^ 2v^2) \frac{\partial  F}{\partial v} 
\] 
and
\[ 
3k^3v^2 F+ (13 k^4 v^2-2 k^2 v^2-4 k^2-1) \frac{\partial  F}{\partial k}  + k
(8 k^4 v^2-4 k^2 v^2-5 k^2+1) \frac{\partial ^2 F}{\partial k^2}  +k^2  (1-k^2)(1-k^2v^2)  \frac{\partial ^3 F}{\partial k^3} =0.
\] 
  Again, we conclude that $\om_3$ is D-finite in a
neighborhood of $(0, 1/\vert\cS\vert)$ by composition with algebraic functions.
  \end{proof}

 \begin{Lemma}\label{lem:Z-DF}
   The function $\cZ (y;t)$ defined by~\eqref{Z-def}
 for $t\in (0, 1/\vert\cS\vert)$ and $y \in \mathcal G_\mathcal
L$ can be analytically continued to a domain of $\cs^2$ in which it  is D-finite in $t$ and $y$. 
  \end{Lemma}
  
  \begin{proof}
 In order to understand the nature of $\cZ$, we need to go back
   to the parametrization of the curve $K(x,y)=0$ by the function
   $\wp_{1,2}$. Let us first assume that $y_4$ is finite. Then $\wp_{1,2}$ has been constructed in such a
   way that, for any $z$, the pair $(x,y)$ defined by
\begin{align}
\label{eq:expression_unif}
y  &= y_4 + \dfrac{\widetilde d'(y_4)}{\wp_{1,2}(z) - \frac{1}{6} \widetilde d''(y_4)},
\\
2\widetilde a(y)x+\widetilde b(y)  &= \dfrac{\widetilde d'(y_4) \wp_{1,2}'(z)}{2(\wp_{1,2}(z) -
\frac{1}{6} \widetilde d''(y_4))^2}, \nonumber
\end{align} 
satisfies $K(x,y)=0$ (see~\cite[Lem.~3.3.1]{FIM-99} in the probabilistic setting). In other words, if $f(y)=\wp_{1,2}(z)$, with $f$
defined by~\eqref{eq:def_f}, then~\eqref{eq:expression_unif}
  holds and
\beq\label{dtilde}
\widetilde d(y) = \dfrac{(\widetilde d'(y_4) \wp_{1,2}'(z))^2}{4(\wp_{1,2}(z) -
\frac{1}{6} \widetilde d''(y_4))^4}.
\eeq
The identity $f(y)=\wp_{1,2}(z)$ holds in particular for $z=\wp_{1,2}^{-1}( f(y))=\cZ(y;t)+
(\om_1+\om_2)/2$.

Let us now differentiate $\cZ$ with respect to $y$:
\[ 
\cZ'(y) =\frac  {f'(y)}{ \wp_{1,2}'\circ \wp_{1,2}^{-1}( f(y))}=\frac  {f'(y)}{ \wp_{1,2}'(z)}.
\] 
 Upon squaring this identity, and using
first~\eqref{dtilde}, and then~\eqref{eq:expression_unif}, we obtain
\beq\label{DE-Z}
 \left(\cZ'(y)\right)^2= \frac
  {(f'(y))^2 }{4\widetilde d(y)}\dfrac{(\widetilde d'(y_4) )^2}{(\wp_{1,2}(z) -
\frac{1}{6} \widetilde d''(y_4))^4}= \frac
  {\left(f'(y)\right)^2 }{4\widetilde d(y)}\frac{(y-y_4)^4}{\widetilde
    d'(y_4)^2}=  \frac
 1{4\widetilde d(y)}
\eeq
by definition of $f$. If $y_4$ is infinite 
(that is, if  $\widetilde d_4=0$),
then the parametrization of $K(x,y)=0$ is
\begin{align*}
  y &= \frac{\wp_{1,2}(z)-\widetilde d''(0)/6}{\widetilde d'''(0)/6},
\\
2\widetilde a(y)x+\widetilde b(y)  &= - \frac{3 \wp'_{1,2}(z)}{\widetilde d'''(0)},
\end{align*}
and the identity  $\left(\cZ'(y)\right)^2=1/(4\widetilde d'(y))$ still holds.

 Another property of the parametrization of the kernel by $\wp_{1,2}$ is 
 that $f(y_2)= \wp_{1,2}(\frac{\om_1+\om_2}2)$ (see~\cite{Ra-12},
 below (18), recalling that we have swapped the roles of $x$ and $y$). Hence, given our
 convention in the definition of $\wp_{1,2}^{-1}$ in Section~\ref{subsec:weak_invariants}, we have 
\[ 
\cZ(y_2;t)=0.
\] 

Finally, recall that $\widetilde d(y)$ is real and positive for $y \in (y_2,
y_3)$ (Lemma~\ref{lem:branch_points}). Hence,  for $y \in \GL \cap
[y_2, y_3]$, it follows from~\eqref{DE-Z} that
\[ 
\cZ(y)=- \frac 1 4 \int_{y_2}^y \frac {\textnormal{d}u}{\sqrt{\widetilde d (u)}}.
\] 
(The minus sign comes again from the determination of $\wp_{1,2}^{-1}$ that
we have chosen, which has real part in $[0, \om_2/2]$. Hence the real
part of $\cZ(y;t)$ is non-positive.)
This integral can be expressed in terms of the incomplete elliptic
function $F(v,k)$ defined by~\eqref{def:F}, as we did for the period
$\om_3$ in the proof of Lemma~\ref{lem:omega-DF}:
\[ 
\cZ(y)= - \frac{1 }{\alpha} F\left(
  \sqrt{\frac{(y-y_2)(y_3-y_1)}{(y-y_1)(y_3-y_2)}},
  \sqrt{\frac{(y_3-y_2)(y_4-y_1)}{(y_3-y_1)(y_4-y_2)}}\right),
\] 
where the prefactor $\alpha$ is given by~\eqref{eq:def_alpha}
and the second argument of $F$ is
$\sqrt{\frac{y_3-y_2}{y_3-y_1}}$ if $y_4$ is infinite. Since $F$ is
D-finite, and its arguments are algebraic in $t$, we conclude
once again by a composition argument.
  \end{proof}

\subsection{The \gf\ $\boldsymbol{Q(x,y;t)}$ of decoupled quadrant walks}

We now return to  the $9$ models with an infinite group for which we
have obtained a rational expression {for} 
$Q(0,y;t)$ in terms of the weak invariant $w(y;t)$ (Theorem~\ref{thm:9models}). We want to prove that the series
$Q(x,y;t)$ is D-algebraic (in $t, x$ and $y$)  for each of them, as
claimed in Theorem~\ref{thm:DA}.

Let us first prove that $Q(0,y;t)$ is DA. Theorem~\ref{thm:9models}
gives an expression {for} 
$S(y;t):=K(0,y;t) Q(0,y;t)$ in terms of the
weak invariant $w(y;t)$, valid for $t \in (0, 1/\vert\cS\vert)$ and $y$ in $\GL$. By
Theorem~\ref{prop:w-DA}, the weak invariant has an analytic continuation on a
complex domain, where it is DA. The closure properties of
Propositions~\ref{cor:closure} and~\ref{prop:closure-comp}  
then imply that $S(y;t)$ is also DA, first as a
meromorphic function of $y$ and $t$, then as a series in these variables. The same
then  holds for $Q(0,y;t)$.

Let us now go back to $Q(x,0;t)$, using 
\[ 
R(x):=K(x,0;t) Q(x,0;t)  =xY_0(x;t)+S(0;t)-S(Y_0(x;t);t),
\] 
where $Y_0(x;t)$ is the root of the kernel that is a power
series in $t$ (with coefficients in $\qs[x, \bx]$).
Again, we conclude that $Q(x,0;t)$ is DA using the closure properties
of Propositions~\ref{cor:closure} and~\ref{prop:closure-comp} (since
$Y_0$ is a series in $t$ with coefficients in $\qs[x, \bx]$, this
is where we take $\qs(x)((t))$ as our differential field, as discussed
at the beginning of Section~\ref{sec:genDA}).

A final application of Proposition~\ref{cor:closure}, applied to the
main functional equation~\eqref{eq:functional_equation}, leads us to conclude that $Q(x,y;t)$
is DA as a three-variate series. \qed

\subsection{Explicit differential equations in $\boldsymbol y$}
\label{sec:diff-explicit}
We now explain how to construct, for the $9$ models of Table~\ref{tab:decoupling_functions-infinite}, 
an explicit DE in $y$ satisfied by the series $Q(0,y)\equiv
Q(0,y;t)$. This DE has polynomial coefficients in $t$ and
$y$. Depending on the model, the order of this DE ranges from $3$ to
$5$, {and the (total) degree in $Q(0,y)$ and its $y$-derivatives {ranges} from
$2$ to $5$}. We do not claim that it is minimal. The $9$ DEs thus obtained
have been checked numerically by expanding $Q(0,y)$ in $t$ up to order~$30$. The corresponding {\sc Maple} session is available on the 
authors' webpages. The construction of  explicit DEs in $t$ seems more
difficult, as discussed later in Section~\ref{sec:DE-t}.

We start from the expression {for} 
$S(y)=K(0,y)Q(0,y)$ given by
Theorem~\ref{thm:9models}, which can be written as:
\beq\label{Qw}
K(0,y)Q(0,y)-G(y)= \frac{\alpha}{w(y)-\beta}+ \gamma,
\eeq
for $\alpha$, $\beta$ and $\gamma$ depending on $t$ only.
The weak invariant satisfies a first order DE, derived
  from~\eqref{diff-p} and~\eqref{DE-Z}:
\beq\label{ED-w}
4\, \widetilde d(y) \left( w'(y)\right)^2 = 4 w(y)^3 -g_2w(y) -g_3.
 \eeq
Here, $g_2\equiv g_2(\om_1,\om_3)$ and $g_3\equiv g_3(\om_1,\om_3)$ depend (only)
on $t$.

Upon solving~\eqref{Qw} for $w(y)$,~\eqref{ED-w}  gives a first order DE for
$Q(0,y)$, the coefficients of which are polynomials in $t, y, \alpha,
\beta, \gamma, g_2$ and $g_3$. By expanding this DE around $y=0$, we
obtain algebraic relations between the 5 unknown series $\alpha,
\beta, \gamma, g_2, g_3$ and  the  series
$Q_{0,i}:=\frac 1 {i!}\, \partial^
iQ/\partial y^i(0,0)$ that count walks ending at $(0,i)$, for $0 \leq i\leq m-1$ (where $m$ depends on
the model).   We then eliminate $\alpha,
\beta, \gamma, g_2, g_3$  to obtain a DE in $y$ that only involves $Q(0,y)$
and the $Q_{0,i}$, for $0 \leq i\leq m-1$. For instance, for model \#4, we
obtain a DE with coefficients in $\qs[t,y, Q_{0,0},
Q_{0,1}]$ (hence $m=2$), while for model  \#6, the first $4$ series
$Q_{0,i}$ are involved (hence $m=4$). Note that this DE
\emm is,
informative: expanding it further around $y=0$ allows one to relate the series
$Q_{0,i}$ for $i\ge m$ to those with smaller index. For instance, for
model \#4 we find:
\[ 
6t^2Q_{0,2}=-2{t}^{3}\left({ Q_{0,0}}\right)^{2}-4{t}^{2}\,{ Q_{0,1}}+3t\,{
  Q_{0,0}}+{ Q_{0,1}}-4t.
\] 
Two remarks are in order, regarding models \#5 and \#9. For model \#5,  the
decoupling function $G(y)$ is singular at $y=-1$ (rather than $y=0$ for
the other models), which leads us to write the equation in terms of the
$y$-derivatives of $Q(0,y)$ at $y=-1$ rather than $y=0$. 
For model~\#9, a simplification occurs, since
$Q_{0,0}=1+tQ_{0,1}$ (due to the choice of steps), and only two derivatives of $Q(0,y)$,
namely $Q_{0,1}$ and $Q_{0,2}$, occur in the equation.

At this stage, we can proceed as described below Definition~\ref{def:DA} to
eliminate from the equation the series $Q_{0,i}$ (or $\partial^
iQ/\partial y^i(0,-1)$ for model \#5). If $m$ of them actually occur, then the order
of the final DE (with coefficients in $\qs[y,t]$) will be $m+1$. For
model \#4 for instance, for which $m=2$, we find the following third order DE:
\begin{multline*}
  y ( {t}^{2}{y}^{3}-4{t}^{2}y-2t{y}^{2}-4{t}^{2}+y ) {
\frac {{\rm d}^{3}Q}{{\rm d}{y}^{3}}} ( 0,y  ) + ( 9{
t}^{2}{y}^{3}-24{t}^{2}y-15t{y}^{2}-18{t}^{2}+6y ) {
\frac {{\rm d}^{2}Q}{{\rm d}{y}^{2}}}  ( 0,y  ) \\
- 6\left( 2
{t}^{3}yQ  ( 0,y  ) -  ( ty+2t-1  )   ( t
y-2t-1  )  \right) {\frac {\rm d Q}{{\rm d}y}}  ( 0,y
  ) -12{t}^{3}  Q  ( 0,y  )  ^{2}-6t
 ( 5ty-3 ) Q  ( 0,y  ) =24t.
\end{multline*}
Needless to say, we have no combinatorial understanding of this identity. 
The orders and degrees of the DE obtained for the 9 decoupled models are
as follows:

\smallskip
\begin{center}
  \begin{tabular}{c|ccccccccc}
    model&    \#1&\#2&\#3&\#4&\#5&\#6&\#7&\#8&\#9\\
    \hline
    order & 4 & 4 & 4 & 3 & 4 & 5 & 4 & 4 & 3\\
    degree& 3 & 3 & 4 & 2 & 4 & 5 & 3 & 4 & 2
  \end{tabular}
\end{center}
\bigskip

\section{Decoupling functions for other starting points}
\label{sec:start}

We have proved in the previous sections that when  the function $xy$
is decoupled (in the sense of Definition~\ref{def:dec-gen}),  the
nature of the
series $Q(x,y;t)$ that counts quadrant walks starting at~$(0,0)$ tends
to be  simpler: algebraic when the group $\cG(\cS)$ is finite,
D-algebraic otherwise.  In this section, we explore the existence of decoupling
functions for other starting points. We expect similar implications in
terms of the nature of the associated \gf\ (but we have not worked
this out). 
Remarkably, we find that some infinite group models that are \emm  not,
decoupled for walks starting at $(0,0)$ are still decoupled for other
starting points --- and we thus expect the associated \gf\ to be D-algebraic.

\medskip
For a given model $\cS$, and  $a,b\in \ns$, we denote by
$q_{a,b}(i,j;n)$ the number of walks in $\ns^2$ with steps in $\cS$
starting at $(a,b)$ and ending at $(i,j)$. We define the \gf\ of walks
starting at $(a,b)$ by:
\[ 
     Q_{a,b}(x,y) \equiv Q^\cS_{a,b}(x,y;t)=\sum_{i,j,n\geq 0} q_{a,b}(i,j;n) x^iy^jt^n.
\] 
This series satisfies the following generalization of~\eqref{eq:functional_equation}:
\[  
     K(x,y) Q_{a,b}(x,y) = K(x,0) Q_{a,b}(x,0) + K(0,y)
     Q_{a,b}(0,y)-K(0,0) Q_{a,b}(0,0)- x^{a+1}y^{b+1}.
\] 
This leads us  to ask  for
which models $\cS$ and which values of $a$ and $b$ the function
$H(x,y):=x^{a+1}y^{b+1}$ is decoupled.

We first give a complete answer in  the finite group
case (Proposition~\ref{thm:starting-points-finite}). Then we give what we believe to be the complete list of
decoupled cases for infinite groups
(Proposition~\ref{thm:starting-points-infinite}). We conclude in
Proposition~\ref{thm:starting-points-weighted} with the 4 weighted
models of Figure~\ref{fig:alg_models}.

\medskip
\noindent {\bf{Remarks}}\\
{\bf 1.} Clearly, if a model $\cS$ with starting point $(a,b)$ is
decoupled, then the model obtained after reflection in the first
diagonal is decoupled for $(b,a)$. Hence the ``complete answer'' and
``complete list'' mentioned above are complete up to diagonal
symmetry.\\
{\bf 2.} If for some model $\cS$ the series $Q^\cS_{a,b}(x,y)$  is
  algebraic (resp.\ D-algebraic), then for all $(c,d)\in \ns^2$, the
  coefficient of $x^cy^d$ in this series is also  (D-)algebraic. This series counts quadrant walks with steps in
  $\cS$ going from $(a,b)$ to $(c,d)$, or, upon reversing steps,
  quadrant walks from $(c,d)$ to $(a,b)$ with steps in $\overline
  \cS:=\{(-i,-j): (i,j)\in \cS\}$. This means that the coefficient of
  $x^ay^b$ in  $Q^{\overline\cS}_{c,d}(x,y)$ is  (D-)algebraic for all $c,d$.  For instance,
 for each  model $\cS$ of Table~\ref{tab:decoupling_functions-finite} (resp.~\ref{tab:decoupling_functions-infinite}), and
 each starting point $(c,d)$, the series $Q^{\overline
   \cS}_{c,d}(0,0)$ is
 algebraic (resp.~D-algebraic). But what we have in mind in this
 section is the (D-)algebraicity of the three-variate series $Q^{\overline
   \cS}_{c,d}(x,y)$.

\subsection{Models with a finite group}
\begin{Proposition}\label{thm:starting-points-finite}
Let $\cS$ be one of the $23$ unweighted models with a finite group,
listed in~\cite[Tables~1--3]{BMM-10}. Let $H(x,y):=x^{a+1}y^{b+1}$, with
$(a,b)\in \ns^2$.
\begin{enumerate}[label={{\rm(\arabic{*})}},ref={{\rm(\arabic{*})}}]
\item\label{it:1}If $\cS$ is none
of the models of
  Figure~\ref{fig:alg_models} (the Kreweras trilogy and
  Gessel's model), then  $H(x,y)$ is
  not decoupled.
\item\label{it:2}If  $\cS$ belongs to the Kreweras trilogy, then
  $H(x,y)$ is decoupled if and only if   $a=b$.
\item\label{it:3}If $\cS$ is Gessel's model, then
  $H(x,y)$ is decoupled if and only if  either $a=b$ or $a=2b+1$.
\end{enumerate}
\end{Proposition}

\begin{proof}
Recall from Theorem~\ref{Thm:decoupling_orbit-sum} that  $H(x,y)$ is
decoupled if and only if $H_\al(x,y)=0$, where $\al=\sum_{\gamma\in\cG(\cS)}\sign(\gamma)\gamma$. We refer to~\cite[Tables~1--3]{BMM-10}
for the explicit description of the group $\cG(\cS)$.
We will  use the following notation: for a Laurent polynomial $P$ in a
variable $z$, we denote by $[z^>]P$ (resp.\ $[z^<]P$) the sum of
monomials of positive (resp.\ negative) exponents in $z$. We call it
the \emm positive (resp.\ negative) part, of $P$ in $z$.
 
Let us first consider  Gessel's model. The group $\cG(\cS)$ has order 8 and 
\begin{multline*}
  H_\al(x,y)={H}(x,y)-{H}\left({\bx
    \by},y\right)+{H}\left({\bx\by},x^2y\right)-{H}\left({\bx},x^2y\right)\\
+H(\bx,\by)-H(xy, \by)+H(xy,\bx^2 \by)-H(x, \bx^2\by),
\end{multline*}

It is easy to check that if $a=b$ or $a=2b+1$, then $H_\al(x,y)=0$. 
 Conversely,
\begin{compactitem}
\item if  $a< b$, then $[x^>][y^<]H_\al(x,y)=-x^{a+1}y^{a-b}\neq 0$,
\item  if $b< a<2b+1$, then $[x^>][y^<]H_\al(x,y)=x^{2b-a+1}y^{b-a}\neq 0$,
\item  and $2b+1< a$, then
  $[x^>][y^<]H_\al(x,y)=-x^{a-2b-1}y^{-b-1}\neq 0$, 
\end{compactitem}
hence $H(x,y)$ is not decoupled. This proves Claim~\ref{it:3}.

Claims~\ref{it:1} and~\ref{it:2} are proved in a similar fashion. For
  instance, for the $16$ models  having
a vertical symmetry, 
\[ 
H_\al(x,y)=H(x,y)-H\left({\bx},y\right)+H\left({\bx},\by\,
  \frac{c(x)}{a(x)}\right)-H\left(x,\by\, \frac{c(x)}{ a(x)}\right),
\] 
where as before $a(x)=[y^2]K(x,y)$ and $c(x)=[y^0]K(x,y)$.
Thus $H_\al(x,y)$ is a Laurent polynomial in $y$, with positive part
$H(x,y)-H\left({\bx},y\right)$, and finally,
\[ 
[x^>][y^>]H_\al(x,y) =x^{a+1}y^{b+1}\neq 0,
\]   
showing that $H(x,y)$ is never decoupled.

\medskip
One can also give explicit decoupling functions for the four algebraic
models: upon generalizing Lemma~\ref{lem:dec-weak} to the function
$H(x,y)=x^{a+1}y^{a+1}$, we can check that all four models  admit
$F(x)=-x^{-a-1}$ as $x$-decoupling function. 
Similarly,  for Gessel's model with starting point $(2b+1,b)$, a
$y$-decoupling function is $G(y)=-y^{-b-1}$. 
\end{proof}

\begin{table}[ht]
\begin{center}
\begin{tabular}{|c|c|c|c||c|c|c|c|}
\hline
\begin{tikzpicture}[scale=.4] 
    \draw[->] (0,0) -- (1,1);
    \draw[->] (0,0) -- (-1,0);
    \draw[->] (0,0) -- (0,-1);
        \draw[-] (0,-1) -- (0,-1) node[below] {\phantom{$\scriptstyle 1$}};
  \end{tikzpicture} &  \begin{tikzpicture}[scale=.4] 
    \draw[->] (0,0) -- (0,1);
    \draw[->] (0,0) -- (-1,-1);
    \draw[->] (0,0) -- (1,0);
    \draw[-] (0,-1) -- (0,-1) node[below] {\phantom{$\scriptstyle 1$}};
  \end{tikzpicture} &     \begin{tikzpicture}[scale=.4] 
      \draw[->] (0,0) -- (0,-1);
    \draw[->] (0,0) -- (1,1);
    \draw[->] (0,0) -- (-1,0);
    \draw[->] (0,0) -- (0,1);
    \draw[->] (0,0) -- (-1,-1);
    \draw[->] (0,0) -- (1,0);
    \draw[-] (0,-1) -- (0,-1) node[below] {\phantom{$\scriptstyle 1$}};
  \end{tikzpicture} &   \begin{tikzpicture}[scale=.4] 
      \draw[->] (0,0) -- (-1,-1);
    \draw[->] (0,0) -- (1,1);
    \draw[->] (0,0) -- (-1,0);
    \draw[->] (0,0) -- (1,1);
    \draw[->] (0,0) -- (1,0);
    \draw[-] (0,-1) -- (0,-1) node[below] {\phantom{$\scriptstyle 1$}};
  \end{tikzpicture} & \begin{tikzpicture}[scale=.3] 
    \draw[->] (0,0) -- (-1,0) node[left] {$\scriptstyle 1$};
    \draw[->] (0,0) -- (-1,-1) node[left] {$\scriptstyle 1$};
    \draw[->] (0,0) -- (0,-1) node[below] {$\scriptstyle \lambda$};
    \draw[->] (0,0) -- (1,-1) node[right] {$\scriptstyle 1$};
    \draw[->] (0,0) -- (1,0) node[right] {$\scriptstyle 2$};
    \draw[->] (0,0) -- (1,1) node[right] {$\scriptstyle 1$};
  \end{tikzpicture} & \begin{tikzpicture}[scale=.3] 
    \draw[->] (0,0) -- (-1,0) node[left] {$\scriptstyle 1$};
    \draw[->] (0,0) -- (-1,1) node[left] {$\scriptstyle 1$};
    \draw[->] (0,0) -- (0,1) node[above] {$\scriptstyle 2$};
    \draw[->] (0,0) -- (1,1) node[right] {$\scriptstyle 1$};
    \draw[->] (0,0) -- (1,0) node[right] {$\scriptstyle 2$};
    \draw[->] (0,0) -- (1,-1) node[right] {$\scriptstyle 1$};
    \draw[->] (0,0) -- (0,-1) node[below] {$\scriptstyle 1$};
   \end{tikzpicture}  &
  \begin{tikzpicture}[scale=.3] 
    \draw[->] (0,0) -- (-1,0) node[left] {$\scriptstyle 2$};
    \draw[->] (0,0) -- (-1,1) node[left] {$\scriptstyle 1$};
    \draw[->] (0,0) -- (0,1) node[above] {$\scriptstyle 1$};
    \draw[->] (0,0) -- (-1,-1) node[left] {$\scriptstyle 1$};
    \draw[->] (0,0) -- (1,0) node[right] {$\scriptstyle 1$};
    \draw[->] (0,0) -- (1,-1) node[right] {$\scriptstyle 1$};
    \draw[->] (0,0) -- (0,-1) node[below] {$\scriptstyle 2$};
  \end{tikzpicture} & \begin{tikzpicture}[scale=.3] 
    \draw[->] (0,0) -- (-1,0) node[left] {$\scriptstyle 2$};
    \draw[->] (0,0) -- (-1,1) node[left] {$\scriptstyle 1$};
    \draw[->] (0,0) -- (0,1) node[above] {$\scriptstyle 2$};
    \draw[->] (0,0) -- (1,1) node[right] {$\scriptstyle 1$};
    \draw[->] (0,0) -- (1,0) node[right] {$\scriptstyle 1$};
    \draw[->] (0,0) -- (0,-1) node[below] {$\scriptstyle 1$};
        \draw[->] (0,0) -- (-1,-1) node[left] {$\scriptstyle 1$};
  \end{tikzpicture}  \\
\hline
$(a,a)$ &$(a,a)$ &$(a,a)$ &$(a,a)$ & $(0,0)$ & $(a,a)$ &$(a,a)$ &$(0,0)$ \\
&&& $(2b+1,b) $ &&&& $(1,0)$\\
\hline
\end{tabular}
\end{center}
\medskip
  \caption{Exhaustive list of decoupled cases among models
    with a finite group (left) and among the 4 weighted models of
    Figure~\ref{fig:alg_models}.}
  \label{tab:dec-ab-finite}
\end{table}

\noindent{\bf Remarks}\\
{\bf 1.}  We recall from~\cite[Prop.~8]{BMM-10} that the $19$ models  that never decouple (case~\ref{it:1} above) can be
solved by extracting the positive part (in $x$ and $y$) in the
alternating sum $\widetilde Q_\alpha(x,y)$, where
$\widetilde Q(x,y)=xyQ(x,y)$. Indeed, this positive part turns out to be simply
$xyQ(x,y)$. This property is closely related to the above extraction
procedure, and to the non-existence of decoupling functions.
\\
{\bf 2.}  Given a step set $\cS$,  one can also ask whether a linear combination
 $\sum_{a,b} c_{a,b} Q_{a,b}(x,y)$  is (D-)algebraic. This makes sense for instance in a
probabilistic setting, where the $c_{a,b}$'s  describe  an initial
law for the starting point.
  Again, we expect
this to be equivalent to the polynomial $H(x,y):=\sum _{a,b}c_{a,b} x^{a+1}y
^{b+1}$ being decoupled.
We can extend the proof of
Proposition~\ref{thm:starting-points-finite} to study this
question.  If $\cS$ is one of the $19$ models listed
in~\ref{it:1}, then $H(x,y)$  is never decoupled.
If $\cS$ is one of the Kreweras-like models, then $H(x,y)$ is  decoupled if and only if $c_{a,b}=c_{b,a}$ for all
$(a,b)$.  For instance, we expect $Q_{0,1}+Q_{1,0}$ to be
algebraic. The condition is a bit more complex in Gessel's case.
\\
{\bf 3.} As discussed above, the existence of a decoupling function
  for a finite group model does not imply algebraicity in a completely
  automatic fashion, and further work is required to prove it.  We have done this for Kreweras' walks
    starting anywhere on the diagonal: the associated \gf, which
    involves one more variable recording the position of the starting
    point, is indeed still algebraic.

\subsection{Models with an infinite group}
We now address models with an infinite group, and exhibit decoupling
functions in a number of cases.  Remarkably, we find that three models
that are \emm not, decoupled for walks starting at~$(0,0)$ still admit
decoupling functions {for other starting points}. This contrasts with
the finite group case.

\begin{Proposition}\label{thm:starting-points-infinite}
Let $\cS$ be one of  the $12$ models with an infinite group shown in
Table~\ref{tab:dec-ab}. Then  the function $x^{a+1}y^{b+1}$ is
decoupled for the values of $(a,b)$ shown in the corresponding column.
\end{Proposition}
Based on an
  (incomplete) argument and a systematic search (for $a,b \le 10$), we believe
  these values of $(\cS,a,b)$ to be the only decoupled cases (for
  infinite groups).

\renewcommand\Tstrut{\rule{0pt}{9.0ex}}         
 \begin{table}[h!]
\begin{center}
\begin{tabular}{|c|c|c|c|c|c|c|c|c||c|c|c|}
\hline
 \begin{tikzpicture}[scale=.4] 
    \draw[->] (0,0) -- (0,1);
    \draw[->] (0,0) -- (1,0);
    \draw[->] (0,0) -- (-1,0);
    \draw[->] (0,0) -- (-1,-1);
    \draw[-] (0,-1) -- (0,-1) node[below] {$\scriptstyle \# 1$};
  \end{tikzpicture} &   \begin{tikzpicture}[scale=.4] 
    \draw[->] (0,0) -- (0,1);
    \draw[->] (0,0) -- (1,0);
    \draw[->] (0,0) -- (-1,1);
    \draw[->] (0,0) -- (-1,-1);
    \draw[-] (0,-1) -- (0,-1) node[below] {$\scriptstyle \# 2$};
  \end{tikzpicture} &     \begin{tikzpicture}[scale=.4] 
    \draw[->] (0,0) -- (0,1);
    \draw[->] (0,0) -- (1,1);
    \draw[->] (0,0) -- (0,-1);
    \draw[->] (0,0) -- (-1,0);
    \draw[-] (0,-1) -- (0,-1) node[below] {$\scriptstyle \# 3$};
  \end{tikzpicture} & \begin{tikzpicture}[scale=.4] 
    \draw[->] (0,0) -- (0,1);
    \draw[->] (0,0) -- (1,0);
    \draw[->] (0,0) -- (1,-1);
    \draw[->] (0,0) -- (-1,0);
    \draw[-] (0,-1) -- (0,-1) node[below] {$\scriptstyle \# 4$};
  \end{tikzpicture} &     \begin{tikzpicture}[scale=.4] 
    \draw[->] (0,0) -- (0,1);
    \draw[->] (0,0) -- (1,0);
    \draw[->] (0,0) -- (1,1);
    \draw[->] (0,0) -- (-1,-1);
    \draw[->] (0,0) -- (-1,0);
    \draw[-] (0,-1) -- (0,-1) node[below] {$\scriptstyle \# 5$};
  \end{tikzpicture} &       \begin{tikzpicture}[scale=.4] 
    \draw[->] (0,0) -- (0,1);
    \draw[->] (0,0) -- (0,-1);
    \draw[->] (0,0) -- (1,1);
    \draw[->] (0,0) -- (-1,-1);
    \draw[->] (0,0) -- (-1,0);
    \draw[-] (0,-1) -- (0,-1) node[below] {$\scriptstyle \# 6$};
  \end{tikzpicture} &   \begin{tikzpicture}[scale=.4] 
    \draw[->] (0,0) -- (-1,1);
    \draw[->] (0,0) -- (-1,0);
    \draw[->] (0,0) -- (1,0);
    \draw[->] (0,0) -- (-1,-1);
    \draw[->] (0,0) -- (0,1);
    \draw[-] (0,-1) -- (0,-1) node[below] {$\scriptstyle \# 7$};
  \end{tikzpicture} & \begin{tikzpicture}[scale=.4] 
    \draw[->] (0,0) -- (1,1);
    \draw[->] (0,0) -- (0,-1);
    \draw[->] (0,0) -- (1,0);
    \draw[->] (0,0) -- (-1,0);
    \draw[->] (0,0) -- (0,1);
    \draw[-] (0,-1) -- (0,-1) node[below] {$\scriptstyle \# 8$};
  \end{tikzpicture} &    \begin{tikzpicture}[scale=.4] 
    \draw[->] (0,0) -- (1,0);
    \draw[->] (0,0) -- (0,-1);
    \draw[->] (0,0) -- (0,1);
    \draw[->] (0,0) -- (-1,1);
    \draw[->] (0,0) -- (1,-1);
    \draw[-] (0,-1) -- (0,-1) node[below] {$\scriptstyle \# 9$};
  \end{tikzpicture} &
\begin{tikzpicture}[scale=.4] 
    \draw[->] (0,0) -- (1,1);
    \draw[->] (0,0) -- (-1,-1);
    \draw[->] (0,0) -- (0,1);
    \draw[->] (0,0) -- (1,-1);
    \draw[-] (0,-1) -- (0,-1) node[below] {$\scriptstyle \# 10$};
  \end{tikzpicture}
&\begin{tikzpicture}[scale=.4] 
    \draw[->] (0,0) -- (1,1);
    \draw[->] (0,0) -- (0,1);
    \draw[->] (0,0) -- (-1,0);
    \draw[->] (0,0) -- (1,-1);
    \draw[-] (0,-1) -- (0,-1) node[below] {$\scriptstyle \# 11$};
  \end{tikzpicture}
 &\begin{tikzpicture}[scale=.4]  
    \draw[->] (0,0) -- (1,1);
    \draw[->] (0,0) -- (-1,-1);
    \draw[->] (0,0) -- (0,1);
    \draw[->] (0,0) -- (0,-1);
    \draw[->] (0,0) -- (1,-1);
    \draw[-] (0,-1) -- (0,-1) node[below] {$\scriptstyle \# 12$};
  \end{tikzpicture} \Tstrut\\
\hline
$(a,a)$ & $(0,0)$ & $(a,a)$ & $(0,0)$  &$(a,a)$& $(a,a)$& $(0,0)$  &
                                                                     $(0,0)$
                    & $(0,0)$& & &  \\
& & $(0,1)$ & & $(0,1)$ &$(0,1)$ &&&& $(0,1)$ &$(0,1)$ & $(0,1)$\\
&&&&$(1,0)$ &&&& & &&\\
& & & & $(1,3)$ & & & &$(1,1)$ & &&\\
\hline
\end{tabular}
\normalsize
\end{center}
\medskip
  \caption{A (conjecturally exhaustive) list of decoupled cases among
    models  with an infinite group.}
  \label{tab:dec-ab}
\end{table}

\begin{proof}
Consider a model with kernel $K(x,y)$, and take a rational function $H(x,y)$. 
Lemma~\ref{lem:dec-weak}  can be readily extended to show that the following conditions are equivalent: 
\begin{enumerate}[label={{\rm(\alph{*})}},ref={{\rm(\alph{*})}}]
\item\label{it:a}$H(x,y)$ is decoupled, that is, there exist rational functions $F(x),G(y)$ such that $K(x,y)$ divides $H(x,y)-F(x)-G(y)$,
\item\label{it:b}there exists a rational function $F(x)$ such that
  $H(X_0,y)-F(X_0)=H(X_1,y)-F(X_1)$, where $X_0,X_1$ are the roots of
  $K(x,y)$ (when solved for $x$), 
\item\label{it:c}there exists a rational function $G(y)$ such that
  $H(x,Y_0)-G(Y_0)=H(x,Y_1)-G(Y_1)$, where $Y_0,Y_1$ are the roots of
  $K(x,y)$ (when solved for $y$). 
\end{enumerate}
We call \emph{$x$-decoupling function} (resp.\ \emph{$y$-decoupling function}) of $H(x,y)$ a rational function $F(x)$ (resp.\ $G(y)$) satisfying Condition~\ref{it:b} 
(resp.\ Condition~\ref{it:c}).

\smallskip
We begin with the starting point $(a,a)$. When $a=0$, we have proved
that the $9$ decoupled models with an infinite group are those numbered from \#1 to \#9.
Now let $a$ be arbitrary, and denote $H(x,y)=x^{a+1}y^{a+1}$. For models \#3 and \#5, we have $\frac{1}{X_0X_1}=y$. Hence, 
\[ H(X_0,y)-H(X_1,y)=\frac{X_0^{a+1}}{(X_0X_1)^{a+1}}-\frac{X_1^{a+1}}{(X_0X_1)^{a+1}}=\frac{1}{X_1^{a+1}}-\frac{1}{X_0^{a+1}}.\] 
This shows that $F(x)=-x^{-a-1}$ is an $x$-decoupling function for
$H(x,y)$. Similarly, for models \#1 and \#6, one has
$\frac{1}{Y_0Y_1}=x$, so that $G(y)=-y^{-a-1}$ is a $y$-decoupling
function for $H(x,y)$. Finally, for $a=1$ and model \#9, one easily
checks that the function 
\begin{multline*}
G(y)={y}^{4}-2{\frac {{y}^{3} ( 1+2t ) }{t}}+{\frac {{y}^{2} ( 5{t}^{2}+4t+1 ) }{{t}^{2}}}\\-2{\frac {y ( 1+2t)  ( 1+t) }{{t}^{2}}}
-2{\frac {(1+t)^2}{{t}^{2}y}}+{\frac {2{t}^{2}+2t+1}{{t}^{2}{y}^{2}}}-{\frac {2}{t{y}^{3}}}+\frac{1}{y^4}
\end{multline*}
is a $y$-decoupling function for $H(x,y)$.

\smallskip
Next we consider the starting point $(0,1)$, that is, $H(x,y)=xy^2$. For model \#3, one can take $F(x)=-x-\frac{1}{xt}+\frac{1}{x^2}$. For model \#5, one may take $F(x)={x}^{2}-{\frac {x}{t}}-{\frac {1+t}{tx}}+\frac{1}{x^2}$. For model \#6, one can take 
$ F(x)= x+ \frac{1+t}{t^2(1+x)} - \frac{(1+t)^2}{t^2(1+x)^2}$.
 For models \#10 and \#12, one can take 
$F(x)=-x-1/x$. For model \#11, one can take  $F(x)=-x-1/(tx)+1/x^2$.

\smallskip
Now we consider model \#5 and the starting point $(1,0)$, that is, $H(x,y)=x^2y$.
Then   one can take $F(x)=\frac{x}{t}-x^2+\frac{1}{x}$.

\smallskip
Finally, for $(a,b)=(1,3)$ and model \#5,  we can check that
\begin{multline*}
F(x)=-{x}^{4}+2{\frac {{x}^{3}}{t}}-{\frac {{x}^{2} ( 2{t}^{2}+2t+1 ) }{{t}^{2}}}+2{\frac {x ( 1+t ) ^{2}}{{t}^{2}}}\\
+2{\frac { ( 1+2t )  ( 1+t ) }{x{t}^{2}}}-{\frac {5{t}^{2}+4t+1}{{x}^{2}{t}^{2}}}+2{\frac {1+2t}{{x}^{3}t}}-\frac{1}{x^4}
\end{multline*}
is an $x$-decoupling function for  $H(x,y)=x^2y^4$.
\end{proof}

\subsection{Weighted models with a finite group}

\begin{Proposition}\label{thm:starting-points-weighted}
Consider the four weighted models  of
Figure~\ref{fig:alg_models}. The list of starting points $(a,b)$ for
which they decouple is given in Table~\ref{tab:dec-ab-finite}. Specializing $ \lambda$
to some complex value in the first model does not yield more decoupled
cases.
\end{Proposition}

\begin{proof}
We denote as before $H(x,y)=x^{a+1}y^{b+1}$. The first weighted model has a group of order~$6$. Writing $y=z-1$
makes its elements more compact,  and  we find:
\begin{multline*}
H_\al(x,z-1) = H(x,z-1)-H\left({\frac {1}{xz}},z-1\right)
+H\left({\frac {1}{xz}},{\frac {{x}^{2}{z}^{2}+\lambda\,xz+1}{z-1}}\right)\\
  -H\left({\frac {x \left( z-1 \right) }{{x}^{2}z+\lambda\,x+1}},{\frac {{x}^{2}{z}^{2}+\lambda\,xz+1}{z-1}}\right)
 +H\left({\frac {x\left( z-1 \right) }{{x}^{2}z+\lambda\,x+1}},{\frac {\lambda\,x+{x}^{2}+1}{{x}^{2} \left( z-1 \right) }}\right)
-H\left(x,{\frac {\lambda\,x+{x}^{2}+1}{{x}^{2} \left( z-1 \right) }}\right).
\end{multline*}
Setting $z=x$ in this expression and taking the limit $x\to \infty$, {we find} 
\[ H_\al(x,x-1)=x^{a+b+2}-x^{3b-a+2}+o(x^{a+b+2}+x^{3b-a+2}).\] 
For $H(x,y)$ to be decoupled, we need  $H_\al(x,x-1)=0$, which forces $a=b$. Under this assumption we further obtain
\[ 
H_\al(x,x-1)=-(a+1)x^{2a+1}+x^{a+1}+o(x^{2a+1}+x^{a+1}),
\] 
which now forces $a=b=0$. Conversely, if $(a,b)=(0,0)$
then $H(x,y)$ is decoupled as proved in Section~\ref{subsec:decoupling_functions}. 

\medskip 
The second weighted model has a group of order $10$, and 
\begin{multline*}H_\al(x,y)=\widetilde{H}(x,y)-\widetilde{H}\left({\frac {y}{x ( 1+y ) }},y\right)
+\widetilde{H}\left({\frac {y}{x ( 1+y ) }}, \frac {1}{xy+x+y}\right)\\
-\widetilde{H}\left({\frac {x}{y ( x+1 ) }}, \frac{1}{xy+x+y}\right)
+\widetilde{H}\left({\frac {x}{y ( x+1 ) }},x\right),
\end{multline*}
where $\widetilde{H}(x,y)=H(x,y)-H(y,x)$.
Setting $y=x^2$ and  taking the limit at $x\to \infty$, {we find} 
\[ H_\al(x,x^2)=x^{a+2b+3}-x^{2a+b+3}+o(x^{a+2b+3}-x^{2a+b+3}).\] 
Hence $H_\al(x,y)=0$ implies  $a=b$. Conversely, if $a=b$, then $\widetilde{H}(x,y)=0$, so $H_\al(x,y)=0$.
The proof for the third model is similar (except that it is easier to  expand of $H_\al(x,x^2)$ at $x=0$).

\medskip 
Lastly, the fourth model has a group of order 10, and
\begin{multline*}
H_\al(x,y)=\widehat{H}(x,y)
-\widehat{H}\left({\frac {1+y}{xy}},y\right)
+\widehat{H}\left({\frac {1+y}{xy}},{\frac {x}{xy+y+1}}\right)\\
 -\widehat{H}\left(y ( x+1 ) ,{\frac {x}{xy+y+1}}\right)
+\widehat{H}\left(y ( x+1 ) ,\frac{1}{x}\right),
\end{multline*}
where $\widehat{H}(x,y)=H(x,y)-H(\by,\bx)$. Setting $y=x^2$ and  taking the limit $x\to \infty$, {we find} 
\[ H_\al(x,x^2)=x^{a+2b+3}+x^{3a-b+2}-x^{2a+b+3}+o(x^{a+2b+3}+x^{3a-b+2}+x^{2a+b+3}).\] 
Hence $H_\al(x,y)=0$ implies either $a=b$ or $a=2b +1$. In both cases, expanding
further $H_\al(x,x^2)$ as $x\to \infty$ leads to $b=0$. Hence
either $(a,b)=0$ (and then we know that the model decouples),
or $(a,b)=(1,0)$. We conclude by  checking that indeed, $H_\al(x,y)=0$ for $H(x,y)=x^2y$.
\end{proof}

\noindent{\bf Remark.} As in the infinite group case, there exist
weighted models that do not decouple at $(0,0)$, but do decouple at
other starting points. For instance, the model obtained by reversing
all steps of the first weighted model decouples at $(1,0)$ when $\lambda=0$. This model
is of interest in the study of 3-dimensional walks confined to the
first octant~\cite[Sec.~8.2]{BoBoKaMe16}.

\section{Final comments and questions}
\label{sec:final}
We begin here with two comments on our results. In the  first
subsection, we relate weak and rational
invariants. In the second one, we discuss  the link between our new
expressions for 
$Q(0,y)$ (Theorem~\ref{thm:9models}) and the integral expressions 
formerly obtained in~\cite{Ra-12}. We then go on with a list of open
questions and perspectives (Sections~\ref{sec:DE-t} to~\ref{sec:other}).

\subsection{Weak invariants vs.\ rational invariants in the finite
  group case}\label{sec:rational_weak}

In the finite group case --- and in this case only --- we were able to
exhibit both a rational invariant $J(y)$ (Definition~\ref{def:rat-inv})
and a weak invariant $w(y)$ (Definition~\ref{def:weak-inv}).
The weak invariant $w(y)$ is more intrinsic than the rational
invariant $J(y)$: indeed, the analytic invariant lemma (Lemma~\ref{lem:inv-analytic}) together with Proposition~\ref{Prop:properties_w} implies that
many invariants, as $S(y)-L(y)$ in  Section~\ref{sec:analysis}, have a rational expression  in terms of $w(y)$.}
On the other hand, $w(y)$ depends on $t$
and $y$ in a more complex fashion than $J(y)$.
Indeed, $w(y)$ is known   to be algebraic in $y$ 
(and in fact   rational in  $19$  cases), see~\cite[Thm.~3]{Ra-12}; moreover, it is D-algebraic in $t$   by Theorem~\ref{prop:w-DA}. In fact it can even be proved to be  algebraic in $t$ (in the finite group case), but still not as simple as $J(y)$.

In this section, we show how to relate $J(y)$ and $w(y)$ using the
analytic invariant lemma.

\smallskip
Let us first consider one of the $16$ models with a horizontal symmetry,
for which a rational invariant is $J(y):=y+1/y$. One can check
that the curve $\cL$ is the unit circle (see~\cite[Thm.~5.3.3
(i)]{FIM-99} for the probabilistic case $t=1/\vert\cS\vert$) and in particular
the pole $y=0$ never lies on~$\cL$. Hence $J(y)$
is a weak invariant, in the sense of
Definition~\ref{def:weak-inv}. Applying the analytic invariant lemma
 shows that
\[ 
J(y)= \frac{w'(0)}{w(y)-w(0)}+ \frac{w''(0)}{2w'(0)}.
\] 
In particular, we have thus rederived the fact that $w(y)$ is rational in $y$.

\begin{figure}[hb!]
\begin{center}
\begin{tikzpicture}[scale=.45] 
    \draw[->] (0,0) -- (0,1);
    \draw[->] (0,0) -- (-1,0);
    \draw[->] (0,0) -- (1,-1);
  \end{tikzpicture}\hspace{2mm}
  \begin{tikzpicture}[scale=.45] 
    \draw[->] (0,0) -- (1,0);
    \draw[->] (0,0) -- (-1,0);
    \draw[->] (0,0) -- (1,-1);
    \draw[->] (0,0) -- (-1,1); 
  \end{tikzpicture}\hspace{3mm}
    \begin{tikzpicture}[scale=.45] 
      \draw[->] (0,0) -- (0,1);
    \draw[->] (0,0) -- (1,0);
    \draw[->] (0,0) -- (-1,1);
    \draw[->] (0,0) -- (1,-1);
    \draw[->] (0,0) -- (0,-1);
    \draw[->] (0,0) -- (-1,0);
  \end{tikzpicture}   
  \end{center}
 \vskip -2mm  
\caption{The three D-finite transcendental models that have no
  horizontal nor vertical symmetry.}
  \label{fig:transc-models}
\end{figure}
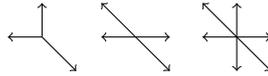

\smallskip
We now address the three models that have no horizontal symmetry and
have a transcendental \gf\ (Figure~\ref{fig:transc-models}). 
For the first one, $J(y)=ty^2-y-t/y$ has a single pole at $0$. The
curve $\cL$ is bounded and $\cG_\cL$ contains  $0$,
so that
\[ 
J(y)= -\frac{t\,w'(0)}{w(y)-w(0)}+ \gamma,
\] 
for some $\gamma$ that depends on $t$ only.

For the second model, $J(y)=\frac y{t(1+y)^2} +t \frac {(1+y)^2}y$
(as in Gessel's model)
has a simple pole at $0\in \cG_\cL$ and a double pole at $-1\not \in
\cL \cup \cG_\cL$. The invariant lemma gives 
\[ 
J(y)= \frac{t\,w'(0)}{w(y)-w(0)}+ \gamma.
\] 

Finally for the third model,  $J(y)=t/y-yt-(2t+1)/(y+1)$  has a
simple pole at $0\in \cG_\cL$ and another one at $-1\not \in
\cL \cup \cG_\cL$, and the  previous expression  $J$ in terms of $w$ holds as well.

\medskip
Let us now address the four algebraic models, for which invariants are
given in Table~\ref{tab:ratinv}. 
 For Kreweras' model, $J(y)$ has a
double pole at $0$, and the invariant lemma results in:
\[ 
J(y)=\frac t {y^2}-\frac 1 {y} -ty = \frac{t\, w'(0)^2}{(w(y)-w(0))^2 }+
\frac{\beta}{w(y)-w(0)}+ \gamma,
\] 
showing that $w(y)$ is quadratic in $y$. 

For reverse Kreweras walks, the curve
$\cL$ is not bounded, and the invariant $J(y)=ty^2-y-t/y$ is not
bounded at infinity. Hence we cannot apply directly Lemma~\ref{lem:inv-analytic}. However, 
it follows from Lemma~\ref{lem:behavior_gluing_infinity} (proved below)
that 
\beq\label{Jw-RK}
J(y)=\frac{\alpha}{w(y)-w(\infty)}-\frac{t\, w'(0)}{w(y)-w(0)}-\gamma.
\eeq
This shows that $w(y)$ is quadratic in $y$.

For the double Kreweras model, $J(y)$ has two poles, at $0$ and at
$-1$. Both belong to $\cG_\cL$, and the invariant lemma results in:
\[ 
J(y)= \frac t y -ty - \frac{1+2t}{1+y}=
\frac{\alpha}{w(y)-w(0)}+\frac{\beta}{w(y)-w(-1)}+\gamma,
\] 
showing that $w(y)$ is again quadratic in $y$.

Finally, for Gessel's model, $J(y)$ has poles at $0$ and $-1$, both
belonging to $\cG_\cL$, and respectively simple and double. The
invariant lemma gives
\[ 
J(y)= \frac y{t(1+y)^2}+t\by (1+y)^2= \frac{\alpha}{(w(y)-w(-1))^2}
+\frac \beta{w(y)-w(-1)}+ \frac{\gamma}{w(y)-w(0)}+\delta,
\] 
showing that $w(y)$ is (at most) cubic in $y$.

\medskip
We conclude this section with the lemma used above for reverse Kreweras
walks (see~\eqref{Jw-RK}).
\begin{Lemma}
\label{lem:behavior_gluing_infinity}
If the curve $\mathcal L$ is unbounded, then the weak invariant $w(y)$
is analytic at infinity, where the following expansion holds:
\[ 
 \frac{w_2}{w(y)-w(\infty)}=y^2-\frac{y}{t}+O(1)
\] 
for some $w_2 \not = 0$.
\end{Lemma}
(This lemma is essentially a version of the
  identity~\eqref{model9-w} that we wrote for  model \#9,
with the point $0$ replaced by $\infty$.)
\begin{proof}
If $\mathcal L$ is unbounded then the
branch point $x_1$ is zero, and  none of the steps $(-1,0)$ and
$(-1,1)$ belong to $\cS$  (Lemma~\ref{Lem:curve_L_0_infty}). This
forces $(-1,-1)$ and $(0,1)$ to be in $\cS$. Solving the kernel for
$y$ gives, as $x\rightarrow 0$,
\beq
\label{eq:expansion_Y(x)_0}
     Y_{0,1}(x)=\pm \frac{i}{\sqrt{x}} + \frac{1}{2t}+O(\sqrt{x}).
\eeq  
Let us return to  the form $w(y)=\wp_{1,3}(\cZ(y))$
  of~\eqref{eq:expression_gluing}. The parametrization of the curve
  $K(x,y)$ by $\wp_{1,2}$ has been designed so that $f(Y(x_1))=
  \wp_{1,2}((\om_2-\om_3)/2)$, see~\cite[Sec.\ 3.2]{KR-12}, which in our case reads $f(\infty)=
  \widetilde d''(y_4)/6 =\wp_{1,2}((\om_2-\om_3)/2)$ (it is readily
  checked that $y_4$ is finite under our hypotheses). Hence
  $\cZ(\infty)=-(\om_1+\om_3)/2$, which is a zero of $\wp_{1,3}'$, but
  not a pole of $\wp_{1,3}$. Thus $w$ is analytic at infinity. Let us denote
\[ 
w(y)=w(\infty) + \frac {w_1} y +\frac{w_2} { y^2}+\frac{w_3} { y^3} + O\left(\frac 1
  {y^4}\right).
\] 
Writing that $w(Y_0)=w(Y_1)$ near $x=0$, with the $Y_k(x)$'s given
by~\eqref{eq:expansion_Y(x)_0}, {we obtain}
$w_1=0$ and $w_3=w_2/t$. As in the proof of~\eqref{model9-w}, the fact that
$w_2\not = 0$ comes from the fact that  $\cZ(\infty)=-(\om_1+\om_3)/2$
is a zero of $\wp_{1,3}'$ but not of its derivative. The lemma then follows.
\end{proof}

\subsection{A connection with  integral representations of $\boldsymbol{Q(x,y)}$}
\label{sec:integrals}

Prior to this paper, for a non-singular model with an infinite group, the series
$Q(x,y)$ was expressed as a contour integral involving the gluing
function $w(y)$ (a.k.a.\ weak invariant)~\cite{Ra-12}. 
If the model has a decoupling function, we have now obtained
a simpler, integral-free expression in Theorem~\ref{thm:9models}. We
explain here, without giving all details, how to derive it from the
integral one, in the analytic setting of Section~\ref{sec:analysis}.
To avoid technicalities we only consider models 
such that $0\notin[x_1,x_2]$, thereby excluding models \#2, \#7 and \#9.

Let $\wx(x)$ be the counterpart of the weak invariant $w(y)$, but for
the variable $x$. In particular, $\wx(x)$ is a gluing function for the
 domain $\mathcal G_\mathcal M$ already introduced in the proof of
 Proposition~\ref{prop:S-properties}.
Then, for $x\in\mathcal G_\mathcal M\setminus [x_1,x_2]$,
it is known that
\begin{equation}
\label{eq:Ra-12_integral_contour_bis}
     R(x)-R(0)= xY_0(x)+\frac{1}{2\pi i}\int_{x_1}^{x_2} u(Y_0(u-0i)-Y_1(u-0i))\left\{\frac{\wx'(u)}{\wx(u)-\wx(x)}-\frac{\wx'(u)}{\wx(u)-\wx(0)}\right\}\text{d}u,
\end{equation}
where $Y_k(u\pm 0i)$ stands for $\lim Y_k(x)$ when $x\to u$ with
$\Im(\pm x)>0$. This  is
Theorem~1 in~\cite{Ra-12}, stated here with greater precision (indeed,
the first
term in the integrand is written as $Y_0(u)-Y_1(u)$ in~\cite{Ra-12}).
Recall from Subsection~\ref{sec:prel} that
  the functions $Y_0$ and $Y_1$ are not meromorphic on $[x_1,x_2]$,
  but  admit limits from above and below. These limits satisfy  
\begin{equation}
\label{eq:limits_prop}
     Y_0(u\pm 0i)=\overline{Y_1(u\pm 0i)},\quad
     \Im(Y_0(u-0i))>0,\quad 
     \Im(Y_0(u+0i))<0. 
\end{equation}     
More details can be found in the proof of~\cite[Thm.~1]{Ra-12}, or in
\cite[Sec.~4]{KuRa-11} for Gessel's model. Note that the
  assumption $0\not \in [x_1, x_2]$ guarantees that the term $\wx
  (u)-\wx(0)$ does not vanish. When $x\in [x_1,
  x_2]$, and more generally when $x$ is in the unit disk, $R(x)$ is
  analytic, even though the two terms
  of~\eqref{eq:Ra-12_integral_contour_bis} are not analytic along
  this interval.

The first crucial point is that
we can replace $\wx$ by $\wy(Y_0)$ in~\eqref{eq:Ra-12_integral_contour_bis}, where $w$ is the gluing function~\eqref{eq:expression_gluing}
for $\cG_\cL$. This comes from a combination of
three facts: 
\begin{itemize}
     \item as demonstrated in~\cite[Thm.~6]{Ra-12}, $\wy(Y_0(x))$ is a
       conformal gluing function for $\cG_\cM$, in the sense
       that {it} 
       satisfies  Proposition~\ref{Prop:properties_w} ---
       except that we are now in the $x$ variable, and that the pole
       is located at $X(y_2)$ rather than $x_2$ (note that the
         invariance
 property $w(Y_0(x))=w(Y_1(x))$ spares us the
         trouble of taking upper or lower limits when defining
         $w(Y_0(x))$ for $x\in [x_1, x_2]$);
     \item any two conformal gluing functions $w_1$ and $w_2$ are
       related by a homography. That is, $w_1=\frac{aw_2 +b}{cw_2
         +d}$, for some coefficients $a,b,c,d\in\mathbb C$
       (depending on $t$) such that $ad-bc\neq0$, see~\cite[Rem.~6]{Ra-12};
     \item  the quantity
  \[ \frac{\wx'(u)}{\wx(u)-\wx(x)}-\frac{\wx'(u)}{\wx(u)-\wx(0)}\]  in the right-hand side of~\eqref{eq:Ra-12_integral_contour_bis} takes the 
same  value, should  $\wx$ be replaced by $\frac{a\wx +b}{c\wx +d}$.
\end{itemize}

Now assume that the model admits a (rational) decoupling function $G$.
Then 
$u(Y_0(u)-Y_1(u))=G(Y_0(u))-G(Y_1(u))$ (Lemma~\ref{lem:dec-weak}). 
The integral in~\eqref{eq:Ra-12_integral_contour_bis} thus becomes:
\[
  \frac{1}{2\pi i}\int_{x_1}^{x_2}
\left(G(Y_0(u-0i))-G(Y_1(u-0i))\right)
D(Y_0(u))\,Y_0'(u)\text{d}u,
\] 
with
\[  {D(v)=
\frac{\wy'(v)}{\wy(v)-\wy(Y_0(x))}-\frac{\wy'(v)}{\wy(v)-\wy(Y_0(0))}.}
\] 
Again, the invariant condition
  $w(Y_0(u+0i))=w(Y_0(u-0i))$ for $u\in[x_1,x_2]$ allows us to replace
    $u$ by $u\pm0i$ in the term $D(Y_0(u))Y_0'(u)$ above. 

Let us write the above integral as a difference
  $T_0-T_1$ of two terms, one (namely $T_0$) involving $G(Y_0(u-0i))$ and the
  other one $T_1$ involving $G(Y_1(u-0i))$.
Recall that  $Y(x_1)\leq0$ and $Y(x_2)>0$
(Lemma~\ref{Lem:curve_L_0_infty}), and the properties~\eqref{eq:limits_prop}. The change of variable $v=Y_0(u-0i)$ in $T_0$ gives 
\begin{equation*}
     T_0=\frac{1}{2\pi i}\int_{\mathcal L\cap \{v:\Im(v)\geq 0\}} G(v)
     {D(v)}\text{d}v,
\end{equation*}
where the contour is oriented
 clockwise. For the  integral $T_1$, replacing $G(Y_1(u-0i))$ by
 $G(\overline{Y_0(u-0i)})$ and performing the same change of variables
 gives 
 \[ 
     T_1=\frac{1}{2\pi i}\int_{\mathcal L\cap
 \{v:\Im(v)\geq 0\}} G(\overline{v}){D(v)} \text{d}v
=
\frac{1}{2\pi i}\int_{\mathcal L \cap \{v:\Im(v)\leq 0\}} G(
v)D(v) \text{d}v,
\] 
where the contour in the second expression is now oriented counterclockwise (we have used the
invariant property $w(v)=w(\overline{v})$ on $\cL$). Finally, for $x \in
\cG_\cM\setminus[x_1, x_2],$ we have
rewritten~\eqref{eq:Ra-12_integral_contour_bis} as:
\begin{equation}
\label{eq:bef_res}
     R(x)-R(0) =xY_0(x)   -\frac{1}{2\pi i}\int_{\mathcal L} G(v)
\left\{\frac{\wy'(v)}{\wy(v)-\wy(Y_0(x))}-\frac{\wy'(v)}{w(v)-\wy(Y_0(0))}\right\}
\text{d}v,
\end{equation}
with $\mathcal L$ oriented counterclockwise.
The integrand is meromorphic in $\cG_\cL$, and we are going to compute the above integral with the residue theorem.

Recall that we  only discuss  here  models 1, 3, 4, 5, 6 and 8 of Theorem~\ref{thm:9models}, where $G$
has a unique pole $p$, which is simple,  equals $0$ or
$-1$, and belongs to $\cG_\cL$. The residue of $G$ at $p$ is still
denoted by $r$. The poles of the above integrand lying in $\cG_\cL$
are thus  $p$, $Y_0(x)$ and
$Y_0(0)$ (indeed, it is readily checked that the unique pole of $w$,
  located at $y_2$, does not give any pole in the integrand). Note that $Y_0(0)$ is the value denoted $\alpha$ in
  Theorem~\ref{thm:9models}, and belongs to $\{-1,0\}$. As in
  Theorem~\ref{thm:9models}, there are  two  cases: if $p\neq \alpha$
  (models 1 and 6), there are  three distinct poles,  all of
    which are simple, and 
\begin{equation*}
     R(x)-R(0)= xY_0(x)-r
     \left\{\frac{\wy'(p)}{\wy(p)-\wy(Y_0(x))}-\frac{\wy'(p)}{\wy(p)-\wy(\alpha)}\right\}
     -G(Y_0(x))+G(\alpha).
\end{equation*}
We recover the expression~\eqref{thm:9models} {for} 
$S(y)$ using~\eqref{e2}
 and the fact that $Y_0(X_0(y))=y$ in $\mathcal G_\mathcal L$
 (see~\cite[Lem.~3(ii)]{Ra-12}).

Now if $p=Y_0(0)= \alpha$ (models 3, 4, 5, 8)
there are only two poles, one at $Y_0(x)$ (of order $1$) and the other
at $p$ (of order $2$). The residue at $Y_0(x)$ is again
$-G(Y_0(x))$. The expansion around $p$ of the integrand
in~\eqref{eq:bef_res} is 
\begin{equation*}
     -\frac{r}{(v-p)^2}-\frac{1}{v-p}\left(g_0+r \left\{\frac{\wy'(p)}{\wy(Y_0(x))-\wy(p)}+\frac{\wy''(p)}{2\wy'(p)}\right\}\right)+O(1),
\end{equation*}
where $g_0$ still denotes the constant term in the expansion of $G$
around $p$.
The residue theorem gives
\begin{equation*}
     R(x)-R(0)= xY_0(x)-G(Y_0(x))+g_0+r \left\{\frac{\wy'(p)}{\wy(Y_0(x))-\wy(p)}+\frac{\wy''(p)}{2\wy'(p)}\right\},
\end{equation*}
and we conclude as above using~\eqref{e2}.

\subsection{Explicit differential equations in $\boldsymbol t$}
\label{sec:DE-t}
In Section~\ref{sec:diff-explicit}, we have obtained explicit differential equations in
$y$ for the series $Q(0,y)$, in the $9$ decoupled cases. What about the
length variable $t$? It seems extremely heavy to make the closure properties
used in Section~\ref{sec:DA} effective. One alternative approach would be to mimic Tutte's solution
of~\eqref{eq:Tutte}: he first found  a non-linear differential equation
valid for infinitely many values of $q$ (for which $G(1,0)$ is in fact
algebraic), and  then concluded
by a continuity argument.
In our context, this would mean introducing weights so as to obtain a family
of algebraic
models converging to a D-algebraic one.

Let us mention another analogy with Tutte's work. Theorem~1
in~\cite{KR-12} states  that for any non-singular
  infinite group model, there exists a dense set of values
  $t\in(0,1/\vert\mathcal S\vert)$ such that the generating function
  $Q(x,y;t)$ is D-finite in $x$ and $y$. This paper leads us to
  believe that  for decoupled models, this specialization of $Q(x,y;t)$ will even
  be algebraic. 
Then  $Q(x,y)$ would be algebraic over
$\rs(x,y)$ for infinitely many values of $t$, while for Tutte's problem, $G(1,0)$
is algebraic over $\cs(t)$ for infinitely many values of the parameter $q$.

\subsection{Completing the classification of quadrant walks}

For each of the 79 quadrant models, one now knows
  whether the series $Q(x,y;t)\equiv Q(x,y)$ is algebraic/D-finite/D-algebraic or not (Table~\ref{tab:nat}). 
One can ask the same question for interesting specializations
  of $Q(x,y)$, such as  $Q(0,0)$ and $Q(1,1)$.
These questions are
  solved for finite group models~\cite{BMM-10,BoKa-10,MR3588720}, but
  some remain open  in the case of an  infinite group:
\begin{itemize}
\item For the 5 singular models, it is known that $Q(1,1)$
  is not  D-finite~\cite{MiRe09,MeMi13}. What about  D-algebraicity? 
\item For the 51 non-singular  models for which $Q(x,y)$ is not
  D-finite, could the specialization
      $Q(1,1)$ still  be D-finite? (This is actually known  to be false in $16$ cases, as follows from~\cite{BoRaSa-14,Du-14,denisov-wachtel}.)
  Is it D-algebraic for more models than
  those of  Table~\ref{tab:decoupling_functions-infinite}?
\item For these 51 models, it is known that  $Q(0,0)$ is not D-finite~\cite{BoRaSa-14}. But is
  it D-algebraic for more models than
  those of  Table~\ref{tab:decoupling_functions-infinite}?
\end{itemize}

\subsection{Towards uniform proofs}

Maybe the most tantalizing open problem about the classification of quadrant walks would be to give a \emm uniform, proof of Table~\ref{tab:nat} (as for instance in the continuous setting of~\cite{penta}).
Ideally, one could dream of a uniform criterion which would apply automatically to any weighted quadrant model and determine the nature of the associated \gf. At this point, we know that the classification of the 79 models in the algebraic/D-finite/D-algebraic/D-transcendental hierarchy coincides with the classification in terms of the existence or non-existence of rational invariants and decoupling functions. However only  some of the implications are constructive. For instance, this paper derives \emm positive results, (like algebraicity and D-algebraicity) from the existence of invariants and/or decoupling functions. But the transcendence  results (in the finite group case) have \emm  not, been derived from the non-existence of decoupling functions,  but instead rely on independent arguments~\cite{MR3588720}. Until a very recent preprint~\cite{hardouin-singer2020}, the same was true of the D-transcendental results, originally established in~\cite{DHRS-17,DHRS-sing,DH-t}. Here are some open questions in this direction:

\begin{itemize}
\item 
  The present paper shows that exactly $4$ of the $23$ finite-group models have a decoupling function, and uses this   function to prove  algebraicity of the associated \gf. Can the transcendence of the remaining 19 models be deduced from the non-existence of a decoupling function? Would such a criteria hold
  for weighted models?

  \item
  Is there a way to prove
D-finiteness for the other 19 other finite-group models, using only the rational invariant? Is there maybe something like a \emm weak decoupling function,?

\item Can one provide a proof of non-D-finiteness of the $9$ D-algebraic models with an infinite group based on the fact that no rational invariant exists for them?
\end{itemize}

\subsection{Other walk  models}
\label{sec:other}
Could there be an invariant approach for quadrant walks with
  large steps~\cite{FaRa,BBMM-18}? For walks in a higher dimensional
  cone~\cite{BaKa,BoBoKaMe16,DuHoWa-16}? For walks avoiding a
  quadrant~\cite{BM-three-quadrants,BMW2020,Ra-Tr} or more generally, confined in
  an arbitrary cone~\cite{budd,elvey-price2020}? 
An invariant approach has  already been applied successfully to walks of the Kreweras trilogy avoiding a quadrant~\cite{mbm-K-tq}.

       \bigskip

\noindent {\bf Acknowledgments.} We  thank Charlotte Hardouin and Irina Kurkova
for interesting discussions, {and Andrew Elvey Price, who indicated a
shorter proof of Proposition~\ref{prop:P-DA}}.

\bibliographystyle{alpha}

\newcommand{\etalchar}[1]{$^{#1}$}

\label{sec:biblio}

\appendix

\section{Solving algebraic models}
\label{app:alg}
In this section, we consider in turn the eight models of
Figure~\ref{fig:alg_models}
and solve them using the invariants of Tables~\ref{tab:ratinv} and~\ref{tab:ratinv-weighted} and the decoupling
functions of Table~\ref{tab:decoupling_functions-finite}. We work systematically with the variable
$y$ (as in Section~\ref{sec:gessel}), thus using the invariant $J(y)$,
the decoupling function $G(y)$ and
\[ 
L(y)=S(y)-G(y),
\] 
with $S(y)=K(0,y)Q(0,y)$.
In each case (except for the reverse Kreweras walks), we construct from $J(y)$ and $L(y)$ a series in~$t$ with
polynomial coefficients in $y$ satisfying the conditions of
Lemma~\ref{lem:inv}. This construction is very similar to what we did in
Section~\ref{sec:gessel} for Gessel's model.
Applying Lemma~\ref{lem:inv}, and replacing $S(y)$ by its expression in terms
of $Q(0,y)$, gives an equation of the form
\beq\label{quad1}
\Pol(Q(0,y), A_1, \ldots, A_k,t,y)=0,
\eeq
 where $\Pol(x_0, x_1, \ldots, x_k, t,y)$ is a polynomial with
 rational coefficients, and $A_1, \ldots, A_k$ are $k$ auxiliary series depending
on  $t$ only (in what follows, they are always derivatives of
$Q(0,y)$ with respect to $y$, evaluated at $y=0$ or $y=-1$).  In the
 case of reverse Kreweras' walks, Lemma~\ref{lem:inv} is replaced by
 the substitution-free approach of Section~\ref{sec:alternative},
 and~\eqref{quad1} follows from~\eqref{eqinv-revK}.

We have described in~\cite{mbm-jehanne} a strategy to solve equations
of the form~\eqref{quad1}, which we apply successfully in all eight cases. One key point is
to decide how many solutions $Y\equiv Y(t)$ the following equation has:
\beq
\frac{\partial \Pol }{\partial x_0}(Q(0,Y), A_1, \ldots, A_k,t,Y)=0,
\label{quad2} \eeq
and to note that each of them also satisfies
\beq
\frac{\partial \Pol }{\partial y}(Q(0,Y), A_1, \ldots, A_k,t,Y)=0.
\label{quad3} \eeq
Each of these series $Y$ is also a double root of the  discriminant of $\Pol$
with respect to its first variable, evaluated at $A_1,\ldots,  A_k,
t,y$ (and seen as a polynomial in $y$);
see~\cite[Thm.~14]{mbm-jehanne}.  Note that this method does not
require  to \emm determine, the series $Y$,
but only  to decide how many such series exist, and, possibly, compute
their first few terms.

In addition to the original paper~\cite{mbm-jehanne}, we refer the
reader to~\cite[Sec.~3.4]{mbm-gessel} where an equation of this type, arising in
Gessel's model and involving three series $A_i$, is solved.

This section is supported by a {\sc Maple} session available on the
authors' webpages, where all calculations are detailed.

 \subsection{Kreweras' model}
The invariant $J(y)$ and the decoupling function $G(y)$ have poles at
$y=0$, respectively double and simple. By eliminating these poles and
applying Lemma~\ref{lem:inv} with $\rho=1/2$, we find
\[ 
J(y)=C_2L(y)^2+C_1L(y)+C_0,
\] 
with
\[  
C_2=t, \qquad C_1=-1, \qquad  {C_0=-2tS'(0)}
\] 
(we have used the fact that $S(0)=0$, which stems from $K(0,0)=0$).
Returning to the original series $Q(0,y)$, this gives an equation of
the form~\eqref{quad1}:
\[ 
 t^2y^2Q(0,y)^2+(2t-y)Q(0,y)-2tQ(0,0)+y=0,
\] 
which  coincides with Eq.~(11) in~\cite{bousquet-versailles}.
This equation is then readily solved using the strategy of~\cite{mbm-jehanne} (as
was done in~\cite{bousquet-versailles}), and yields  Thm.~2.1
of~\cite{bousquet-versailles}. The series $Q(0,0)$ is cubic, and has a
rational expression in terms of the unique series $Z\equiv Z(t)$
having constant term~$0$ and satisfying $Z=t(2+Z^3)$. The series
$Q(0,y)$ is quadratic over $\qs(y,Z)$. By symmetry, $Q(x,0)=Q(0,x)$,
and one can get back to $Q(x,y)$ using the main functional equation~\eqref{eq:functional_equation}.

\subsection{The reverse Kreweras model}
This is the model for which we had to develop a substitution-free
version of the invariant lemma in Section~\ref{sec:alternative}. We  start
from~\eqref{eqinv-revK}, which gives an equation of the form~\eqref{quad1}:
\beq\label{eqcat-rev}
t^2yQ(y)^2+(-t^2yA_1+ty^3-y^2+t)Q(y)-tyA_2-tA_1+y^2=0,
\eeq
where $Q(y)$ stands for $Q(0,y)$, $A_1$ is $Q(0,0)\equiv Q(0)$ and $A_2$ is
$Q'_y(0,0)\equiv Q'(0)$.

 Equation~\eqref{quad2} has two roots $Y_+$ and $Y_-$, which are power
series in $\sqrt t$. Following the approach
of~\cite[Sec.~7]{mbm-jehanne}, we write that the discriminant of $\Pol$
with respect to its first variable, evaluated at $A_1, A_2, t,y$, has
two double roots in $y$ (namely $Y_+$ and $Y_-$). This gives two polynomial equations relating
$A_1$ and $A_2$, from which one  derives cubic equations for
$A_1$ and $A_2$. Both series have a rational expression in terms of
the unique series $Z\equiv Z(t)$, with constant term 0, satisfying 
$Z=t(2+Z^3)$ (this is the same parametrization as in Kreweras' model).
The series $Z$ is denoted by $W$ {in~\cite[Prop.~14]{BMM-10}}.

Once $A_1$ and $A_2$ are known, one recovers $Q(0,y)$ thanks
to~\eqref{eqcat-rev}, and this yields the expression of
$Q(0,y)$ given in Prop.~14 of~\cite{BMM-10}. This series has
degree $2$ over $\qs(y,Z)$. 

This model was first solved in~\cite[Thm.~2.3]{mishna-jcta}.

\subsection{The double Kreweras model}\label{sec:double}

The series $L(y){=S(y)+1/y}$ has just one simple pole at $y=0$, but the invariant $J(y)$
has  a second pole at $y=-1$. We first eliminate it by considering
$(L(y)-L(-1))J(y)$. 
Note that $K(0,-1)=0$, hence  $S(-1)=0$ and $L(-1)={-1}$.
Then  by eliminating poles at
$y=0$,  and applying Lemma~\ref{lem:inv} with $\rho=1$, we find 
\beq\label{eqcat-double}
(L(y)-L(-1))J(y)=C_2L(y)^2+C_1L(y)+C_0,
\eeq
with
\[  
C_2=t, \qquad C_1=-1-t-{t S(0)}, \qquad  C_0=-t(1+S(0)+S'(0)).
\] 
Returning to the original series $Q(0,y)$, which we denote by
$Q(y)$ here, this gives an equation of the form~\eqref{quad1}, of degree $2$ in
$Q(y)$, and involving two auxiliary series $A_1=Q(0)$ and
$A_2=Q'(0)$.

 Equation~\eqref{quad2} has two roots, which are power
series in $\sqrt t$. Following the approach
of~\cite[Sec.~7]{mbm-jehanne}, we write that the discriminant of $\Pol$
with respect to its first variable, evaluated at $A_1, A_2, t,y$, has
two double roots in $y$. This gives two polynomial equations relating
$A_1$ and $A_2$, from which one  derives quartic equations for
$A_1$ and $A_2$. Both series have a rational expression in terms of
the unique series $Z\equiv Z(t)$, with constant term $0$, satisfying
\[ 
Z(1-Z)^2=t(Z^4-2Z^3+6Z^2-2Z+1).
\] 
This series was introduced in~\cite{BMM-10}, where this model was solved for
the first time.

Once $A_1$ and $A_2$ are known, one recovers $Q(0,y)$ thanks
to~\eqref{eqcat-double} (using $S(y)=K(0,y)Q(0,y)$), and this yields
the expression for
$Q(0,y)$ given in Prop.~15 of~\cite{BMM-10}. 

\subsection{Gessel's model}\label{app:gessel}
We start from~\eqref{eq:JL}, with the values of $C_0, C_1, C_2$
  and $C_3$
given in Proposition~\ref{prop:abcd}. Returning to the original series $Q(0,y)$, which we denote again
$Q(y)$, this gives an equation of the form~\eqref{quad1}, of degree $3$ in
$Q(y)$, and involving three auxiliary series $A_1=Q(0)$, $A_2=Q(-1)$
and $A_3=Q'(-1)$ (note that these series are slightly different
  from those involved in~\eqref{eq-a-voir}, which were expressed in terms of
  $S$ rather than $Q$).

The equation~\eqref{quad2} has three roots, which are power
series in $t$. We can compute their first coefficients, which appear
suspiciously simple: $Y_\circ=1+O(t^6)$, $Y_+=t+2t^2+5t^4+14t^4+42t^5+O(t^6)$,
$Y_-=-t+2t^2+5t^4-14t^4+42t^5+O(t^6)$. The first series thus seems to
be constant, while the other two would involve Catalan numbers. These guesses can be proved as follows: If we eliminate $A_2$ and
$A_3$ between~\eqref{quad1},~\eqref{quad2} and~\eqref{quad3}, we find
that each series $Y$ must satisfy:
\[ 
Y ( Y-1 )  ( t{Y}^{2}+2tY+t+Y )  ( t{
Y}^{2}+2tY+t-Y )  
\left( {t}^{2} ( Y+1 ) ^{2}Q(Y)-{t}^{2}( Y+1 ) A_1-
Y \right)=0.
\] 
Using the first few coefficients of the three series $Y$, we conclude that
indeed $Y_\circ=1$, while $Y_+$ and $Y_-$ satisfy respectively
\[ 
Y_+=t(1+Y_+)^2\qquad \hbox{and} \qquad Y_-=-t(1+Y_-)^2,
\] 
or equivalently,
\[ 
1=t\left (2+Y_++ \frac 1 {Y_+}\right) \qquad \hbox{and} \qquad 1=-t\left
(2+Y_-+ \frac 1 {Y_-}\right).
\] 
Following the approach
of~\cite[Sec.~7]{mbm-jehanne}, we write that the discriminant of $\Pol$
with respect to its first variable, evaluated at $A_1, A_2, A_3, t,y$, has
three double roots in $y$, namely $Y_\circ, Y_+$ and $Y_-$. Up to a power of $y$, this discriminant can
be written as a polynomial $\Delta(s)$ of degree $6$ in $s:=y+1/y$. The above
equations satisfied by the series $Y$ show that $\Delta(s)$ vanishes at $s=2$, at
$s=1/t-2$ and $s=-1/t-2$. This gives three polynomial equations relating
$A_1$, $A_2$ and $A_3$, from which one finally derives a quartic equation for
$A_2$, and equations of degree $8$ for $A_1$ and $A_3$. As in previous
papers dealing with Gessel's model, we introduce  the quartic
series $T\equiv T(t)$ as the unique solution  with constant term $1$ of
\[ 
T=1+ 256 t^2\frac{T^3}{(3+T)^3},
\] 
and denote $Z= \sqrt T=1 +O(t)$.
Then we find that
\begin{align*}
A_1&=Q(0,0)=\frac{32\,Z ^3(3+3Z -3Z ^2+Z ^3)}{(1+Z )(Z ^2+3)^3},
\\
A_2&=Q(0,-1)= 2\,{\frac {{T}^{3}+{T}^{2}+27\,T+3}{ \left( T+3 \right)
    ^{3}}},
\\
A_3&=Q'_y(0,-1)={\frac { \left( Z-1 \right)  \left( {Z}^{8}-{Z}^{7}-8\,{Z}^{5}+19\,{Z}
^{4}+7\,{Z}^{3}+10\,{Z}^{2}+2\,Z+2 \right) }{{Z}^{3} \left( {Z}^{2}+3
 \right) ^{3}}}.
\end{align*}
Once $A_1, A_2$ and $A_3$ are known, one returns
to the equation that relates them to $Q(0,y)$ (this is essentially~\eqref{eq:JL}). Expressing $t^2$ and the series $A_i$ as rational
functions in $Z$ shows that $Q(0,y)$ is cubic over $\qs(y,Z)$, and one
recovers its expression given in~\cite[Thm.~1]{mbm-gessel} or~\cite{BoKa-10}. 

It remains to get back to $Q(x,0)$, which can be done using the
equation $R(x)+S(Y_0)=xY_0+R(0)$, where $Y_0$ is the root of the kernel
that is a power series in $t$. In fact, we prefer to handle $Q(xt,0)$
rather than $Q(x,0)$, because it is an even function of $t$. It is found that $Q(xt,0)$ has degree 3
over $\qs(x,Z)$, and one recovers its expression given in~\cite[Thm.~1]{mbm-gessel} or~\cite{BoKa-10}.

\subsection{First weighted model}
This is the model involving an arbitrary weight $\lambda$. The
invariant $J(y)$, and the decoupling function $G(y)$, have a pole at
$y=-1$, respectively double and simple. We
eliminate it and apply Lemma~\ref{lem:inv} with $\rho=1/2$ 
to obtain
\[ 
J(y)=C_2L(y)^2+C_1L(y)+C_0,
\] 
with 
\[ 
{C_2=-t^2}, \qquad C_1=t, \qquad C_0=-t^2+2t (1+\lambda t) S'(-1).
\] 
Returning to the series $Q(0,y)\equiv Q(y)$, this gives  an equation of the form
\eqref{quad1} involving a single series $A_i$, namely $A_1=Q(0,-1)$:
\beq\label{eqcat-lambda}
 \left( y+1
 \right) ^{2}{t}^{2}  Q (y)   ^{2}
+ \left( 2\,\lambda\,t-y+1 \right) Q ( y ) 
-2\left( \lambda\,t+1 \right)A_1 +y+1 {=0}.
\eeq
 Equation~\eqref{quad2} has a root $Y=1+ 2\lambda t +O(t^2)$, and
the discriminant of $\Pol$ with respect to its first variable thus has a
double root in $y$. This gives for $A_1$ a cubic equation, which can
be parametrized rationally by  the unique series $Z\equiv
Z(t)$, with coefficients in $\qs(\lambda)$ and constant term $0$, satisfying:
\[ 
Z(1+4Z)=t\left(1+6Z+12Z^2+4(2+\lambda)Z^3\right).
\]  
This series was introduced in~\cite{KaYa-15}, where this model
was first solved, using heavy computer algebra. 

By
setting $y=0$ in~\eqref{eqcat-lambda} {(with $t$ and $A_1$ expressed in
terms of $Z$)}, we find that $Q(0,0)$ {lies in $\qs(\lambda,\sqrt{1+4Z})$}, and  has degree $6$ over
$\qs(t,\lambda)$.  
More generally, $Q(0,y)$ is quadratic over $\qs(\lambda,y,Z)$, and
we recover its expression in terms of $Z$ given
in~\cite[Sec.~5.3]{KaYa-15}  (note that the model we consider
here differs  by a diagonal symmetry from the one
of~\cite{KaYa-15}). Moreover, $Q(0,y)$ has a rational expression in
terms of $\lambda$, $Z$ and $V$, where $V$ is the unique series in
$t$, with coefficients in {$\qs(\lambda,y)$} and constant term $0$,
satisfying $V=Z(Zy-1)(1+V)^2$.

It remains to get back to $Q(x,0)$, which can be done using the
equation $R(x)+S(Y_0)=xY_0+R(0)$, where $Y_0$ is the root of the kernel
that is a power series in $t$. It is found that $Q(x,0)$ has degree $4$
over $\qs(\lambda,x,Z)$, {degree 12 over $\qs(\lambda,t,x)$}, and one recovers the expression given
in~\cite[Sec.~5.3]{KaYa-15}. Moreover, $Q(x,0)$ admits a rational
expression in terms of $\lambda$,  $\sqrt{1+4Z}$ and~$U$, where
$U$ is the unique series in $t$ with coefficients in $\qs(\lambda)$
and constant term $0$ satisfying $U=Z((2+\lambda) xZ+x-1)(1+U)^2$.

An alternative solution is described in~\cite[Sec.~4]{mbm-gessel}.

\subsection{Second weighted model}\label{sec:first10}
In this example, $L(y)$ has 
simple poles at $0$ and at $-1$, while $J(y)$ has a double pole at both
points. Fortunately, eliminating the pole at $0$ also eliminates the
pole at $-1$, and Lemma~\ref{lem:inv}, applied with $\rho=1$, yields
\[ 
J(y)=C_2L(y)^2+C_1L(y)+C_0,
\] 
with
\[ 
C_2=t^2, \qquad C_1=-t-4t^2, \qquad C_0=-2t-4t^2+2t^2S'(0).
\] 
Returning to $Q(0,y)\equiv Q(y)$ gives a quadratic equation of the
form~\eqref{quad1}, involving a single auxiliary series
$A_1=Q(0,0)$:
\beq\label{eqcat-10-1}
{y}^{2}{t}^{2} ( y+1 ) ^{2}  Q (y)
  ^{2}+ ( 2t{y}^{3}-2t{y}^{2}-{y}^{2}-2t+y ) Q
 (y) +2tA_1 +{y}^{2}-y=0.
\eeq
Equation~\eqref{quad2} has \emm two, solutions (one more than
  needed to determine $A_1$!).
One of them  reads $2t+O(t^2)$, the other is $1-2t+O(t^3)$. Both are
double roots of the discriminant of $\Pol$ with respect to its first
variable. This gives for $A_1$ a cubic equation over $\qs(t)$, and $A_1$ admits a rational
expression in terms of the unique series $Z\equiv Z(t)$, with constant term
$0$, satisfying
\beq\label{Z-first10}
Z=t(2+2Z+4Z^2+Z^3).
\eeq
More precisely,
\beq\label{Q00-first10}
Q(0,0)=\frac{Z(4-4Z+Z^3)}{8t}.
\eeq
(In fact, $Z$ is one of the series $Y$
satisfying~\eqref{quad2}.) The series $Q(0,y)$ is quadratic over
$\qs(Z,y)$, as follows from~\eqref{eqcat-10-1}, and admits
a rational expression in  $Z$ and $\sqrt{1-yZ(2+Z)}$. Since the
model is $x/y$-symmetric, this completes its solution. 

As mentioned
in~\cite[Sec.~4]{mbm-gessel}, this model can also be solved using the
``half-orbit'' approach of~\cite[Sec.~6]{BMM-10}.

\subsection{Third weighted model}
This model is obtained by reversing steps of the previous one. In
particular, its $y$-invariant is obtained by replacing $y$ by $1/y$ in
the invariant of the previous model. It has poles at $0$ and $-1$,
while $L(y)$ has a simple pole at $0$ only. We first eliminate the
(double) pole of $J(y)$
 at $-1$, by considering $(L(y)-L(1))J(y)$: this indeed suffices, as
 $L(y)-L(1)$ has a double root at $-1$.  Then we eliminate the
 resulting double  pole at 0, and apply Lemma~\ref{lem:inv} with
 $\rho=1$ to obtain:
\[ 
(L(y)-L(-1))J(y)=C_2L(y)^2+C_1L(y)+C_0,
\] 
where
\[ 
C_2=t^2, \quad C_1=-t(2+5t+tS(0)),\quad
C_0= t ( 5\,t+2 ) S(0) +t ( 1+3\,t ) 
 S'(0) +{\frac { ( 1+3\,t
 )  ( 13\,{t}^{2}+7\,t+1 ) }{t}}.
\] 
Returning to $Q(0,y)$ gives a quadratic equation of the
form~\eqref{quad1}, involving two auxiliary  series $A_1=Q(0,0)$ and
$A_2=Q'_y(0,0)$:
\begin{multline}\label{eqcat-10-2}
y{t}^{3}  ( y+1  ) ^{4}{Q}(y)^{2}+ 
Q(y) \times 
\\
\left( {t}^{2}{y}^{5}+2 {t}^{2}{y}^{4}-{t}^{2}{y}^{3}-
t{y}^{4}-{t}^{2}{y}^{2}-3 t{y}^{3}-11 {t}^{2}y+t{y}^{2}-3 {t}^{2}-4
 ty+{y}^{2}-t -y{t}^{3}  ( y+1
  ) ^{2}{A_1} \right)
\\ -t  ( t{y}^{3}-t{y}^{2}-11 ty-{y}^{2}-3 
t-4 y-1  ) { A_1}+ty  ( 1+3 t  ) {A_2}+{y}^{2}
  ( t{y}^{2}+2 ty-2 t-1  ) 
=0.
\end{multline}
 Two series (in $\sqrt t$)
satisfy~\eqref{quad2}. Hence  the discriminant of
$\Pol$ with respect to its first variable admit two double roots. This
gives a pair of equations satisfied by $A_1$ and $A_2$, and finally
both series  turn out to be 
cubic over $\qs(t)$. Moreover, they have rational expressions in terms of the
series $Z$ defined by~\eqref{Z-first10}. Of course, $A_1=Q(0,0)$ is
still given by~\eqref{Q00-first10}, since reversing steps does not
change the excursion \gf. For $A_2$, we find
\[ 
A_2=\frac{Z^2(Z+2)(Z^5+28Z^4+4Z^3-56Z^2+32)}{{256} \, t(1+Z)^2}.
\] 
Returning to~\eqref{eqcat-10-2} shows that $Q(0,y)$ is quadratic
over $\qs(y,Z)$, and can be expressed rationally in terms of $y, Z$
and $\sqrt{{4}-4(y-2)Z+(4-8y+y^2)Z^2-6yZ^3-yZ^4}$.

As mentioned
in~\cite[Sec.~4]{mbm-gessel}, this model can also be solved using the
``half-orbit'' approach of~\cite[Sec.~6]{BMM-10}.

\subsection{Fourth (and last) weighted model}
This model,  which has never been solved so far, differs from the one
of Section~\ref{sec:first10} by a reflection in a vertical line. Hence
it has the same $y$-invariant, but
the $x/y$-symmetry is lost. 
The $y$-invariant has double poles at $0$ and
$-1$, while $L(y)$ only has a simple pole at $0$. We first eliminate the
pole of $J(y)$ at $-1$ by considering $(L(y)-L(-1))J(y)$ (again, this
is sufficient since $-1$ is a double root of $L(y)-L(-1)$). Then we
eliminate the pole at $0$, apply Lemma~\ref{lem:inv} with $\rho=1$,
and obtain:
\beq\label{eqcat-third10}
(L(y)-L(-1))J(y)=C_3L(y)^3+C_2L(y) ^2+C_1L(y)+C_0,
\eeq
where
\[ 
C_3=t^2, \qquad C_2=-t-2t^2S(0), \qquad C_1= -3t(1+3t)+t(1-2t)S(0)+t^2S(0)^2-2t^2S'(0),
\] 
and
\[ 
C_0=-1-7t-17t^2+2t(1+2t)S(0)+2t^2S(0)^2+t(1-2t)S'(0)-t^2S''(0).
\] 
Returning to $Q(0,y)\equiv Q(y)$, this gives a cubic equation of the
form~\eqref{quad1}, involving no less than three auxiliary series,
namely $A_1=Q(0)$, $A_2=Q'(0)$ and $A_3=Q''(0)$, the
derivatives being still taken with respect to $y$.

As in Gessel's case, we find that three series $Y$
cancel~\eqref{quad2}. One of them is a series in $t$, namely
$Y_\circ=2t+4t^2+24t^3+O(t^4)$, and the other two are series in $s:=\sqrt
t$, namely $Y_+=s+s^2+7s ^5/2+O(s^6)$ and
$Y_-=-s+s^2-7s ^5/2+O(s^6)$. Here there is no obvious guess for their
exact values. However, upon eliminating $A_2$ and
$A_3$ between~\eqref{quad1},~\eqref{quad2} and~\eqref{quad3}, we find
that each series $Y$ must satisfy:
\[ 
Y ( Y+1 )  ( 3t{Y}^{3}-3\,t{Y}^{2}+{Y}^{3
}+tY-{Y}^{2}+t )  ( t{Y}^{3}+4t{Y}^{2}+2tY+2t-Y)  ( tYQ(Y)+1 ) =0.
\] 
Using the first few coefficients of the three series $Y$, we conclude that
$P_0(Y_\circ)=0$ and $P_1(Y_+)=P_1(Y_-)=0$, with 
\[ 
P_0(y)=t{y}^{3}+4t{y}^{2}+2ty+2t-y \qquad\hbox{and}\qquad P_1(y)=3t{y}^{3}-3t{y}^{2}+{y}^{3
}+ty-{y}^{2}+t.
\] 
In particular, the three series $Y$ are cubic.
Following the approach
of~\cite[Sec.~7]{mbm-jehanne}, we conclude that the discriminant of $\Pol$
with respect to its first variable, evaluated at $A_1, A_2, A_3, t,y$, has
three double roots in $y$, namely $Y_\circ, Y_+$ and $Y_-$.

To get a clearer view of what happens, we first note that
$P_1(y)=y^3P_0(-1+1/y)$. Hence the three roots of $P_0$ are $Y_\circ$,
$-1+1/Y_+$ and $-1+1/Y_-$. Equivalently, $P_0(y)=(1+y)^3P_1(1/(1+y))$,
and the three roots of $P_1$ are $Y_+, Y_-$ and $1/(1+Y_\circ)$. 

Up to powers of $t$, $y$ and $y+1$, the discriminant of  $\Pol$
with respect to its first variable, evaluated at $A_1, A_2, A_3, t,y$,
is a polynomial in $t$, $A_1, A_2, A_3, y$, of degree $15$ in $y$, with
dominant coefficient $-4t^3(1+3t)^3$. We denote it by $\Delta(y)$.
Then the polynomial
\beq\label{pol-diff}
(y-1)^6\Delta(y)-y^{18}(y+1)^6 \Delta(-1+1/y)
\eeq
has a factor $P_1(y)^2$. Specializing this identity at $y=Y_+$ shows
that $-1+1/Y_+$ is also a root of $\Delta(y)$. Of course, the same
holds for $-1+1/Y_-$, so that finally \emm all, roots of $P_0(y)$ cancel
$\Delta(y)$. Moreover, since the above polynomial~\eqref{pol-diff} admits $P_1(y)^2$ as
a factor,  $-1+1/Y_+$ and $-1+1/Y_-$ are in fact \emm
double, roots of $\Delta(y)$. This means that $\Delta(y)$ has a factor
$P_0(y)^2$. 

Replacing $y$ by $1/(1+y)$ in~\eqref{pol-diff} shows that 
\[ 
y^6(y+1)^{18}\Delta(1/(1+y))-(y+2)^6 \Delta(y)
\] 
admits $P_0(y)^2$ as a factor. Using the same argument as before, we
conclude that $\Delta(y)$ is also divisible by $P_1(y)^2$. 

We have now proved that $\Delta(y)$ is divisible by
$P_0(y)^2P_1(y)^2$. Performing the Euclidean division of $\Delta(y)$
by $P_0(y)^2P_1(y)^2$ yields a
remainder of degree 11 in $y$, and all its coefficients (which are
polynomials in $t$ and the $A_i$'s) must vanish. By performing
eliminations between three of them (we have chosen the coefficients of
$y^{11}, y^1$ and $y^0$), we find that the three
series $A_i$ have  degree~$8$, and in fact belong to the
same extension of degree~8 of $\qs(t)$. In
particular, $A_1=Q(0,0)$ reads
\[ 
A_1= \frac{-1-6t+\sqrt{Z}}{2t^2},
\]           
where $Z=1+12t+40t^2+O(t ^3)$ satisfies a quartic equation:
\begin{multline*}
   27\,{Z}^{4}- 
18 \left( 10000\,{t}^{4}+9000\,{t}^{3}+2600\,{t}^{2}+240\,t+1 \right)
{Z}^{2}
\\
+8 \left( 10\,{t}^{2}+6\,t+1 \right) 
 \left( 102500\,{t}^{4}+73500\,{t}^{3}+14650\,{t}^{2}+510\,t-1
 \right) Z\\
= \left( 10000\,{t}^{4}+9000\,{t}^{3}+2600\,{t}^{2}+240\,t+1
 \right) ^{2}.
\end{multline*}
This equation has genus 1, so there will not be any rational
parametrization in this case. The Galois group of the above polynomial
is the symmetric group on four elements, hence there is no extension
of order $2$ between $\qs(t)$ and $\qs(t,Z).$

Once $A_1, A_2$ and $A_3$ are determined, one returns
to~\eqref{eqcat-third10}. Expressing  the series $A_i$ as rational
functions in $t$ and $A_1$ shows that $Q(0,y)$ is cubic over
$\qs(t,y,A_1)$, and by eliminating $A_1$, one finds that it has degree
$24$ over $\qs(t,y)$.

It remains to get back to $Q(x,0)$, which can be done using the
equation $R(x)+S(Y_0)=xY_0+R(0)$, where $Y_0$ is the root of the kernel
that is a power series in $t$.  It is found that $Q(x,0)$ has degree $3$
over $\qs(t,x,A_1)$, and degree $24$ over $\qs(t,x)$.

\end{document}